\documentclass[12pt, reqno, a4paper]{article}
\usepackage{amsmath,amssymb,epsfig,amsfonts,color,bm}
\usepackage{tikz}
\usepackage{comment}

\usepackage{mathbbol}

\usepackage[left=1.5cm,right=1.5cm,top=2cm,bottom=2cm]{geometry}
\usepackage[colorlinks=true,breaklinks=true,linkcolor=lightblue,citecolor=lightblue,urlcolor=lightblue]{hyperref}

\usepackage{subfigure}

\newcommand{\aer}[1]{\textcolor{black}{#1}}

\definecolor{lightblue}{rgb}{0.22,0.45,0.70}
\definecolor{dkgreen}{rgb}{0,0.6,0}
\definecolor{gray}{rgb}{0.5,0.5,0.5}
\definecolor{mauve}{rgb}{0.58,0,0.82}

\newtheorem{remark}{Remark}[section]
\newtheorem{theorem}{Theorem}[section]
\newtheorem{lemma}[theorem]{Lemma}

\newenvironment{proof}{\noindent{\it Proof.}}{\hfill$\square$}

\numberwithin{equation}{section}
\numberwithin{remark}{section}
\numberwithin{figure}{section}
\numberwithin{table}{section}

\newcommand{\bP}{{\mathbf{P}}}
\newcommand{\bV}{{\mathbf{V}}}

\newcommand{\cP}{{\boldsymbol{\mathcal{P}}}}

\newcommand{\Rd}{{\mathbb{R}^d}}

\newcommand{\bx}{{\boldsymbol{x}}}
\newcommand{\bu}{{\boldsymbol{u}}}
\newcommand{\bv}{{\boldsymbol{v}}}
\newcommand{\bw}{{\boldsymbol{w}}}

\newcommand{\bzero}{{\boldsymbol{0}}}
\newcommand{\bn}{{\boldsymbol{n}}}
\renewcommand{\bf}{{\boldsymbol{f}}}
\newcommand{\bsigma}{{\boldsymbol{\sigma}}}
\newcommand{\btheta}{{\boldsymbol{\theta}}}
\newcommand{\bzeta}{{\boldsymbol{\zeta}}}
\newcommand{\bbeta}{{\boldsymbol{\beta}}}
\newcommand{\btau}{{\boldsymbol{\tau}}}
\newcommand{\bxi}{{\boldsymbol{\xi}}}
\newcommand{\bpsi}{{\boldsymbol{\psi}}}
\newcommand{\bphi}{{\boldsymbol{\varphi}}}

\newcommand{\bepsilon}{{\boldsymbol{\epsilon}}}
\newcommand{\bnabla}{{\boldsymbol{\nabla}}}
\newcommand{\dx}{\,{\mathrm{d}\bx}}
\newcommand{\ds}{\,{\mathrm{d}s}}
\newcommand{\dt}{\,\mathrm{d}t}

\newcommand{\rrb}[1]{{\leavevmode\color{black}{#1}}}
\newcommand{\cred}[1]{{\leavevmode\color{red}{#1}}}

\allowdisplaybreaks

\begin{document}

\title{A virtual element method for Kelvin--Voigt viscoelasticity\thanks{This work has been supported by the Anusandhan National Research Foundation (ANRF), Department of Science and Technology (DST), Government of India, through the Advanced Research Grant (ARG)–MATRICS Research Grant [Grant No. ANRF/ARGM/2025/002359/MTR]; and by the Australian Research Council through the Future Fellowship grant FT220100496 and Discovery Project grant DP22010316.}}

\author{\textsc{Utkarsh Rajput}\thanks{Department of Mathematics, Indian Institute of Space Science and Technology, Thiruvananthapuram 695 547, Kerala, India. Email: utkarshrajput.19@res.iist.ac.in.},\quad \textsc{Ankit Kumar}\thanks{Department of Mathematics, Birla Institute of Technology and Science, Pilani, Pilani Campus, Vidhya Vihar, Pilani, Rajasthan, 333031, India. Email: p20210041@pilani.bits-pilani.ac.in.},\quad 
\textsc{Andres E. Rubiano}\thanks{School of Mathematics, Monash University, 9 Rainforest Walk, Melbourne, VIC 3800, Australia. Email:  andres.rubianomartinez@monash.edu.}, \\
\textsc{Ricardo Ruiz-Baier}\thanks{School of Mathematics, Monash University, 9 Rainforest Walk, Melbourne, VIC 3800, Australia; and Universidad Adventista de Chile, Casilla 7-D, Chill\'an, Chile. Email:  ricardo.ruizbaier@monash.edu}, \quad \textsc{Sarvesh Kumar}\thanks{Department of Mathematics, Indian Institute of Space Science and Technology, Thiruvananthapuram 695 547, Kerala, India. Email:  sarvesh@iist.ac.in.}.}

\maketitle

\begin{abstract}
\noindent Considering the computational advantages of virtual element methods (VEM), this work employs a conforming VEM for the numerical approximation of the Kelvin--Voigt model of viscoelasticity. For clarity and simplicity, we focus on the primal formulation. The spatial discretization is carried out using the virtual element method, while the temporal discretization is handled via the {second-order Crank--Nicolson} scheme. We establish the well-posedness of both the semi-discrete and fully discrete problems and derive {\it a priori} error estimates. Several representative numerical examples are presented to validate the theoretical results and to demonstrate the effectiveness of the proposed formulation.


\end{abstract}


\section{Introduction}
Viscoelastic materials, which exhibit characteristics of both elastic solids and viscous fluids, arise in a wide spectrum of scientific and engineering applications. Examples include polymers, biological tissues, soft composites, and geophysical materials, where stress relaxation, creep, and hysteresis play a central role in the mechanical response\cite{chaudhuri2015hydrogels}. Understanding and predicting such behavior requires constitutive models that capture both instantaneous elastic effects and time-dependent viscous dissipation.

Among the earliest and most widely used models is the Kelvin--Voigt model, introduced independently by Kelvin and Voigt in the late nineteenth century \cite{LEMAITRE200171}. We recall that in this model, the material response is represented by a spring (having elastic modulus $E$) and a dashpot (with viscosity $\eta$) arranged in parallel, leading to the constitutive law 
\[
	\bsigma(t) = E \bepsilon(t) + \eta \frac{\mathrm{d}}{\dt}\bepsilon(t),
\]
where
\( \bsigma(t) \)  and 
                    \( \bepsilon(t) \) are the stress and strain at time \( t \),
and										    
\( \frac{\mathrm{d}\bepsilon(t)}{\dt} \) is the rate of strain.
This constitutive equation captures the response of the material to applied loads, incorporating both the immediate elastic deformation (proportional to \( E \bepsilon(t) \)) and the time-dependent viscous deformation (proportional to \( \eta \frac{\mathrm{d}\bepsilon(t)}{\dt} \)).  Essential viscoelastic phenomena include  creep (progressive deformation under constant stress) and stress relaxation (decay of stress under constant strain), commonly observed in polymers, biological tissues, and other viscoelastic materials. Despite its simplicity, this model provides a building block for more complex rheological models and continues to serve as a benchmark in the analysis of time-dependent material behavior.

From the numerical analysis perspective, the study of Kelvin--Voigt-type problems leads to systems of partial differential equations coupling elastic and dissipative terms. Their discretization requires careful consideration of both spatial and temporal approximations. On the one hand, standard and mixed finite element methods (FEM), discontinuous Galerkin, as well as multipoint stress approximation schemes \cite{lee2012mixed,rognes10,becache2005mixed,wang2024multipoint,jang2024discontinuous,gatica21,meddahi23} offer a solid foundation but typically face limitations on irregular meshes, particularly when dealing with heterogeneous coefficients or evolving interfaces. On the other hand, finite difference and spectral methods may lack geometric flexibility or robustness. These challenges have motivated the use of modern discretization techniques that combine accuracy, robustness, and geometric generality.

The Virtual Element Method (VEM), introduced in \cite{vem}, has emerged as a natural generalization of FEM to polygonal and polyhedral meshes. Its key feature is the construction of consistent and stable bilinear forms using only degrees of freedom and projection operators, without requiring explicit basis functions. 
This allows VEM to support highly general meshes, including those with arbitrarily small edges and faces, while retaining stability and optimal convergence properties. VEM has proven especially effective in problems where geometric flexibility is essential, such as interface problems, fracture mechanics, and heterogeneous materials \cite{brenner}.
Over the past decade, VEM has been successfully applied to a broad range of models, including elasticity \cite{da_Veiga_2013}, Stokes and Navier--Stokes equations \cite{stream,antonietti2023virtual,transient}, poroelasticity \cite{kumar2024numerical}, and parabolic integro-differential equations \cite{yadav2024conforming}, to name a few examples. Regarding problems in solid mechanics, very recent developments (published within the last year) have advanced stabilization-free techniques, robust formulations for nearly incompressible materials, brittle fracture simulations, and interior penalty-type enhancements \cite{xu2025stabilization,chen2025brittle,qiu2025interior,khot2025robust,silva2025node,visinoni2024family}. These contributions demonstrate both the versatility and maturity of the VEM framework in addressing complex PDEs.

Within viscoelasticity, several VEM formulations have been proposed in recent years. The work \cite{artioli2017arbitrary} introduced early VEM schemes for arbitrary-order viscoelastic analysis in linear regimes. More recently,  \cite{xiao} developed VEM discretizations tailored to viscoelastic flow problems, while  \cite{sarvesh25} investigated time-dependent mixed VEM formulations for viscoelastic solids. These studies highlight the potential of VEM in this area but also point to open challenges, such as the treatment of fully discrete schemes, the rigorous derivation of {\it a priori} error estimates, and the robust handling of heterogeneous viscoelastic parameters. {In this paper, we consider the general form of the Kelvin--Voigt viscoelastic model and employ VEM for the spatial discretization in both two and three dimensions.  We combine the VEM spatial discretization with the Crank--Nicolson time-stepping schemes, to also ensure stability and robustness in long-time simulations. To the best of our knowledge, the virtual element framework has not yet been applied to this class of problems. We stress that the application of VEM to the Kelvin–Voigt model presents significant analytical and computational challenges due to the coupling between elastic and viscous effects and the need to handle complex geometries with high accuracy. By introducing suitable intermediate projection operators and carefully designing the discrete bilinear forms, we establish optimal a priori error estimates for both the semi-discrete and fully discrete formulations. These results demonstrate the robustness and flexibility of the virtual element approach in addressing viscoelastic problems governed by the Kelvin--Voigt model.}

This paper is organized as follows. In the {remainder of this} section, we introduce the continuous problem along with its variational formulation. Section \ref{sec:vem} is devoted to the virtual element discretization of the continuous problem, where appropriate virtual spaces are proposed. The well-posedness of both the semi-discrete formulation and the optimal error analysis using suitable projection operators has been carried out in Section \ref{sec:semi-discr}. The fully discrete problem using the {Crank--Nicolson} scheme for time-discretization has been discussed in Section \ref{sec:fully-discr}. Finally, numerical experiments are presented in Section \ref{sec:results} to validate the theoretical rates of convergence using different types of meshes in 2D and 3D.

In this article, we adhere to the standard notation for functional spaces and norms. Let $\mathcal{B}$ be the open and bounded domain, then we denote the norm by $\|\bullet\|_{p,\mathcal{B}}$ and the semi-norm by $|\bullet|_{p,\mathcal{B}}$ in the Sobolev space $H^p(\mathcal{B})$, where $p$ is any integer. In case $p=0$, \textit{i.e.},  $H^0(\mathcal{B}) = L^2(\mathcal{B})$, the inner-product is denoted by $(\bullet,\bullet)_{\mathcal{B}}$ and the norm is denoted by $\|\bullet\|_{\mathcal{B}}$. The subscript $\mathcal{B}$ will be omitted from all the notations if $\mathcal{B} = \Omega$.
 
\subsection{Continuous problem}
In this section, we present the model problem along with its corresponding variational formulation. We also discuss the key properties of the continuous bilinear form that underpins the framework of the variational formulation. The model problem is defined over a two-dimensional spatial domain. A precise mathematical description of the model problem in a two-dimensional setting is provided below.

Let $\Omega \subset \Rd, \ d = 2,3$, be a bounded and simply connected polygonal or polyhedral domain. Denote by {$\mathbb{M} = \mathbb{R}^{d \times d}$ the space of real $d \times d$} matrices, and let the physical parameters $\rho, \mu_1, \mu_2, \lambda_1, \lambda_2 \in \mathbb{R}^+$ be given. In the context of infinitesimal deformations, the strain tensor $\bepsilon : H^1(\Omega; \Rd) \to L^2(\Omega; \mathbb{M})$ is defined by $\bepsilon(\bu)=(\bnabla \bu+(\bnabla \bu)^{\tt t})/2$, where $\bu$ represents the displacement field. Let $A_i:\mathbb{M}\to\mathbb{M},\;\;i=1,2$ be the rank-four tensor-valued maps defined as $A_i \btau=2\mu_i\btau+\lambda_i (\tau_{1,1}+\tau_{2,2})\mathbb{I}$, where $\mu_i, \lambda_i$ are different pairs of Lam\'e coefficients of Hooke's law. Consider the following problem of finding $\bu:\Omega\times[0,t]\to\Rd$ such that
\begin{subequations}\label{model_prblm}
    \begin{align}
        \rho \bu''-\bnabla\cdot(A_1\bepsilon(\bu')+A_2\bepsilon(\bu)) &=\bf \ \ \ \text{in} \ \  \Omega\times(0,T], \\
    \bu|_{\partial\Omega\times[0,T]} &= 0,  \\
    \bu(0)=\bu_0 \ \ \text{and} \ \ \bu'(0)&=\bu_1 \ \ \text{in} \ \ \Omega,
    \end{align}
\end{subequations}
where $\bu_0,\bu_1:\Omega\to\Rd$ are given sufficiently smooth initial data and $\bf:\Omega\times[0,T]\to\Rd$ is the given load vector. The given boundary conditions are of {homogeneous} Dirichlet type, implying a fixed boundary.

In order to state the variational problem, we denote the space $\bV = H_0^1(\Omega;\Rd)$ and for $i = 1,2$, define the symmetric bilinear forms $a_i:\bV\times \bV\to \mathbb{R}$ by $a_i(\bu,\bv)=(A_i\epsilon(\bu),\epsilon(\bv))$. Under the assumptions made on the tensor-valued maps $A_i, i = 1,2$, we obtain $\forall \ \bu,\bv \in \bV$ 
\begin{align}
    a_1(\bu,\bv) \leq C|\bu|_1|\bv|_1, \quad \quad  a_2(\bu,\bv) \leq C|\bu|_1|\bv|_1.
\end{align}
Also, for all $\bv \in \bV$, we have the following coercivity bounds,
\begin{align}
    a_1(\bv,\bv) \geq C|\bv|_1^2, \quad \quad  a_2(\bv,\bv) \geq C|\bv|_1^2,
\end{align}
{where $C>0$ depends on  Korn's inequality.}

Now, the variational formulation corresponding to the problem \eqref{model_prblm} seeks $\bu\in L^\infty(0,T;\bV)$ such that for $\bf\in L^2(0,T;L^2(\Omega;\Rd))$, $\bu_0\in \bV$, $\bu_1\in L^2(\Omega;\Rd)$, 
\begin{equation}\label{vf}
    (\bu'',\bv)+ a_1( \bu', \bv)+ a_2( \bu, \bv) = (\bf,\bv) \ \ \forall \ \bv \in \bV \ \ \text{a.e.} \ \ t \in (0,T], 
\end{equation}
subject to the initial condition that $\bu(0)=\bu_0$ and $\bu'(0)=\bu_1$. For the well-posedness of the system \eqref{vf} we refer to \cite{lee2012mixed}.

\section{{Preliminaries on} virtual element discretizations}\label{sec:vem}
In order to determine a VEM discretization of the problem \eqref{vf}, we consider a sequence $\mathcal{T}_h$ of decompositions of $\Omega$ into simply connected polygonal elements $K$ with boundary $\partial K$. The diameter of the element $K$ is denoted by $h_K$ and the mesh size $h$ will be the maximum of the diameter of all the elements in $\mathcal{T}_h$, \textit{i.e.} $h = \displaystyle \max_{K \in \mathcal{T}_h} h_K$. Every element $K$ satisfy the following:
\begin{itemize}
\item There exists a $\gamma_1$ such that for all $h$, each $K$ in $\mathcal{T}_h$ is star-shaped with respect to a ball of radius greater than $\gamma_1 h_K$.
\item There exists a $\gamma_2$ such that for all $h$, for each $K$ in $\mathcal{T}_h$, the distance between any two vertices is greater than $\gamma_2 h_K$.
\end{itemize}

{We start with the construction of scalar local spaces and suitable projectors. The vector case results simply by tensor product of the scalar counterpart.} 
For any $K\in \mathcal{T}_h$, we denote $\mathbb{P}_k(K)$ as the set of all polynomials of degree $\leq k$. For a given positive integer $k$ and for each element $K$, we define the local auxiliary space of order $k$ by \cite{ahmad2013equivalent}
\begin{equation*}
    \tilde{U}^K_h:=\{v_h: v_h|_{\partial K} \in \mathbb{B}_k(\partial K) \text{ and } \Delta v_h|_K \in \mathbb{P}_{k-2}(K)\},
\end{equation*}
where $\mathbb{B}_k(\partial K):=\{v\in C^0(\partial K):v|_e\in\mathbb{P}_k(e)$ $\forall \text{ edge } e \in \partial K\}$. Any function $v_h \in \tilde{U}^K_h$ can be uniquely determined by the following set of degrees of freedom:
\begin{itemize}
\item The values of $v_h$ at internal vertices.
\item The values of $v_h$ at $k-1$ uniformly spaced points on each internal edge $e$.
\item For $k>1$, the moments $\frac{1}{|K|}\int_K m(\bx) v_h(\bx) \dx$ $\forall \ m \in M_{k-2}(K)$ in each element $K$, 
where $\it{M}_{k-2}(K):=\{ \big(\frac{\bx-\bx_{K}}{h_{k}}\big)^{s}, |s|\leq k-2 \} $ and $\bx_{K}$ denotes the barycenter of $K$.
\end{itemize}
The degrees of freedom for a octagonal element for $k = 1,2$ and $3$ are {illustrated} in Figure \ref{fig:dofs}. The degrees of freedom on edges are represented by red dots while the internal degrees of freedom are represented by blue squares.

\begin{figure}[t!]
\begin{center}
\subfigure[$k=1$]{\includegraphics[width=0.32\textwidth]{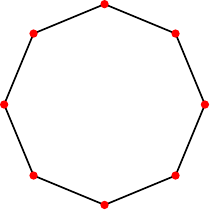}}
\subfigure[$k=2$]{\includegraphics[width=0.32\textwidth]{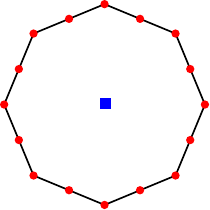}}
\subfigure[$k=3$]{\includegraphics[width=0.32\textwidth]{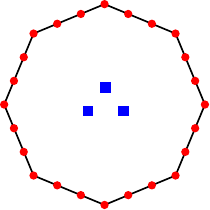}}
    \caption{Degrees of freedom for a octagonal element {using polynomial degrees} $k = 1,2$, and $3$.}
    \label{fig:dofs}
\end{center}\end{figure}

Next, and similarly to \cite{da_Veiga_2013} we define the {scalar} elliptic projection operator $\Pi^{\nabla,K}:\tilde{U}^K_h\to \mathbb{P}_k(K)$ such that it takes $v_h\in U^K_h$ to the solution of the following equations
\begin{align*} (\nabla \Pi^{\nabla,K} v_h,\nabla q)_{L^2(K)}&=(\nabla v_h,\nabla q)_{L^2(K)}\;\;\forall q \in \mathbb{P}_k(K), \\
 j_K(\Pi^{\nabla,K} v_h)&=j_K(v_h), \end{align*}
with $j_K(q)$ as the average of $q$ over the vertices of $K$. In order to compute the $L^2$-projection operator, we define the local virtual space with the help of the extended local virtual element space as following $U^K_h:=\{v_h: v_h|_{\partial K} \in \mathbb{B}_k(\partial K) \text{ and } \Delta v_h|_K \in \mathbb{P}_{k}(K)\}$ as:
$$W^K_h:=\{w\in U^K_h : (w-\Pi^{\nabla,K} w,q)=0\;\;\forall \ q\in \mathbb{P}_k(K)/\mathbb{P}_{k-2}(K)\}.$$

This modified virtual element space, introduced in \cite{ahmad2013equivalent} allows the calculation of the projection of the functions in the virtual element space to the space of piecewise $k$-degree polynomials. So, we now define the $L^2$-projection operator by
\begin{equation*}
    (\Pi^{0,K}v_h,q)_K = (v_h,q)_K, \ \forall \ q \in \mathbb{P}_k(K).
\end{equation*}

\subsection{Three-dimensional spaces}
Let us now consider $K$ to be a polyhedron element, $\partial K$ be the set of all faces of the polyhedron and $e$ be its edges. For any $f \subset \partial K$, we define the space
\begin{equation*}
    \mathbb{B}_k(\partial f):=\{v\in C^0(\partial f):v|_e\in\mathbb{P}_k(e) \ \forall \ e \in \partial f\}.
\end{equation*}
We also denote the space
\begin{equation*}
    W^f_h:=\{w\in U^f_h : (w-\Pi^{\nabla,f} w,q)=0\;\;\forall \ q\in \mathbb{P}_k(f)/\mathbb{P}_{k-2}(f)\},
\end{equation*}
where $U^f_h = \{v_h : v_h|_{\partial f} \in \mathbb{B}_k(\partial f) \text{ and } \Delta v_h|_f \in \mathbb{P}_k(f)\}$. The three-dimensional preliminary space is set as  
\begin{equation*}
    \tilde{\mathcal{U}}^K_h:=\{v_h \in H^1(K) \cap C^0(\partial K) : v_h|_{f} \in W^f_h \ \forall \ f \subset \partial K \text{ and } \Delta v_h|_K \in \mathbb{P}_{k-2}(K)\}.
\end{equation*}
We choose the following set of degrees of freedom in $\tilde{\mathcal{U}}^K_h$:
\begin{itemize}
    \item The values of $v_h$ at the vertices.
    \item The values of $v_h$ at $k-1$ uniformly spaced points on each edge $e$.
    \item For $k>1$, the moments $\frac{1}{|f|}\int_f m(\bx) v_h(\bx) \dx$ $\forall \ m \in M_{k-2}(f)$ on each face $f \subset \partial K$.
    \item For $k>1$, the moments $\frac{1}{|K|}\int_K m(\bx) v_h(\bx) \dx$ $\forall \ m \in M_{k-2}(K)$ in each element $K$.
\end{itemize}
Figure \ref{fig:dofs_3d} demonstrate the degrees of freedom for a three-dimensional cubic element. As in two-dimensional case, the solid red circles are corresponding to degrees of freedom on edges and blue squares represent the internal degrees of freedom. While green triangles are used to show the degrees of freedom on each face of the element.

Again, in $\tilde{\mathcal{U}}^K_h$ using the chosen set of degrees of freedom, we can define the elliptic projection operator but the $L^2$-projection operator can not be computed. Hence, we define the local virtual space as follows:
\begin{align*}
    W_h^K := &\{w \in H^1(K) \cap C^0(\partial K) : v_h|_{f} \in W^f_h \ \forall \ f \subset \partial K,  \\
    &\Delta v_h|_K \in \mathbb{P}_{k}(K) \text{ and } (w-\Pi^{\nabla,K} w,q)=0\;\;\forall \ q\in \mathbb{P}_k(K)/\mathbb{P}_{k-2}(K)\}.
\end{align*}

\begin{figure}[t!]
\begin{center}
\subfigure[$k=1$]{\includegraphics[width=0.32\textwidth]{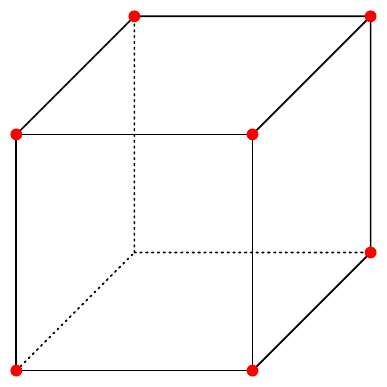}}
\subfigure[$k=2$]{\includegraphics[width=0.32\textwidth]{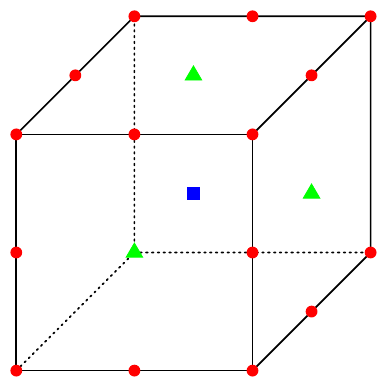}}
\subfigure[$k=3$]{\includegraphics[width=0.32\textwidth]{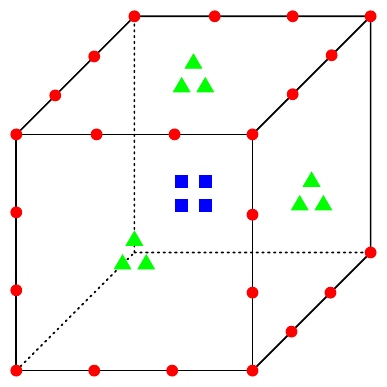}}
    \caption{Degrees of freedom for three-dimensional cubic element using polynomial degrees $k = 1,2$, and $3$.}
    \label{fig:dofs_3d}
\end{center}\end{figure}

The corresponding global discrete space is given by
$$W_h=\{v_h\in H_0^1(\Omega) : v_h|_K \in W^K_h\;\;\forall K \in \mathcal{T}_h\}.$$

Finally, we are ready to define the local virtual space for   vector-valued functions  as 
\begin{align*}
    \bV^K_h=\left\{ \begin{pmatrix} v_1 \\ \vdots \\ v_d\end{pmatrix} : v_1,\ldots,v_d\in W^K_h \right\},
\end{align*}
and similarly for the global virtual space $\bV_h$.  
Further, for $\bv_h=\left( v_1, \ldots, v_d\right)^{\tt t} \in \bV^K_h$, we use the same notation for  the elliptic projection operator $\Pi^{\bnabla,K}$ applied to the vector field $ \bv_h$, and it is defined as $\begin{pmatrix}  \Pi^{\nabla,K} v_1 \\ \vdots \\ \Pi^{\nabla,K} v_d \end{pmatrix}$. Similarly, we have $\Pi^{0,K} \bv_h =\begin{pmatrix}  \Pi^{0,K} v_1 \\ \vdots \\ \Pi^{0,K} v_d\end{pmatrix}$.
{Finally, the broken $H^1$-norm on $\bV_h$ is 
\begin{equation*}
    |\bv_h|_{1,h} = \bigl( \sum_{K \in \mathcal{T}_h}|\bv_h|_{1,K}^2 \bigr)^{\frac{1}{2}}.
\end{equation*}}

\section{Semi-discretization}\label{sec:semi-discr}
\rrb{In this section, we formulate} the semi-discrete problem corresponding to \eqref{vf} and derive the optimal error estimates. We define the local discrete bilinear forms that approximate the continuous bilinear forms on each mesh element $K$. These discrete forms are constructed to be computable using the projection operators introduced in the previous section. 

On each mesh element $K$, let us define the following discrete bilinear forms:
\begin{align*}
    m_h^K(\bv_h,\bw_h)&\rrb{:=}(\Pi^{0,K} \bv_h,\Pi^{0,K} \bw_h)+S^K((\mathbb{I}-\Pi^{0,K}) \bv_h,(\mathbb{I}-\Pi^{0,K}) \bw_h),\\
a_{i,h}^K(\bv_h,\bw_h)&\rrb{:=}a_i^K(\Pi^{\bnabla,K} \bv_h,\Pi^{\bnabla,K} \bw_h)+\frac{1}{|K|}S^K((\mathbb{I}-\Pi^{\bnabla,K}) \bv_h,(\mathbb{I}-\Pi^{\bnabla,K}) \bw_h),
\end{align*}
for $i = 1,2$. Here, $S^K(\rrb{\cdot,\cdot})$ is the stabilization term on $\bV^K_h$ which is symmetric and positive definite\rrb{,} with the property that there exists a constant $C_s\rrb{>0}$, independent of $h$, such that \rrb{for all $\bv_h\in \bV^K_h$ satisfying $\Pi^{\bnabla,K} \bv_h=\bzero$, there holds}
$$C_s^{-1}\rrb{\|\bv_h\|^2}_{0,K}\le S^K(\bv_h,\bv_h)\le C_s \rrb{\|\bv_h\|^2}_{0,K}.$$ 
A valid choice for the stabilization bilinear form is as follows (see \cite{vem}):
$$S^K(\bu_h,\bv_h)=\sum_i \chi_i(\bu_h)\cdot\chi_i(\bv_h),$$
where $\chi_i$ {is the} $i$-th local degree of freedom {associated with} the element $K$. \aer{While this operator is standard, we utilize a modified, more robust stabilization in Section~\ref{sec:results}}. \rrb{By a} standard argument, we obtain the polynomial consistency of the bilinear form $m_h^K(\rrb{\cdot,\cdot})$. That is, for all $\bv_h \in \bV_h^K$\rrb{,} 
$$m_{h}^K(\bw,\bv_h) = (\bw,\bv_h)_K, \ \ \forall \ \bw \in \rrb{\mathbb{P}_k(K;\Rd)}.$$

\rrb{The} global discrete bilinear forms $m_h:\bV_h \times \bV_h \to \mathbb{R}$ and $a_{i,h}:\bV_h \times \bV_h \to \mathbb{R}\rrb{, for } i = 1,2\rrb{,}$ are \rrb{then} defined by
\begin{align*}
    m_h(\bu_h,\bv_h):=\sum\limits_{K\in\mathcal{T}_h}m_h^K(\bu_h,\bv_h), \quad a_{i,h}(\bu_h,\bv_h):=\sum\limits_{K\in\mathcal{T}_h}a_{i,h}^K(\bu_h,\bv_h).
\end{align*}

\rrb{With the discrete bilinear forms $a_{i,h}(\bullet,\bullet)$ defined as above, the standard consistency property is not perfectly satisfied. That is, for $i=1,2$, there exist functions $\bu_h, \bv_h \in \bV_h$ where one of them is a polynomial, yet $a_{i,h}(\bu_h,\bv_h)\neq a_i(\bu_h,\bv_h)$.}
We now prove the following consistency error estimate\rrb{s}.

\begin{lemma}\label{lemma_consistency_tau}
For $\btau\in \mathbb{P}\rrb{_k}(K\rrb{;}\mathbb{M})$ such that $\btau(\bx)$ is symmetric for all $\bx\in K$ and \rrb{for all} $\bv_h\in \bV_h^K$, we have  
    $$|(\btau,\bepsilon(\bv_h))_{0,K}-(\btau,\bepsilon(\Pi^{\bnabla,K}\bv_h))_{0,K}|
        \le C_{\btau} h_K^{k} \|\bv_h\|_{1,K}.$$
\end{lemma}
\begin{proof}
\rrb{Using the} symmetry of $\btau$, \rrb{we obtain}
\begin{align*}
        |(\btau,\bepsilon(\bv_h))_{0,K}-(\rrb{\btau},\bepsilon(\Pi^{\bnabla,K}\bv_h))_{0,K}| &= |(\btau,\bnabla \bv_h)_{0,K}-(\btau,\bnabla\Pi^{\bnabla,K}\bv_h)_{0,K}| \\
        &= |(\btau,\bnabla \bv_h-\bnabla\Pi^{\bnabla,K}\bv_h)_{0,K}|\rrb{,}
    \end{align*}
and applying the Cauchy--Schwarz inequality, we readily \rrb{deduce} that
    \begin{align*}
        |(\btau,\bepsilon(\bv_h))_{0,K}-(\rrb{\btau},\bepsilon(\Pi^{\bnabla,K}\bv_h))_{0,K}| &\leq C\|\btau\|_K\|\bnabla \bv_h-\bnabla\Pi^{\bnabla,K}\bv_h\|_{0,K} \\
        &\leq C_{\btau}h_K^k\|\bv_h\|_{1,K}.
    \end{align*}
    \rrb{This completes the proof.}
\end{proof}

\begin{lemma}
    Let $\bP \in \mathbb{P}_k(K\rrb{;}\Rd)$ for all $K \in \mathcal{T}_h$. Then for all $\bv_h \in \bV_h^K$, we have
    \begin{equation*}
        \rrb{|}a_{i,h}^K(\bP,\bv_h) - a_{i}^K(\bP,\bv_h)\rrb{|} \leq C_{\bP}h_K^k\|\bv_h\|_{1,K} \quad \text{for } i = 1,2. 
    \end{equation*}
\end{lemma}

\begin{proof}
    \rrb{Since $\bP \in \mathbb{P}_k(K;\Rd)$, we have for $i = 1,2$ that}
    \begin{align*}
        a_{i,h}^K(\bP,\bv_h) = a_i^K(\bP,\Pi^{\bnabla,K} \bv_h).
    \end{align*}
    \rrb{This} allows us to write
    \begin{align*}
         \rrb{|}a_{i,h}^K(\bP,\bv_h) - a_{i}^K(\bP,\bv_h)\rrb{|} = \rrb{|}a_{i}^K(\bP,\Pi^{\bnabla,K} \bv_h - \bv_h)\rrb{|}.
    \end{align*}
    Using the definition of $a_i^K(\rrb{\cdot, \cdot})$ and the result obtained in Lemma \ref{lemma_consistency_tau}, we deduce that
    \begin{align*}
         \rrb{|}a_{i,h}^K(\bP,\bv_h) - a_{i}^K(\bP,\bv_h)\rrb{|} \rrb{\le} C_{\bP}h_K^k\|\bv_h\|_{1,K},
    \end{align*}
which concludes the proof.
\end{proof}

We now state the results on \rrb{the} continuity and coercivity of the discrete bilinear forms. \rrb{Their} proofs are straightforward and are therefore omitted (see, e.g.,  \cite{zhao2022stabilized}). 

\begin{theorem}
    For all $\bu_h, \bv_h \in \bV_h$, it holds that
    \begin{align*}
        m_h(\bu_h,\bv_h) \leq C \|\bu_h\|\|\bv_h\|, \quad a_{1,h}(\bu_h,\bv_h) \leq C|\bu_h|_{1,h}|\bv_h|_{1,h}, \quad a_{2,h}(\bu_h,\bv_h) \leq C|\bu_h|_{1,h}|\bv_h|_{1,h}.
    \end{align*}
\end{theorem}

\begin{theorem}
    For all $\bv_h \in \bV_h$, it holds that
    \begin{align*}
        m_h(\bv_h,\bv_h) \geq C\|\bv_h\|^2, \quad a_{1,h}(\bv_h,\bv_h) \geq \alpha_0|\bv_h|_{1,h}^2, \quad a_{2,h}(\bv_h,\bv_h) \geq \alpha_1|\bv_h|_{1,h}^2.
    \end{align*}
\end{theorem}

\subsection{Semi-discrete formulation}
The semi-discrete problem is to find $\rrb{\bu_h \in C^2([0,T]; \bV_h)}$ such that $\bu_h(0)=\bu_{h,0}$ and $\bu_h'(0)=\bu_{h,1}$, where $\bu_{h,0}$ and $\bu_{h,1}$ are \rrb{approximations} of the given initial data $\bu_0$ and $\bu_1$, respectively, which are computable using \rrb{the} degrees of freedom of the space $\bV_h$\rrb{, and satisfy:}
\begin{equation}\label{dv}
m_h(\bu_h'',\bv_h)+ a_{1,h}(\bu_h',\bv_h)+a_{2,h}(\bu_h,\bv_h)=(\bf_h,\bv_h)\rrb{\quad \text{for all } \bv_h\in \bV_h, \ \text{a.e. } t\in(0,T],}
\end{equation}
where $\bf_h|_K = \Pi^{0,K}\bf$ for all $K \in \mathcal{T}_h$.

Next, we aim to establish the well-posedness of the semi-discrete problem. \rrb{To this end}, we first prove the stability of the approximate solution and then we discuss the existence of \rrb{a} unique solution. We have the following result on the stability of the approximate solution:

\begin{theorem}
    Let $\bu_h \in \bV_h$ be the solution of \rrb{problem} \eqref{dv}. Then\rrb{, we have} 
    \begin{equation*}
        \sup_{\rrb{t\in [0,T]}}\left(\|\bu_h'(t)\| + |\bu_h(t)|_{1,h} \right) \leq C \left( \|\bu_{h,1}\| + |\bu_{h,0}|_{1,h} + \rrb{\int_0^T}\|\bf_h(s)\|\ds \right).
    \end{equation*}
\end{theorem}

\begin{proof}
    \rrb{Taking} $\bv_h = \bu_h'$ in \eqref{dv}\rrb{, we} obtain
    \begin{equation*}
        m_h(\bu_h'',\bu_h')+ a_{1,h}(\bu_h',\bu_h')+a_{2,h}(\bu_h,\bu_h')=(\bf_h,\bu_h'),
    \end{equation*}
which \rrb{implies} 
    \begin{equation*}
        \frac{1}{2}\frac{\mathrm{d}}{\dt}\left(m_h(\bu_h',\bu_h') +a_{2,h}(\bu_h,\bu_h)\right) + a_{1,h}(\bu_h',\bu_h')=(\bf_h,\bu_h').
    \end{equation*}
    Using the coercivity of the bilinear form $a_{1,h}(\rrb{\cdot,\cdot})$ and the Cauchy--Schwarz inequality, the above \rrb{identity yields the inequality} 
    \begin{equation}\label{energy_eqn1}
        \frac{1}{2}\frac{\mathrm{d}}{\dt}\left(m_h(\bu_h',\bu_h') +a_{2,h}(\bu_h,\bu_h)\right) + \alpha_0|\bu_h'|_{1,h}\rrb{^2} \leq C\|\bf_h\|\|\bu_h'\|.
    \end{equation}
    Integrating \rrb{both} sides from \rrb{$0$} to $t$, and using the coercivity of the \rrb{remaining} bilinear forms, yields
    \begin{equation*}
        \|\bu_h'(t)\|^2 + |\bu_h(t)|_{1,h}^2 \leq C \left( \|\bu_{h,1}\|^2 + |\bu_{h,0}|_{1,h}^2 + \int_0^t\|\bf_h(s)\|\|\bu_h'(s)\|\ds \right).
    \end{equation*}
    Furthermore, the above inequality can be rewritten as:
    \begin{align*}
        \left(\|\bu_h'(t)\| + |\bu_h(t)|_{1,h}\right)^2 &\leq C \left( \left(\|\bu_{h,1}\| + |\bu_{h,0}|_{1,h} \right)^2 + \rrb{\int_0^t}\|\bf_h(s)\|\|\bu_h'(s)\|\ds \right) \\
        & \leq C \left( \left(\|\bu_{h,1}\| + |\bu_{h,0}|_{1,h} \right)^2 + \max_{s\in [0,T]}\|\bu_h'(s)\|\rrb{\int_0^T}\|\bf_h(s)\|\ds \right) \\
        & \leq C \left( \left(\|\bu_{h,1}\| + |\bu_{h,0}|_{1,h} \right)^2 + \max_{s\in [0,T]}\left(\|\bu_h'(s)\| + |\bu_h(s)|_{1,h}\right)\rrb{\int_0^T}\|\bf_h(s)\|\ds \right).
    \end{align*}
    Finally, \rrb{taking the supremum over $t \in [0,T]$ on the left-hand side and} dividing both sides by $\displaystyle \max_{s\in [0,T]}\left(\|\bu_h'(s)\| + |\bu_h(s)|_{1,h}\right)$\rrb{,} we \rrb{arrive at the desired bound.}
\end{proof}

\rrb{By expanding $\bu_h$ and $\bu_h'$ in terms of a basis of $\bV_h$}, and considering the \rrb{symmetry and positive-definiteness} of $m_h(\rrb{\cdot,\cdot})$, the semi-discrete problem can be represented as a system of linear first-order ordinary differential equations. Thus, by Picard's theorem, \rrb{there exists} a unique solution \rrb{to} the semi-discrete problem \rrb{on} the interval $[0,T]$.

\begin{remark}From equation \eqref{energy_eqn1}, we observe that in \rrb{the} case of \rrb{a} zero load, \textit{i.e.}\rrb{,} $\bf = \bzero$, the system shows \rrb{an energy-dissipative} behavior, where the energy is given by
\begin{equation}\label{energy_formula}
    \mathcal{E}(t) = \frac{1}{2}\left( \|\bu_h'(t)\|^2 + |\bu_h(t)|_{1,h}^2 \right).
\end{equation}
\rrb{The equation above indicates} that the total energy of the system is a combination of the kinetic energy and the potential energy. \end{remark}

The solution to the semi-discrete problem satisfies the following stability result (for a proof, see\rrb{, e.g.,} \cite{pradhan2024optimalwave}).
\begin{lemma}\label{lemma_stab_est_der}
    For $k = 3,4$, \rrb{assume} that $\bu_0,\bu_1 \in H_0^1(\Omega; \Rd)\cap H^{2k-2}(\Omega;\Rd)$ and suppose also \rrb{that} $\bf \in H^{k-1}(0,T;H^{2k-4}(\Omega;\Rd))$. Then\rrb{, the following bound} holds\rrb{:} 
    \begin{equation*}
        \|\bu_h^{(k)}(t)\|^2 + \int_0^t|\bu_h^{(k)}(s)|_{1,h}^2\ds \le C \left( |\bu_0|_{2k-2}^2 + |\bu_1|_{2k-2}^2 + \|\bf\|_{H^{k-1}(0,T;H^{2k-4}(\Omega;\Rd))}^2 \right).
    \end{equation*}
\end{lemma}

\subsection{Error analysis}

In this section\rrb{,} we derive the error estimates for the solutions of the semi-discrete problem \eqref{dv}. To establish the optimal order of convergence in both the $H^1$-norm \rrb{and} the $L^2$-norm, we need to define a suitable intermediate projection operator. We introduce the operator $\cP^h : \bV \to \bV_h$ such that for all $t \in (0,T]$, 
\begin{equation}\label{interproj}
    a_{1,h}(\cP^h\bu',\bv_h) + a_{2,h}(\cP^h\bu,\bv_h) = a_{1}(\bu',\bv_h) + a_{2}(\bu,\bv_h), \ \ \text{for all} \ \bv_h \in \bV_h,
\end{equation}
with the initial condition $\cP^h\bu(0) = \bu_h(0)$. The approximation properties of the intermediate projection operator defined above are established in the subsequent lemma; the proof of the estimates given in the lemma follows from \rrb{\cite[Lemma 3.4]{pradhan2024optimal}}. 

\begin{lemma}
    Let $\bu$ be the \rrb{unique solution of problem \eqref{vf},} and assume that $\bu \in L^2(0,T;H^{k+1}(\Omega;\Rd))$.  Then, the following bounds hold true\rrb{:} 
   \begin{subequations} \label{interprojestimates}\begin{align}
        &|\cP^h \bu-\bu|_{1,h} \leq Ch^{k} \|\bu\|_{L^2(0,T;H^{k+1}(\Omega;\Rd))}, \label{interprojerrh1}\\
        &\|\cP^h \bu-\bu\| \leq Ch^{k+1} \|\bu\|_{L^2(0,T;H^{k+1}(\Omega;\Rd))}\rrb{.} \label{interprojerrl2} 
    \end{align}
    \end{subequations}
\end{lemma}

We now carry out the convergence analysis of the solution to the semi-discrete problem. We aim to obtain the optimal order of convergence under minimal regularity assumptions on the analytical solution. The following result shows the optimal rate of convergence for the displacement error in \rrb{the $H^1$-norm, as well as for} the velocity error in the $L^2$-norm.

\begin{theorem}\label{thm_semidscrt_err_est}
\rrb{Let $\bu_h$ and $\bu$ be the unique solutions of problems \eqref{dv} and \eqref{vf}, respectively, and assume that $\bu,\bu',\bu''$, and $\bf \in L^2(0,T;H^{k+1}(\Omega;\Rd))$. Then, we have the following bounds:}
\begin{align*}
    |\bu_h-\bu|_{1,h} & \le Ch^k\left( \|\bf\|_{L^2(0,T;H^{k+1}(\Omega;\Rd))} + \|\bu\|_{L^2(0,T;H^{k+1}(\Omega;\Rd))} + \|\bu''\|_{L^2(0,T;H^{k+1}(\Omega;\Rd))}  \right),\\
    \|\bu_h'-\bu'\| & \le Ch^{k+1}\left( \|\bf\|_{L^2(0,T;H^{k+1}(\Omega;\Rd))} + \|\bu'\|_{L^2(0,T;H^{k+1}(\Omega;\Rd))} + \|\bu''\|_{L^2(0,T;H^{k+1}(\Omega;\Rd))} \right).
\end{align*}
\end{theorem}

\begin{proof} 
\rrb{We decompose the error as} $\bu_h-\bu = \btheta + \bzeta$, where $\btheta = \bu_h - \cP^h\bu$ and $\bzeta = \cP^h\bu - \bu$. \rrb{Since the estimates for $\bzeta$ are established in Lemma \ref{interprojestimates}, it remains to bound $\btheta$. Starting from \eqref{dv} and using \eqref{interproj}, we obtain}
\begin{align}
    m_h(\btheta'',\bv_h)+a_{1,h}(\btheta',\bv_h)+ a_{2,h}(\btheta,\bv_h) &= (\bf_h,\bv_h)-m_h(\cP^h \rrb{\bu}'',\bv_h)-a_{1,h}(\cP^h \bu',\bv_h)   -a_{2,h}(\cP^h \bu,\bv_h) \nonumber \\
    &= (\bf_h-\bf,\bv_h) + (\bu'',\bv_h) - m_h(\cP^h \bu'',\bv_h) \nonumber \\
    &= T_1 + T_2. \label{semidscrt_err_eqn}
\end{align}
\rrb{Bounding $T_1$ using the definition of $\bf_h$, we find}
\begin{align*}
    T_1 &= (\bf_h-\bf,\bv_h) = \sum_{K \in \mathcal{T}_h}(\Pi^{0,K}\bf-\bf,\bv_h) \\
    &\leq C\sum_{K \in \mathcal{T}_h}\|\Pi^{0,K}\bf-\bf\|_K\|\bv_h\|_K 
    \leq C h^{k+1}|\bf|_{k+1}|\bv_h|_{1,h}.
\end{align*}
\rrb{Differentiating \eqref{interproj} twice with respect to time and following the proof of Lemma \ref{interprojestimates}, we deduce that}
    \begin{equation*}
        \|\cP^h \bu'' -\bu''\| \leq Ch^{k+1}\|\bu''\|_{L^2(0,T;H^{k+1}(\Omega; \Rd))}.
    \end{equation*}
\rrb{The term $T_2$ can be bounded using the polynomial consistency of the bilinear form $m_h(\bullet,\bullet)$ as follows:}
\begin{align*}
    T_2 &= (\bu'',\bv_h) - m_h(\cP^h \bu'',\bv_h) \\
    &= \sum_{K \in \mathcal{T}_h} \left( (\bu''-\Pi^{0,K}\bu'',\bv_h)_K + m_h^K(\Pi^{0,K} \bu''-\cP^h \bu'',\bv_h) \right) \\
    &\leq C\sum_{K \in \mathcal{T}_h} \left( \|\bu''-\Pi^{0,K}\bu''\|_K + \|\Pi^{0,K} \bu''-\cP^h \bu''\|_K \right)\|\bv_h\|_K \\
    &\leq C\sum_{K \in \mathcal{T}_h} \left( \|\bu''-\Pi^{0,K}\bu''\|_K + \|\bu''-\cP^h \bu''\|_K \right)\|\bv_h\|_K, 
\end{align*}
and this yields
\begin{equation*}
    T_2 \leq Ch^{k+1}\|\bu''\|_{L^2(0,T;H^{k+1}(\Omega;\Rd))}|\bv_h|_{1,h},
\end{equation*}
\rrb{where the final inequalities for $T_1$ and $T_2$ utilize the Poincar\'e--Friedrichs inequality (see \cite[Lemma 2.1]{gss}). Combining these estimates with \eqref{semidscrt_err_eqn} yields}
$$m_h(\btheta'',\bv_h)+a_{1,h}(\btheta',\bv_h)+ a_{2,h}(\btheta,\bv_h)\le B_{\bf,h,\bu}|\bv_h|_{1,h},$$
where $\rrb{B_{\bf,h,\bu}}:=Ch^{k+1}\left(|\bf|_{k+1} + \|\bu''\|_{L^2(0,T;H^{k+1}(\Omega;\Rd))} \right)$. 

\rrb{Substituting $\bv_h=\btheta'$ into the above inequality gives} 
$$
m_h(\btheta'',\btheta')+a_{1,h}(\btheta',\btheta')+a_{2,h}(\btheta,\btheta')\le B_{\bf,h,\bu}|\btheta'|_{1,h},
$$
\rrb{Defining $E_1(\btheta)=\frac{1}{2}\left( m_h(\btheta',\btheta')+a_{2,h}(\btheta,\btheta) \right)$, this can be rewritten as}
$$ \frac{\text{d}}{\text{d}t} E_1(\btheta) +a_{1,h}(\btheta',\btheta')\le B_{\bf,h,\bu}|\btheta'|_{1,h}.$$
\rrb{Applying Young's inequality, we obtain}
$$ \frac{\text{d}}{\text{d}t} E_1(\btheta) +a_{1,h}(\btheta',\btheta')\le \frac{B_{\bf,h,\bu}^2}{2 \alpha_0}+\frac{ \alpha_0}{2}|\btheta'|_{1,h}^2. $$
\rrb{Furthermore, using the coercivity of the bilinear form $a_{1,h}(\bullet,\bullet)$, we deduce}
$$\frac{\text{d}}{\text{d}t} E_1(\btheta) + \alpha_0|\btheta'|_{1,h}^2\le \frac{B_{\bf,h,\bu}^2}{2 \alpha_0}+\frac{ \alpha_0}{2}|\btheta'|_{1,h}^2,  $$
\rrb{and subtracting $\frac{ \alpha_0}{2}|\btheta'|_{1,h}^2$ from both sides leads to} 
$$
 \frac{\text{d}}{\text{d}t} E_1(\btheta) +\frac{ \alpha_0}{2}|\btheta'|_{1,h}^2\le \frac{B_{\bf,h,\bu}^2}{2 \alpha_0}.  $$
\rrb{Multiplying by $2$ and integrating in time from $0$ to $t$, we can assert that} 
$$
2 E_1(\btheta(t)) + \alpha_0\int_0^t |\btheta'(s)|_{1,h}^2\ds \le \frac{1}{ \alpha_0}\int_0^t B_{\bf,h,\bu}^2(s)\ds.
$$
From the definition of $B_{\bf,h,\bu}$ and the coercivity of the bilinear forms, the above \rrb{bound implies}
\begin{equation*}
    \|\btheta'(t)\|^2 + |\btheta(t)|_{1,h}^2 \le Ch^{2(k+1)}\left(\|\bf\|_{L^2(0,T;H^{k+1}(\Omega;\Rd))}^2 + \|\bu''\|_{L^2(0,T;H^{k+1}(\Omega;\Rd))}^2 \right).
\end{equation*}
Using the estimates for $\bzeta$ given in Lemma \ref{interprojestimates} together with the estimates of $\btheta$ in the last displayed inequality, we obtain
\begin{equation*}
    |\bu_h-\bu|_{1,h} \le Ch^k\left( \|\bf\|_{L^2(0,T;H^{k+1}(\Omega;\Rd))} + \|\bu\|_{L^2(0,T;H^{k+1}(\Omega;\Rd))} + \|\bu''\|_{L^2(0,T;H^{k+1}(\Omega;\Rd))} \right),
\end{equation*}
and
\begin{equation*}
    \|\bu_h'-\bu'\| \le Ch^{k+1}\left( \|\bf\|_{L^2(0,T;H^{k+1}(\Omega;\Rd))} + \|\bu'\|_{L^2(0,T;H^{k+1}(\Omega;\Rd))} + \|\bu''\|_{L^2(0,T;H^{k+1}(\Omega;\Rd))} \right),
\end{equation*}
which gives the desired result.  
\end{proof}


We next derive the optimal error estimates of the displacement in the $L^2$-norm. The results are presented in the following theorem:

\begin{theorem}\label{err_est_thm_l2}
    Let $\bu_h$ and $\bu$ be the solutions of problems \eqref{dv} and \eqref{vf}, respectively. Assume that $\bu,\bu''$\rrb{,} and $\bf \in L^2(0,T;H^{k+1}(\Omega; \Rd))$. Then, the following bound holds:
    \begin{equation*}
    \|\bu_h-\bu\| \le Ch^{k+1}\left(\|\bf\|_{L^2(0,T;H^{k+1}(\Omega; \Rd))} + \|\bu\|_{L^2(0,T;H^{k+1}(\Omega; \Rd))} + \|\bu''\|_{L^2(0,T;H^{k+1}(\Omega; \Rd))} \right).
\end{equation*} 
\end{theorem}

\begin{proof}
\rrb{We begin with the error equation:} 
\begin{align*}
    m_h(\btheta'',\bv_h) + a_{1,h}(\btheta',\bv_h) + a_{2,h}(\btheta,\bv_h) = (\bf_h-\bf,\bv_h)+ (\bu'',\bv_h) - m_h(\cP^h \bu'',\bv_h).
\end{align*}
We can rewrite the above error equation as:
\begin{equation*}
        -m_h(\btheta',\bv_h') + a_{1,h}(\btheta',\bv_h) + a_{2,h}(\btheta,\bv_h) = (\bf_h-\bf,\bv_h)+ (\bu'',\bv_h) - m_h(\cP^h \bu'',\bv_h)  -\frac{\mathrm{d}}{\dt}m_h(\btheta',\bv_h).
\end{equation*}
\rrb{Defining $\bar{\btheta}(t) = \int_t^\xi \btheta(s)\ds$ for $\xi \in [0,T]$ and substituting $\bv_h = \bar{\btheta}$ into the above equation, we obtain}
\begin{equation*}
        -m_h(\btheta',\bar{\btheta}') + a_{1,h}(\btheta',\bar{\btheta}) + a_{2,h}(\btheta,\bar{\btheta}) = (\bf_h-\bf,\bar{\btheta})+ (\bu'',\bar{\btheta}) - m_h(\cP^h \bu'',\bar{\btheta}) 
        -\frac{\mathrm{d}}{\dt}m_h(\btheta',\bar{\btheta}),
\end{equation*}
\rrb{By Leibniz's rule, $\bar{\btheta}'(t) = -\btheta(t)$. Employing this, the relation reduces to} 
\begin{equation*}
    \begin{split}
        \frac{1}{2}\frac{\mathrm{d}}{\dt}m_h(\btheta,\btheta) + \frac{\mathrm{d}}{\dt}a_{1,h}(\btheta,\bar{\btheta}) + a_{1,h}(\btheta,\btheta) - \frac{1}{2}\frac{\mathrm{d}}{\dt}a_{2,h}(\bar{\btheta},\bar{\btheta}) &= (\bf_h-\bf,\bar{\btheta})+ (\bu'',\bar{\btheta})  \\
        &~~- m_h(\cP^h \bu'',\bar{\btheta})-\frac{\mathrm{d}}{\dt}m_h(\btheta',\bar{\btheta}).
    \end{split}
\end{equation*} 
\rrb{Furthermore, using the bounds for $T_1$ and $T_2$ from Theorem \ref{thm_semidscrt_err_est}, we deduce that}
\begin{align*}
&        \frac{\mathrm{d}}{\dt}m_h(\btheta,\btheta) + 2\frac{\mathrm{d}}{\dt}a_{1,h}(\btheta,\bar{\btheta}) + 2a_{1,h}(\btheta,\btheta) - \frac{\mathrm{d}}{\dt}a_{2,h}(\bar{\btheta},\bar{\btheta}) \\
& \qquad \le Ch^{k+1}\left(|\bf|_{k+1} + |\bu''|_{k+1} 
    + \|\bu''\|_{L^2(0,T;H^{k+1}(\Omega;\Rd))} \right)\|\bar{\btheta}\| - 2\frac{\mathrm{d}}{\dt}m_h(\btheta',\bar{\btheta}).
    \end{align*} 
\rrb{Next, choosing $\xi \in [0,T]$ such that $\displaystyle \|\btheta(\xi)\| = \max_{0\le t \le T}\|\btheta(t)\|$, we have $\|\bar{\btheta}(t)\| \le C\|\btheta(\xi)\|$. Integrating both sides with respect to time from $0$ to $\xi$ and using $\bar{\btheta}(\xi) = \bzero$ along with $\btheta(0) = \btheta'(0) = \bzero$, we arrive at}  
\begin{align*}
    &    m_h(\btheta(\xi),\btheta(\xi)) + a_{2,h}(\bar{\btheta}(0),\bar{\btheta}(0)) + 2\int_0^{\xi}a_{1,h}(\btheta({s}),\btheta({s}))\ds \\
    & \qquad \qquad \le Ch^{k+1}\left(\|\bf\|_{L^2(0,T;H^{k+1}(\Omega;\Rd))} + \|\bu''\|_{L^2(0,T;H^{k+1}(\Omega;\Rd))} \right)\|\btheta(\xi)\|.
    \end{align*} 
\rrb{By the coercivity of $m_h(\bullet,\bullet)$, the above inequality yields}
\begin{equation*}
    \|\btheta(\xi)\| \le Ch^{k+1}\left(\|\bf\|_{L^2(0,T;H^{k+1}(\Omega;\Rd))} + \|\bu''\|_{L^2(0,T;H^{k+1}(\Omega;\Rd))} \right).
\end{equation*} 
\rrb{This, in turn, implies}
\begin{equation*}
    \|\btheta(t)\| \le Ch^{k+1}\left(\|\bf\|_{L^2(0,T;H^{k+1}(\Omega;\Rd))} + \|\bu''\|_{L^2(0,T;H^{k+1}(\Omega;\Rd))} \right).
\end{equation*} 
\rrb{Combining \eqref{interprojerrl2} with the above estimate concludes the proof.} 
\end{proof}

\section{Fully discrete case}\label{sec:fully-discr}
\rrb{In this section, we reformulate} the model problem \eqref{model_prblm} as a system of first-order partial differential equations and present the fully discrete problem corresponding to it, obtained by employing the Crank--Nicolson scheme for temporal discretization. \rrb{Letting $\bv = \bu'$, problem \eqref{model_prblm} reduces to finding $\bu$ and $\bv$ such that}
\begin{equation}\label{pde_system}
    \begin{split}
        \bu' &= \bv, \quad  \text{in} \ \Omega \times (0,T], \\
        \rho \bv' - \bnabla \cdot (A_1 \bepsilon(\bv) + A_2 \bepsilon(\bu)) &= \bf, \quad \text{in} \ \Omega \times (0,T],
    \end{split}
\end{equation}
with the initial conditions $\bu(0) = \bu_0$ \rrb{and} $\bv(0) = \bu_1$.

\rrb{Let $N$ be a positive integer, denote the time-step by $\Delta t = T/N$, and define the temporal nodes as $t_n = n\Delta t$ for $n = 0,1,\ldots, N$.} For a generic scalar or vector field $s$, \rrb{we} let $s^n$ denote \rrb{$s(t_n)$} and $s^{n+\frac{1}{2}}$ denote $\frac{s^n+s^{n+1}}{2}$. \rrb{Furthermore, we define the discrete time derivative as} $\partial s^n \rrb{:=} \frac{s^{n+1}-s^n}{\Delta t}$. 
\rrb{The fully discrete Crank--Nicolson scheme applied to \eqref{pde_system} then consists of finding $\underline{\bu}_h^n, \underline{\bv}_h^n \in \bV_h$ for all $n=0,1,\ldots,N$ such that $\underline{\bu}_h^0 = \bu_{h,0}$ and $\underline{\bv}_h^0 = \bu_{h,1}$, and for all $n = 0, \dots, N-1$, there holds:}  
\begin{equation}\label{fd}
\begin{split}
&\partial \underline{\bu}_h^n=\underline{\bv}_h^{n+\frac{1}{2}},\\
&m_h(\partial \underline{\bv}_h^n,\bw_h)+ a_{1,h}(\partial \underline{\bu}_h^n,\bw_h)+a_{2,h}(\underline{\bu}_h^{n+\frac{1}{2}},\bw_h)=(\bf_h^{n+\frac{1}{2}},\bw_h)\rrb{,\quad \forall \ \bw_h\in \bV_h.}
\end{split}
\end{equation}

We now \rrb{derive} the error estimates for the fully discrete problem. 
\begin{theorem}\label{th:fully}
Let \rrb{$\bu_h^n := \bu_h(t_n)$} and $\underline{\bu}_h^n$ be the solutions to \rrb{problems} \eqref{dv} and \eqref{fd}, respectively. Assume that $\bu,\bu',\bu'' \in L^2(0,T;H^{k+1}(\Omega;\Rd))$ \rrb{and} $\bu_0,\bu_1 \in H^{k+1}(\Omega;\Rd) \cap H^6(\Omega;\Rd)$\rrb{, and suppose} further that $\bf \in L^2(0,T;H^{k+1}(\Omega;\Rd))\cap H^3(0,T;H^{4}(\Omega;\Rd))$. Then, it holds that
\begin{align*}
    \max_{0\le n\le N} (\|\bu'^n-\underline{\bv}_h^n\|+h|\bu^n-\underline{\bu}_h^n|_{1,h}) & \le Ch^{k+1}\left( \|\bf\|_{L^2(0,T;H^{k+1}(\Omega;\Rd))} + \|\bu\|_{L^2(0,T;H^{k+1}(\Omega;\Rd))}\right. \\
    &~~\left.  + \|\bu'\|_{L^2(0,T;H^{k+1}(\Omega;\Rd))}  + \|\bu''\|_{L^2(0,T;H^{k+1}(\Omega;\Rd))} \right) \\
    &~~+C \Delta t^2 \left( |\bu_0|_6 + |\bu_1|_6 + \|\bf\|_{H^3(0,T;H^4(\Omega;\Rd))} \right).
\end{align*}
\end{theorem}

\begin{proof}
Define $\bxi^n=\bu_h^n-\underline{\bu}_h^n$, \rrb{$\bbeta^n$}$=\bu_h'^n-\underline{\bv}_h^n$, $\bpsi^n=\partial \bu_h'^n-\bu_h''^{n+\frac{1}{2}}$\rrb{,} and $\bphi^n=\partial \bu_h^n-\bu_h'^{n+\frac{1}{2}}$. \rrb{Subtracting the fully discrete equation \eqref{fd} from the semi-discrete equation \eqref{dv} evaluated at $t = t_{n+1/2}$ yields the following for all $\bv_h\in \bV_h$ and $n=0,\dots,N-1$:}
$$ m_h(\partial\bbeta^n,\bv_h)+a_{1,h}(\bbeta^{n+\frac{1}{2}},\bv_h)+a_{2,h}(\bxi^{n+\frac{1}{2}},\bv_h)=m_h(\bpsi^n,\bv_h).$$
\rrb{Substituting $\bv_h=\bbeta^{n+\frac{1}{2}}$ and} noting that $\bbeta^{n+\frac{1}{2}}=\partial \bxi^n-\bphi^n$, we obtain
$$ m_h(\partial\bbeta^n,\bbeta^{n+\frac{1}{2}})+a_{1,h}(\bbeta^{n+\frac{1}{2}},\bbeta^{n+\frac{1}{2}})+a_{2,h}(\bxi^{n+\frac{1}{2}},\partial \bxi^n-\bphi^n)=m_h(\bpsi^n,\bbeta^{n+\frac{1}{2}})\rrb{,}$$
\rrb{and adding} $a_{2,h}(\bxi^{n+\frac{1}{2}},\bphi^n)$ to both sides of the above equation \rrb{yields}
$$ m_h(\partial\bbeta^n,\bbeta^{n+\frac{1}{2}})+a_{1,h}(\bbeta^{n+\frac{1}{2}},\bbeta^{n+\frac{1}{2}})+a_{2,h}(\bxi^{n+\frac{1}{2}},\partial \bxi^n)=m_h(\bpsi^n,\bbeta^{n+\frac{1}{2}})+a_{2,h}(\bxi^{n+\frac{1}{2}},\bphi^n).$$
\rrb{Defining} $E_2^n=\frac{1}{2}\left(m_h(\bbeta^n,\bbeta^n)+a_{2,h}(\bxi^n,\bxi^n)\right)$, we get the following \rrb{relations}:
\begin{align*}
\partial E_2^n+a_{1,h}(\bbeta^{n+\frac{1}{2}},\bbeta^{n+\frac{1}{2}}) & =  m_h(\bpsi^n,\bbeta^{n+\frac{1}{2}})+a_{2,h}(\bxi^{n+\frac{1}{2}},\bphi^n),
\\ & \le \|\bpsi^n\|\|\bbeta^{n+\frac{1}{2}}\|+|\bphi^n|_{1,h}|\bxi^{n+\frac{1}{2}}|_{1,h},
\\ & \le \rrb{C_1}\left(\|\bpsi^n\|^2+|\bphi^n|_{1,h}^2\right)+\rrb{C_2}\left(\|\bbeta^{n+\frac{1}{2}}\|^2+|\bxi^{n+\frac{1}{2}}|_{1,h}^2\right)\rrb{,}
\end{align*}
\rrb{where} the last inequality \rrb{follows from} Young's inequality. \rrb{Multiplying} by $\Delta t$ and \rrb{summing over} $n=0,1,\dots,l-1$ for $l\le N$ \rrb{gives}:
\begin{align*}
E_2^l-E_2^0  & \le \Delta t \sum_{n=0}^{l-1} C_1(\|\bpsi^n\|^2+|\bphi^n|_{1,h}^2)+\Delta t \sum_{n=0}^{l-1} C_2(\|\bbeta^{n+\frac{1}{2}}\|^2+|\bxi^{n+\frac{1}{2}}|_{1,h}^2),
\\ & \le \Delta t \sum_{n=0}^{N-1} C_1(\|\bpsi^n\|^2+|\bphi^n|_{1,h}^2)+\Delta t \max_{0\le n\le N} N C_2(\|\bbeta^n\|^2+|\bxi^n|_{1,h}^2),
\\ & = \Delta t \sum_{n=0}^{N-1} C_1(\|\bpsi^n\|^2+|\bphi^n|_{1,h}^2)+T \max_{0\le n\le N}  C_2(\|\bbeta^n\|^2+|\bxi^n|_{1,h}^2).
\end{align*}
Noting that $E_2^0=0$\rrb{, and} using the definition of $E_2^n$ \rrb{along with} the coercivity of the bilinear forms \rrb{and a suitable choice of the constants} $C_1$ and $C_2$, we \rrb{deduce that}
$$
\max_{0\le n\le N} (\|\bbeta^n\|^2+|\bxi^n|_{1,h}^2) \le C \Delta t \sum_{n=0}^{N-1} (\|\bpsi^n\|^2+|\bphi^n|_{1,h}^2).
$$
\rrb{Using} standard approximation estimates on the right-hand side\rrb{, we find}
\begin{align*}
\max_{0\le n\le N} (\|\bbeta^n\|^2+|\bxi^n|_{1,h}^2) & \le C \Delta t^4 \sum_{n=0}^{N-1} \int_{\rrb{t_n}}^{\rrb{t_{n+1}}} (\|\bu_h''''(t)\|^2+| \bu_h'''(t)|_{1,h}^2)\dt.
\end{align*}
\rrb{Combining} the integrals \rrb{over $[0,T]$} and making use of the estimates \rrb{from} Theorem \ref{thm_semidscrt_err_est}, we obtain
\begin{align*}
    \max_{0\le n\le N} (\|\bu'^n-\underline{\bv}_h^n\|+h|\bu^n-\underline{\bu}_h^n|_{1,h}) & \le Ch^{k+1}\left( \|\bf\|_{L^2(0,T;H^{k+1}(\Omega;\Rd))} + \|\bu\|_{L^2(0,T;H^{k+1}(\Omega;\Rd))}\right. \\
    &~~\left.  + \|\bu'\|_{L^2(0,T;H^{k+1}(\Omega;\Rd))}  + \|\bu''\|_{L^2(0,T;H^{k+1}(\Omega;\Rd))} \right) \\
    &~~+C \Delta t^2 \int_{0}^{T} (\|\bu_h''''(t)\|+| \bu_h'''(t)|_{1,h})\dt,
\end{align*}
and \rrb{applying Lemma \ref{lemma_stab_est_der} then yields} the desired bound. 
\end{proof}

\section{Numerical experiments}\label{sec:results}
This section illustrates the accuracy and performance of the proposed fully-discrete scheme  through some numerical tests.  We show the optimal behavior of the VEM under different polytopal meshes and time steps for the Crank--Nicholson scheme. Finally, we simulate an application-oriented problem.  

We start by defining the polynomial projection operator onto symmetric polynomial tensors \rrb{needed in 2D}. Given $\bv_h \in \mathbf{V}_h$ and $K \in \mathcal{T}_h$, the strain projection operator $\Pi^{\boldsymbol{\epsilon},K} : \mathbf{V}^K_h \to \boldsymbol{\epsilon}\left([\mathbb{P}_{k}(K)]^2\right)$ is defined by:
\begin{equation*}
\int_K \Pi^{\boldsymbol{\epsilon},K}(\bv_h) : \mathbb{q}_{k-1} = \int_K \boldsymbol{\epsilon}(\bv_h) : \mathbb{q}_{k-1} \quad \forall \mathbb{q}_{k-1} \in \boldsymbol{\epsilon}\left([\mathbb{P}_{k}(K)]^2\right).
\end{equation*}
To establish the connection between the elliptic projection $\mathbf{\Pi}^{\boldsymbol{\nabla},K}$ and the strain projection $\Pi^{\boldsymbol{\epsilon},K}$, we recall the definition of $\mathbf{\Pi}^{\boldsymbol{\nabla},K} \bv_h \in [\mathbb{P}_k(K)]^2$:
\begin{align*}
\int_K \boldsymbol{\epsilon} (\mathbf{\Pi}^{\boldsymbol{\nabla},K} \bv_h) : \boldsymbol{\epsilon}(\mathbf{q}_k) &= \int_K \boldsymbol{\epsilon}(\bv_h) : \boldsymbol{\epsilon}(\mathbf{q}_k) \quad \forall \mathbf{q}_k \in [\mathbb{P}_k(K)]^2,\\
j_K(\mathbf{\Pi}^{\boldsymbol{\nabla},K} \bv_h) &= j_K(\bv_h).
\end{align*}
Notice that any symmetric tensor polynomial $\mathbb{q}_{k-1} \in \boldsymbol{\epsilon}\left([\mathbb{P}_{k}(K)]^2\right)$ can be represented as $\mathbb{q}_{k-1} = \boldsymbol{\epsilon}(\mathbf{q}_k)$ for some $\mathbf{q}_k \in [\mathbb{P}_k(K)]^2$. Restricted to these tensors, the orthogonality relation reduces to:
\begin{equation*}
\int_K \boldsymbol{\epsilon}(\mathbf{\Pi}^{\boldsymbol{\nabla},K} \bv_h) : \mathbb{q}_{k-1} = \int_K \boldsymbol{\epsilon}(\bv_h) : \mathbb{q}_{k-1} \quad \forall \mathbb{q}_{k-1} \in \boldsymbol{\epsilon}\left([\mathbb{P}_{k}(K)]^2\right).
\end{equation*}
Because $\boldsymbol{\epsilon}\left(\mathbf{\Pi}^{\boldsymbol{\nabla},K} \bv_h\right) \in \boldsymbol{\epsilon}\left([\mathbb{P}_k(K)]^2\right)$, uniqueness of the $L^2$-projection onto $\boldsymbol{\epsilon}\left([\mathbb{P}_k(K)]^2\right)$ yields the identity
\begin{equation*}
\Pi^{\boldsymbol{\epsilon},K}\bv_h = \boldsymbol{\epsilon}\left(\mathbf{\Pi}^{\boldsymbol{\nabla},K} \bv_h\right).
\end{equation*}

We now define the total computable error via the local polynomial approximation of the discrete solutions in the last time step as \begin{align*}
    \bar{\mathrm{e}}_{h}^* := \begin{cases}
    \displaystyle{\sum_{K\in \mathcal{T}_h}\frac{\|{(\bu')}^n-\Pi^{0,K}\underline{\bv}_h^n\|_K}{\|{(\bu')}^n\|_K}+h\sum_{K\in \mathcal{T}_h}\frac{\|\bepsilon(\bu^n)-\Pi^{\bepsilon,K}\underline{\bu}_h^n\|_{K}}{\|\bepsilon(\bu^n)\|_{K}}} \quad \text{in 2D},\\
    \displaystyle{\sum_{K\in \mathcal{T}_h}\frac{\|{(\bu')}^n-\Pi^{0,K}\underline{\bv}_h^n\|_K}{\|{(\bu')}^n\|_K}+h\sum_{K\in \mathcal{T}_h}\frac{\|\bepsilon(\bu^n)-\bepsilon(\Pi^{\nabla,K}\underline{\bu}_h^n)\|_{K}}{\|\bepsilon(\bu^n)\|_{K}}} \quad \text{in 3D}.
\end{cases}
\end{align*}
We remark that the observed difference between the two- and three-dimensional error metrics is a consequence of the projection space definitions implemented in the open-source library \texttt{vem++} \cite{dassi2023vem++}, which serves as the computational framework for the forthcoming experiments. Whereas the experimental rate of convergence $r(\bullet)$ in the refinement $1\leq j$ is computed from the formula $$r(\bar{\textnormal{e}}_h^*)^{j+1} = \frac{\log\left(\frac{(\bar{\textnormal{e}}_h^*)^{j+1}}{(\bar{\textnormal{e}}_h)^{j}}\right)}{\log\left(\frac{h^{j+1}}{h^{j}}\right)},$$ where $(\bar{\textnormal{e}}_h)^{j}$ and $h^j$ denote the total computable error and mesh size in the refinement $j$, respectively.

Finally, the stabilization term $S^{K}(\bu_h,\bv_h)$ follows the known ``\texttt{D-recipe}" first introduced in \cite{BEIRAODAVEIGA2017}. A comprehensive discussion and detailed formulation of this operator can also be found in \cite{ARTIOLI2020}. 

\begin{figure}[!h]
    \centering
    \subfigure[Quadrilateral]{\includegraphics[width=0.24\textwidth,trim={9.5cm 1.25cm 8.75cm  1.5cm},clip]{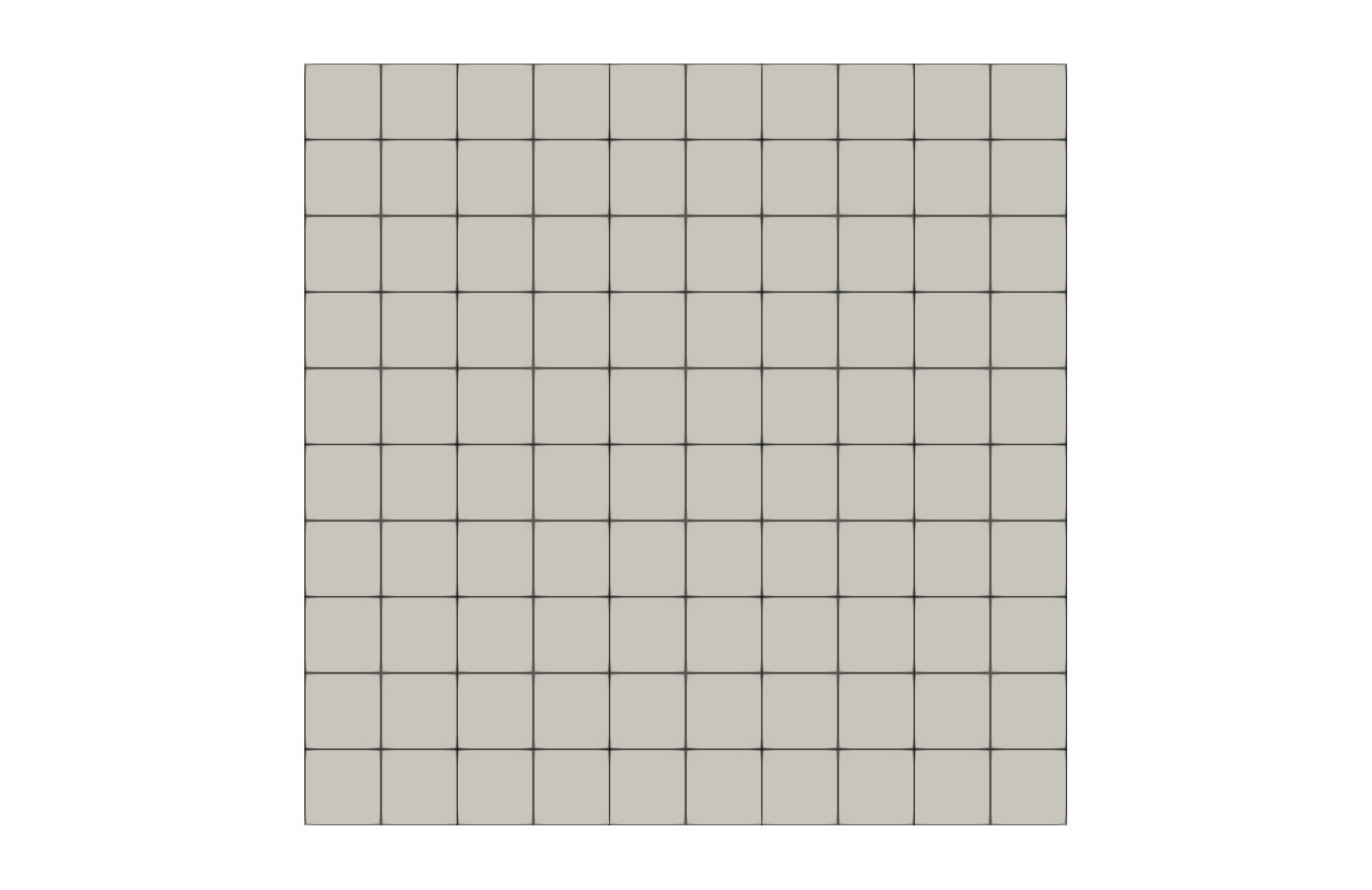}} 
    \subfigure[Distorted]{\includegraphics[width=0.24\textwidth,trim={9.5cm 1.25cm 8.75cm  1.5cm},clip]{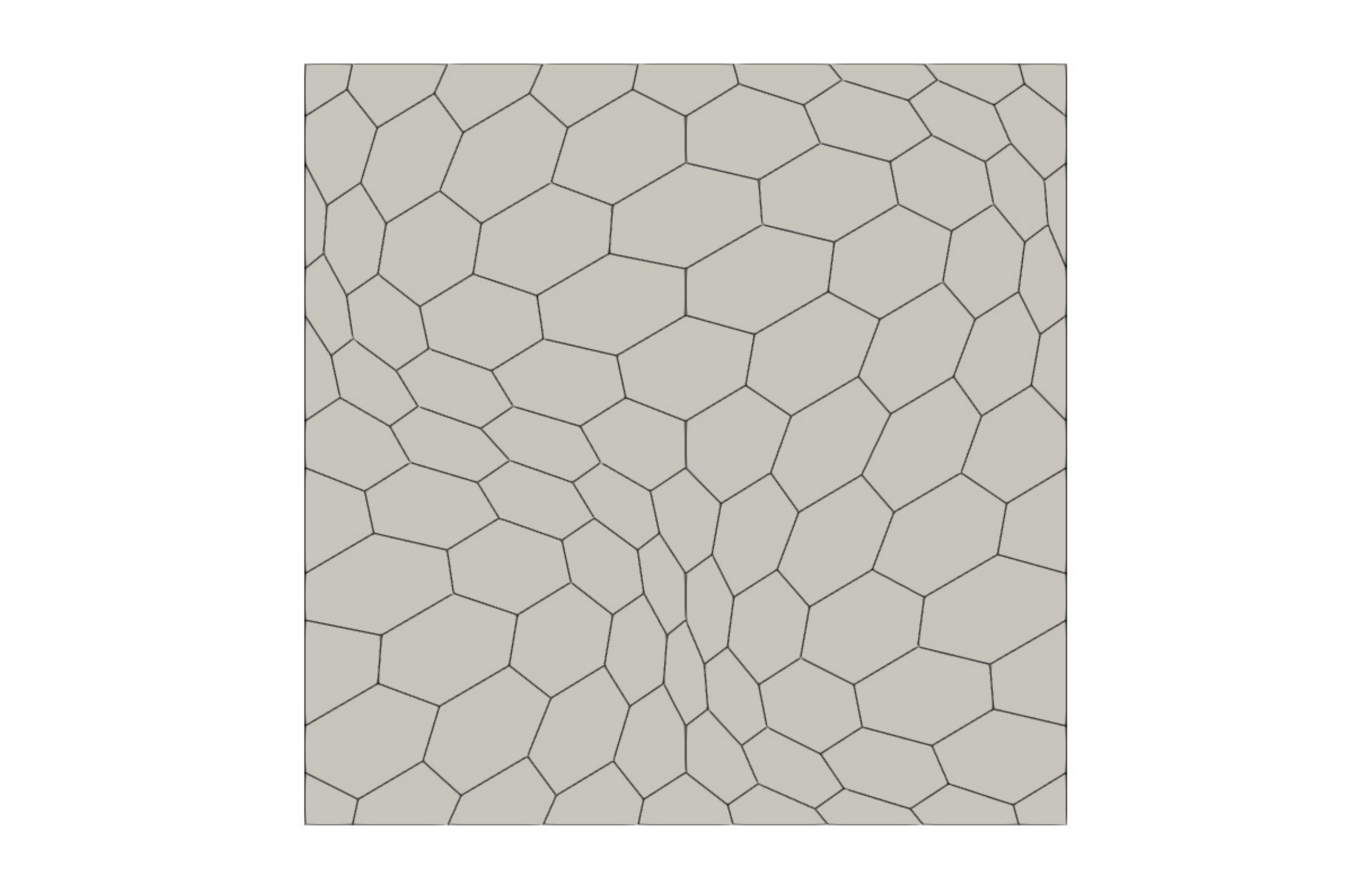}}  
    \subfigure[Hexagonal]{\includegraphics[width=0.24\textwidth,trim={9.5cm 1.25cm 8.75cm  1.5cm},clip]{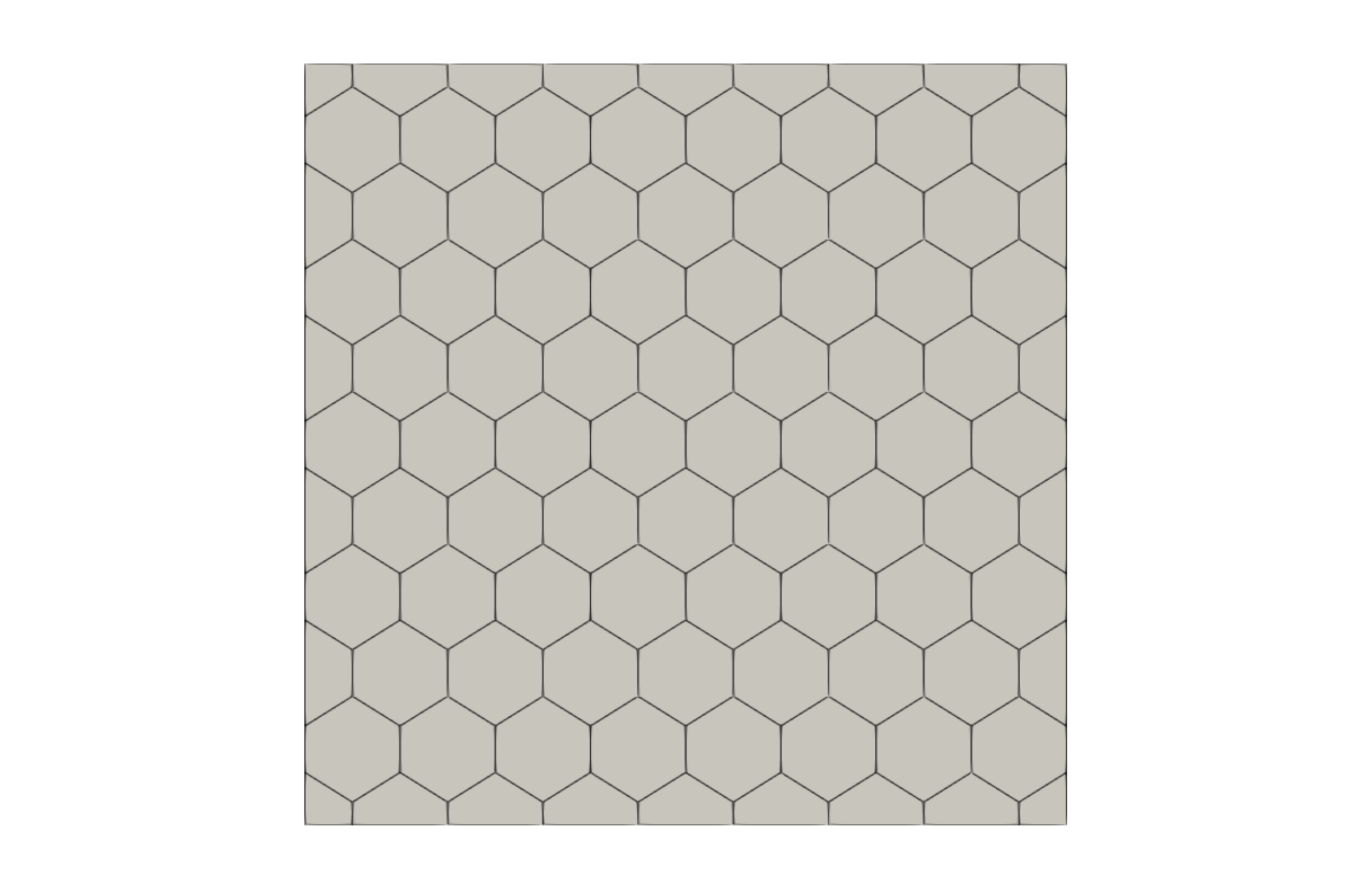}}
    \subfigure[Triangular]{\includegraphics[width=0.24\textwidth,trim={9.5cm 1.25cm 8.75cm  1.5cm},clip]{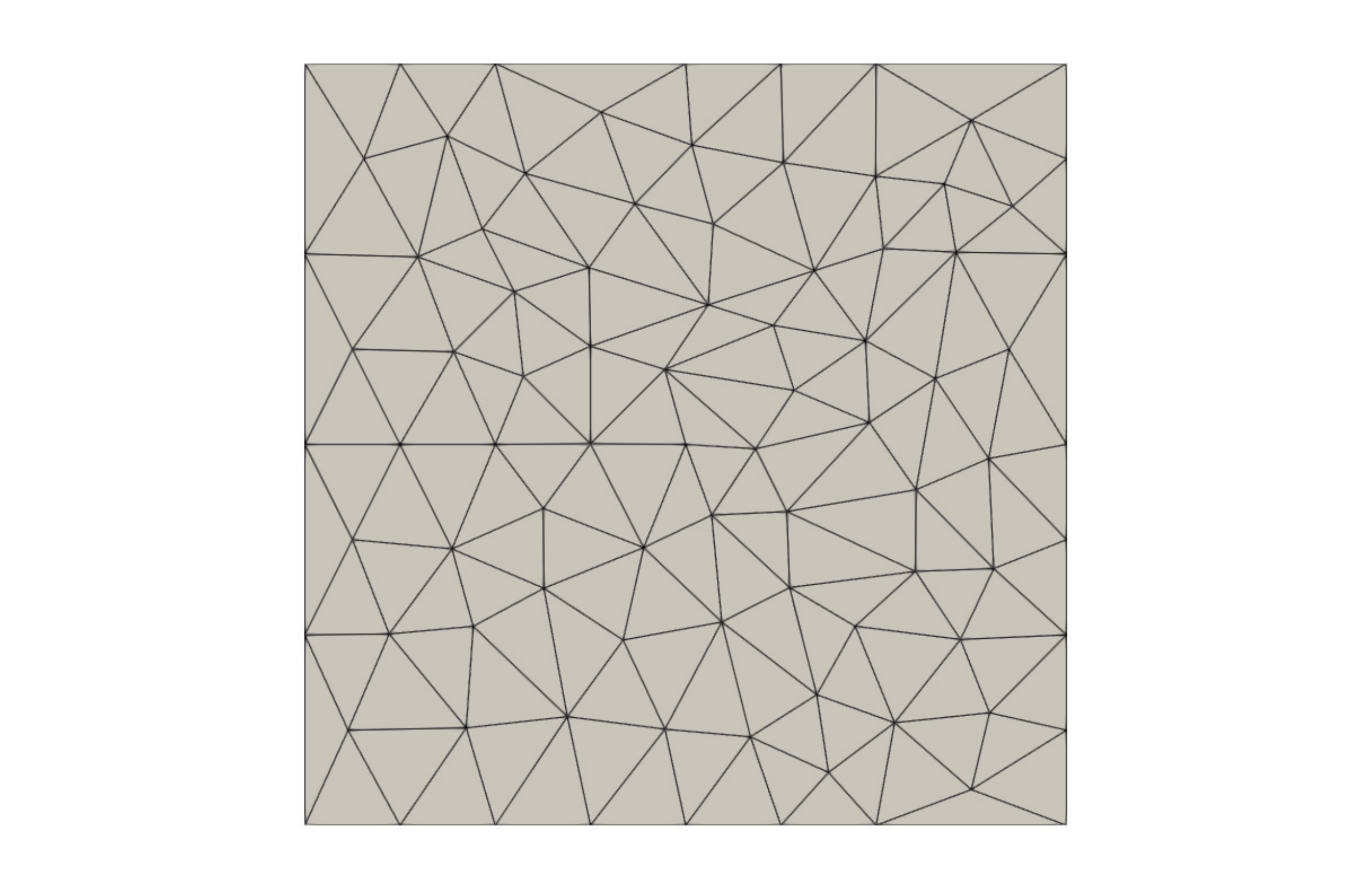}}
    \caption{Experiment 1. Variety of meshes used in the two-dimensional convergence test under $h$-refinement and $\Delta t$-refinement.}\label{fig:meshes2D}
\end{figure}

\subsection{Experiment 1: Convergence under $h$-refinement and $\Delta t$-refinement in two-dimensions}\label{sec:exp1}
This numerical example solves \eqref{model_prblm} for the unit square domain $\Omega=(0,1)^2$ using the proposed space-time scheme (cf. Section~\ref{sec:fully-discr}) across the different spatial mesh discretizations shown in Figure~\ref{fig:meshes2D}. We point out that the physical parameters are set to unity values for the remainder of this test.

We first evaluate the error $\bar{\mathrm{e}}_{h}^*$ under successive $h$-refinement using the manufactured solution
$$\bu(x,y;t)=t^2x^2y^2(1-x)^2(1-y)^2(1,1)^{\tt t}.$$
Because the second-order Crank--Nicolson temporal scheme integrates quadratic time functions exactly, the temporal truncation error vanishes, allowing us to isolate and measure the spatial error exclusively. In this case, we set the time step as $\Delta t = 0.01$.

Figure~\ref{fig:convergence2dSpace} presents the spatial convergence curves for all considered mesh families and polynomial degrees $k \in \{1,2,3,4\}$. As predicted by Theorem~\ref{th:fully}, the optimal spatial convergence rate of $\mathcal{O}(h^{k+1})$ is achieved across all test cases.

\begin{figure}[!h]
    \centering
    \includegraphics[width=\linewidth]{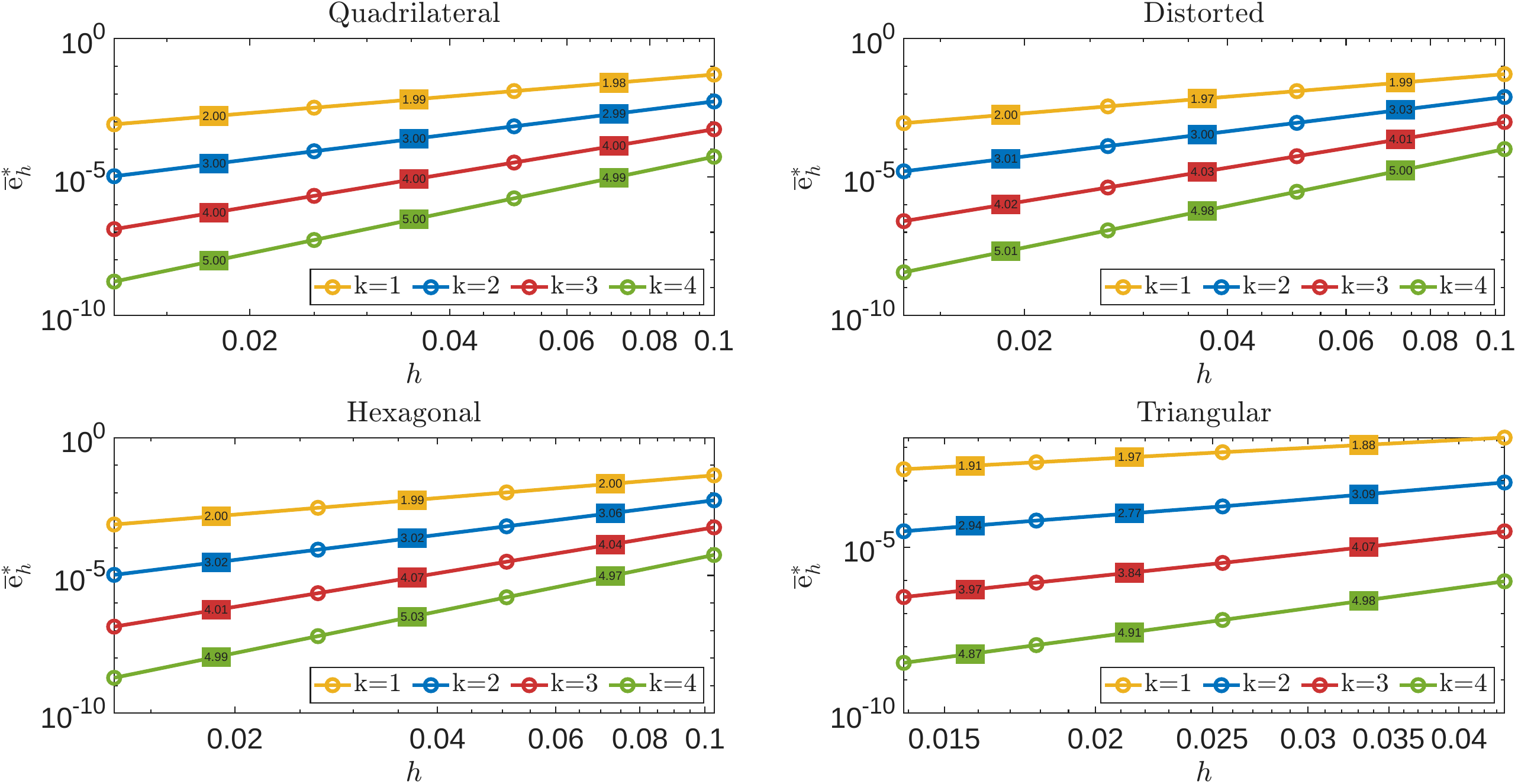}
    \caption{Experiment 1. Behavior of the computable error $\overline{\textnormal{e}}_h^*$ across various meshes and different polynomial degrees, under uniform $h$-refinement.}
    \label{fig:convergence2dSpace}
\end{figure}

Next, we evaluate the temporal convergence rate using the manufactured solution
$$\bu(x,y;t)=(1-\cos(2\pi t))xy(1-x)(1-y)(1,1)^{\tt t}.$$
To prevent spatial errors from polluting the temporal rate, the spatial mesh is fixed to its finest refinement level ($h<0.02$) and the VEM polynomial degree (cf. Section~\ref{sec:vem}) is set to $k=3$.

Figure~\ref{fig:convergence2dTime} displays the resulting temporal error curves for the meshes in Figure~\ref{fig:meshes2D}. The numerical results confirm that the proposed scheme achieves the optimal second-order temporal convergence rate of $\mathcal{O}(\Delta t^2)$ (see Theorem~\ref{th:fully}).

\begin{figure}[h!]
    \centering
    \includegraphics[width=\linewidth]{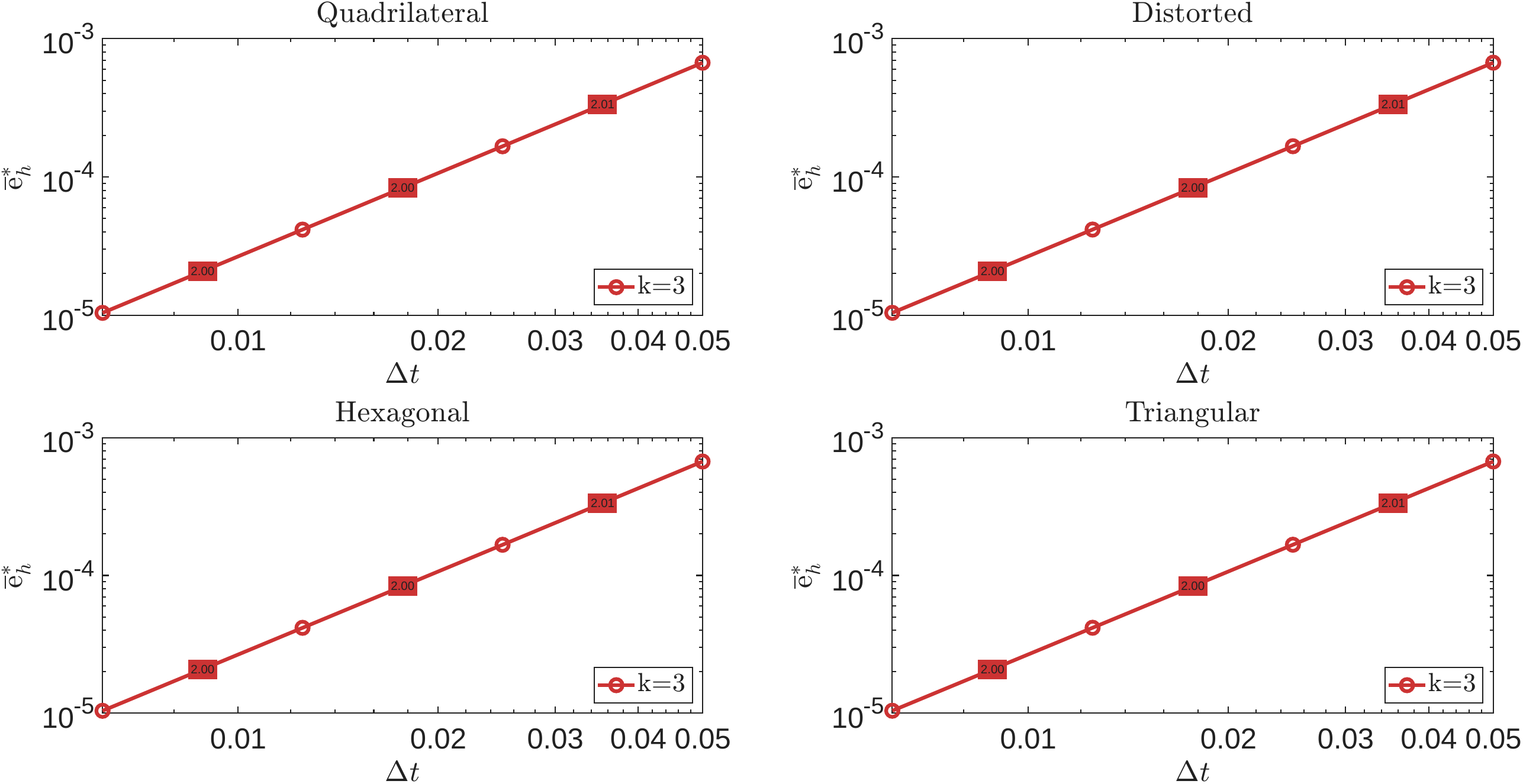}
    \caption{Experiment 1. Behavior of the computable error $\overline{\textnormal{e}}_h^*$ across various meshes for $k=3$, under uniform $\Delta t$-refinement.}
    \label{fig:convergence2dTime}
\end{figure}

\begin{figure}[!h]
    \centering
    \subfigure[Voronoi]{\includegraphics[width=0.3\textwidth,trim={7.5cm 0.cm 8.25cm  1.5cm},clip]{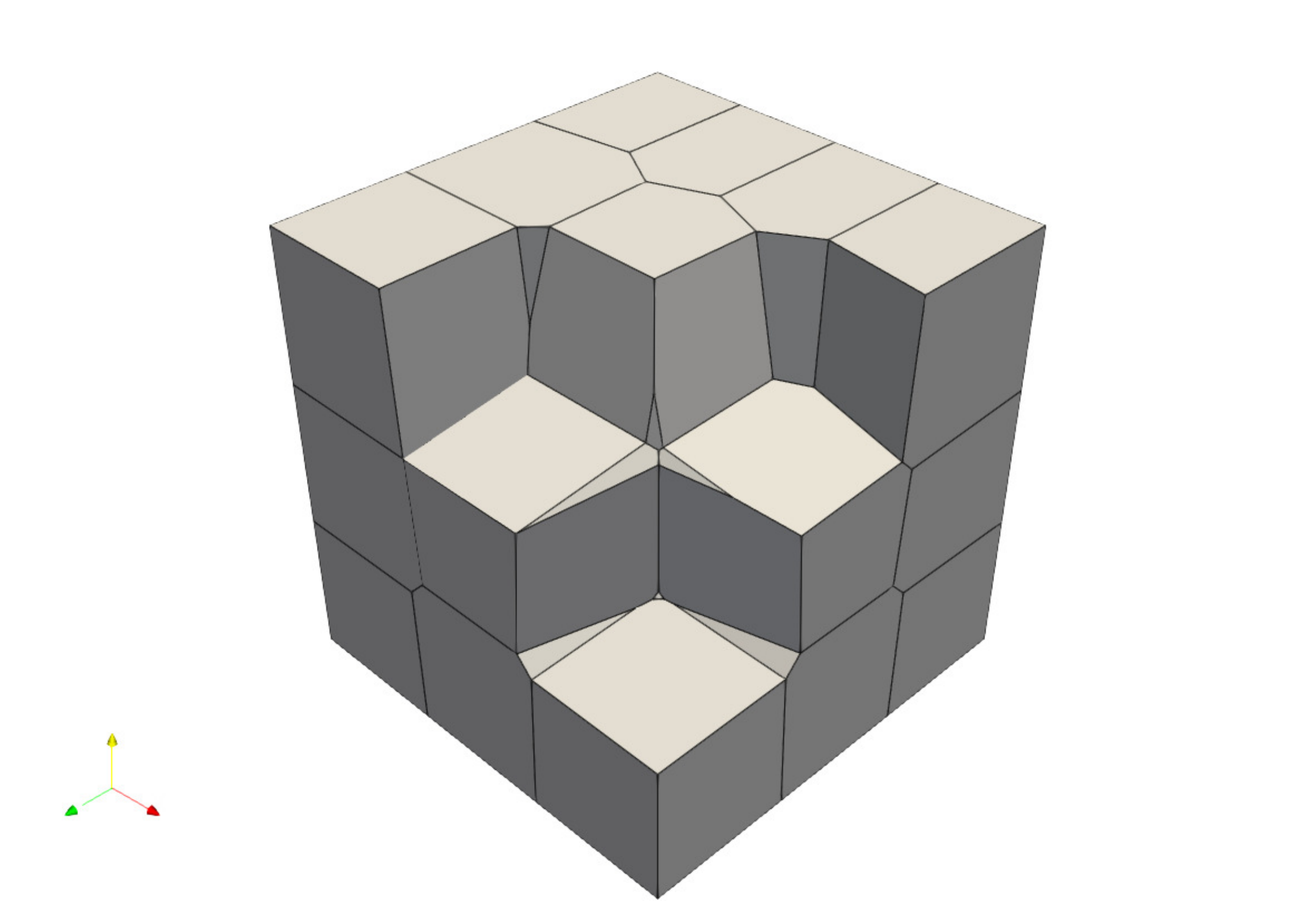}} 
    \subfigure[Random]{\includegraphics[width=0.3\textwidth,trim={7.5cm 0.cm 8.25cm  1.5cm},clip]{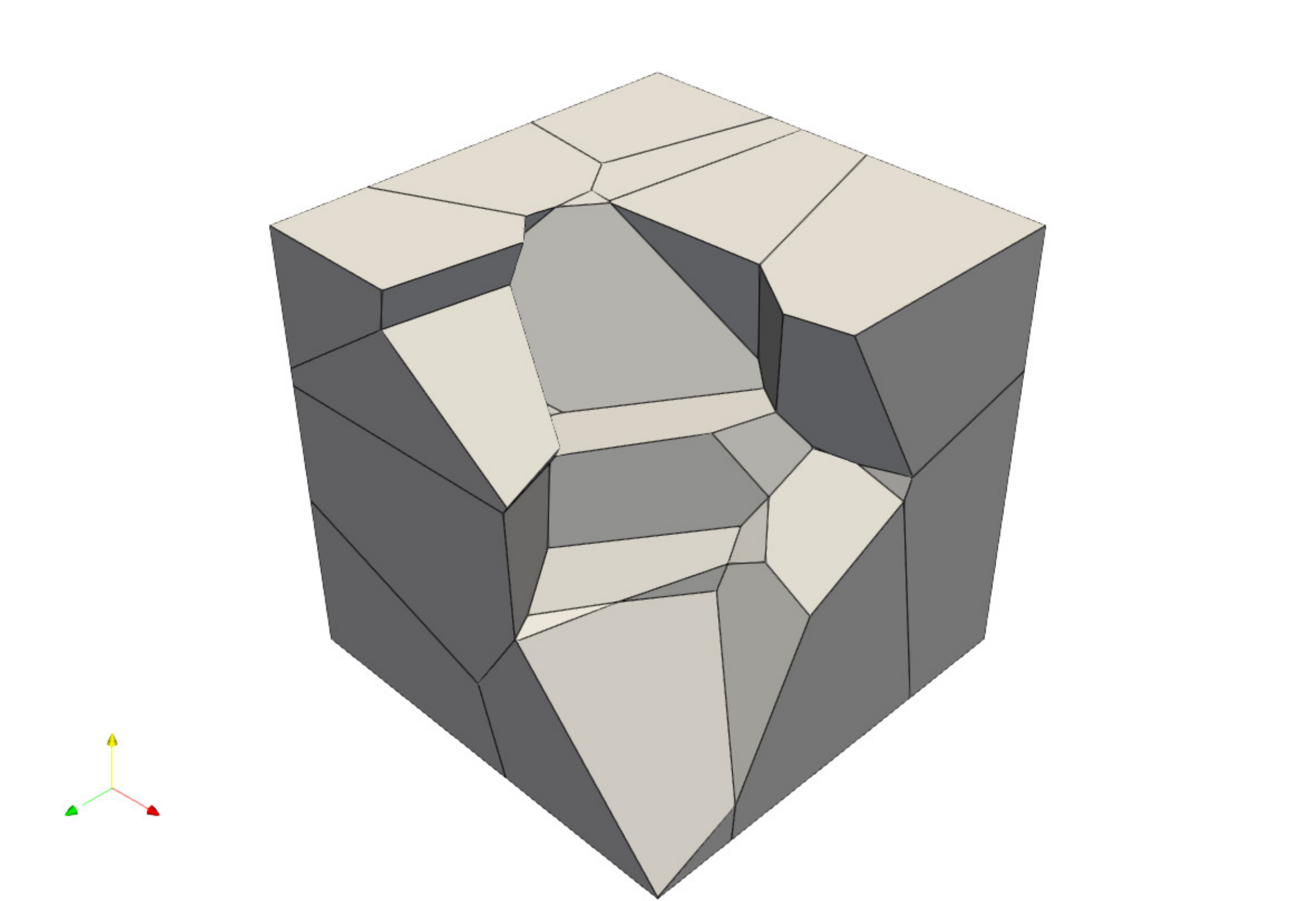}}  
    \subfigure[Cubical]{\includegraphics[width=0.3\textwidth,trim={7.5cm 0.cm 8.25cm  1.5cm},clip]{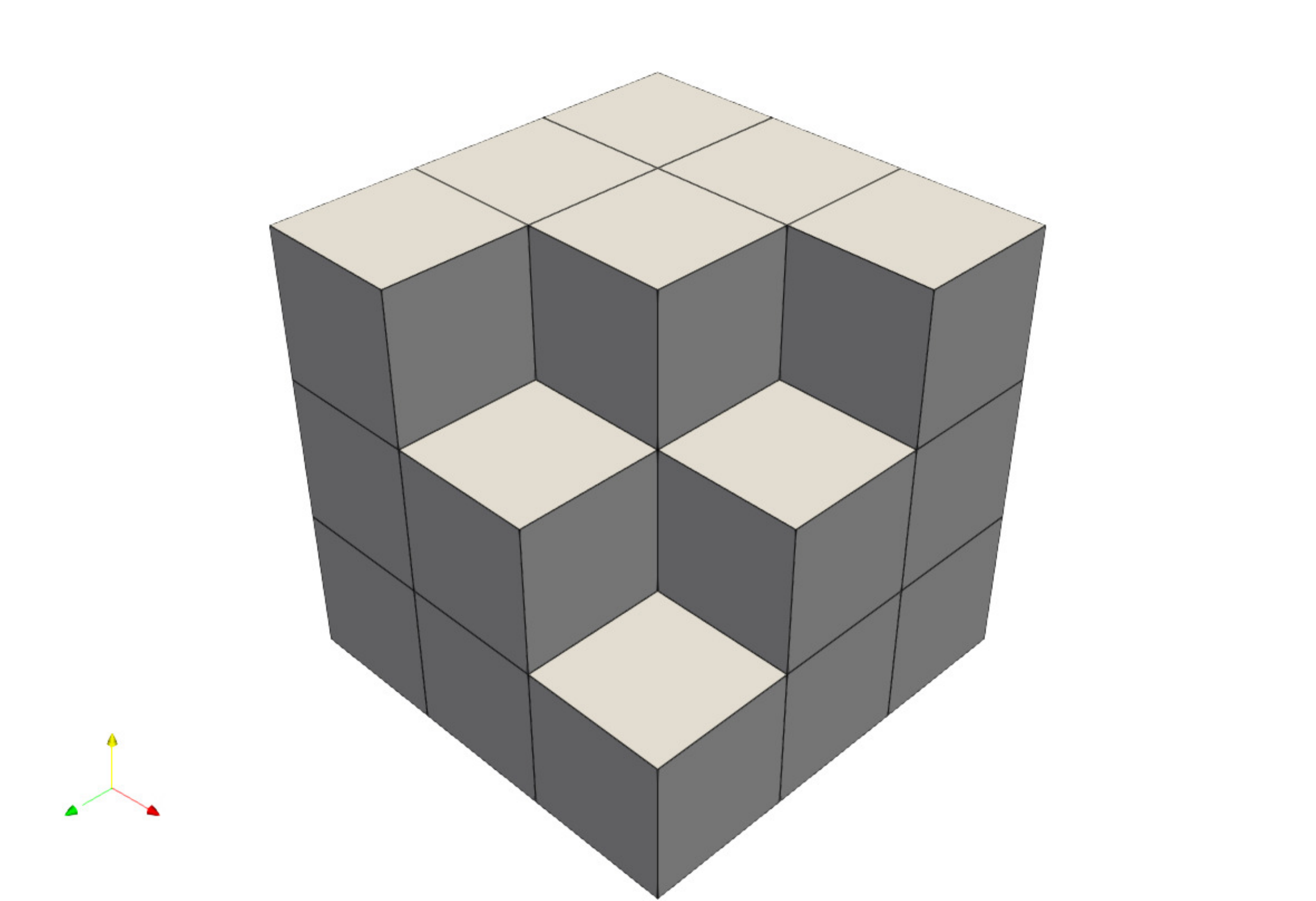}}
    \caption{Experiment 2. Variety of meshes used in the three-dimensional convergence test under $h$-refinement and $\Delta t$-refinement.}\label{fig:meshes3D}
\end{figure}

\subsection{Experiment 2: Convergence under $h$-refinement and $\Delta t$-refinement in three-dimensions}
We now extend the numerical experiments of Section~\ref{sec:exp1} to three dimensions. Specifically, we solve \eqref{model_prblm} on the unit cube domain $\Omega = (0,1)^3$, discretized using the mesh element types shown in Figure~\ref{fig:meshes3D}. As before, all physical parameters are set to unity.

The behavior of the error $\bar{\mathrm{e}}_{h}^*$ is evaluated for $h$-refinement and $\Delta t$-refinement  by defining the following manufactured solutions:
\begin{align*}
    \bu(x,y,z;t)&=t^2 \sin(\pi x) \sin(\pi y) \sin(\pi z)(1,1)^{\tt t}, \text{ and }\\
    \bu(x,y,z;t)&=(1-\cos(2\pi t))\sin(\pi x) \sin(\pi y) \sin(\pi z)(1,1,1)^{\tt t},
\end{align*}
respectively. The simulation presented in this test is defined for the lowest-order 3D VEM space $k=1$ (cf. Section~\ref{sec:vem}). Regarding the spatial test, the time step is set to $\Delta t = 0.01\text{s}$. For the time test, we set the three meshes under consideration with $16\text{,}000$ total number of elements which gives e.g. $313\text{,}746$ Degrees of Freedom ($h<0.07$) for the random mesh.

Figures~\ref{fig:convergence3dSpace} and \ref{fig:convergence3dTime} illustrate the respective spatial and time convergence curves for all considered mesh families. As predicted by Theorem~\ref{th:fully}, the optimal spatial convergence rate of $\mathcal{O}(h^{k+1})$ is achieved across all the meshes tested in the spatial case. Additionally, time discretization displays apparent superconvergence. 

\begin{figure}[!h]
    \centering
    \includegraphics[width=\linewidth]{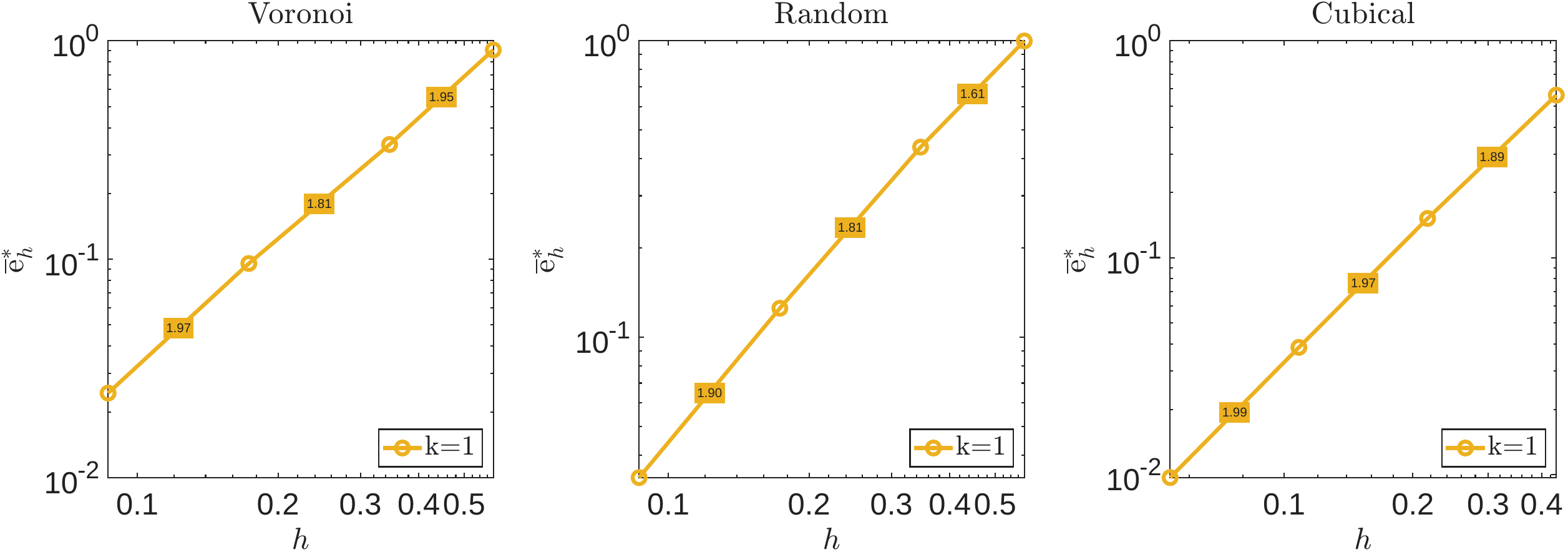}
    \caption{Experiment 2. Behavior of the computable error $\overline{\textnormal{e}}_h^*$ across various meshes for the lowest-case VEM, under uniform $h$-refinement.}
    \label{fig:convergence3dSpace}
\end{figure}

\begin{figure}[!h]
    \centering
    \includegraphics[width=\linewidth]{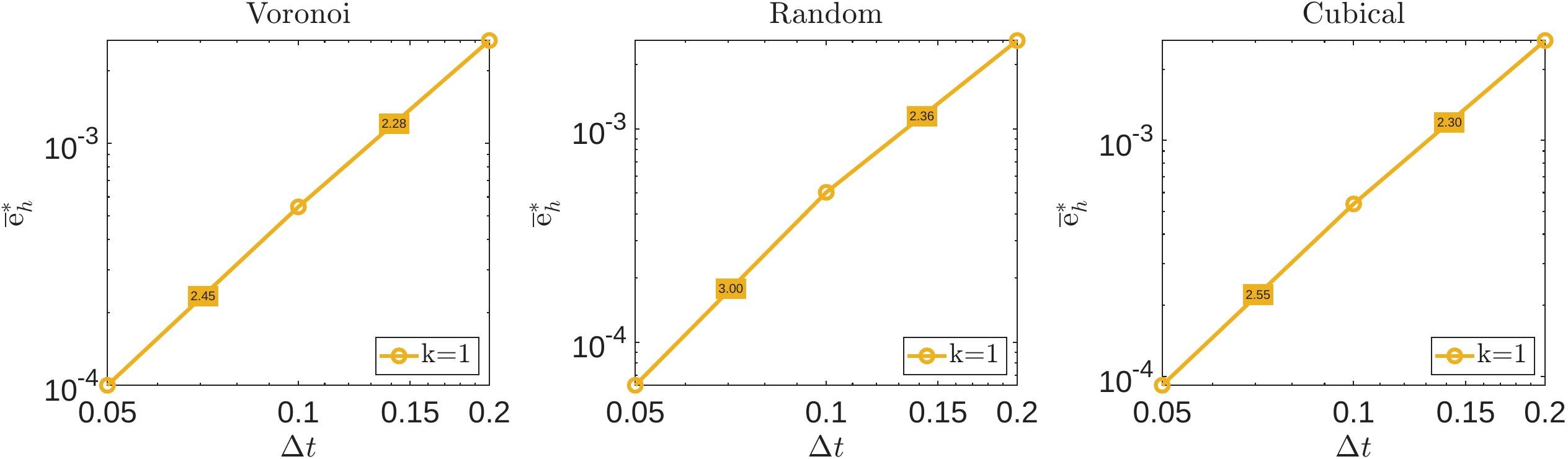}
    \caption{Experiment 2. Behavior of the computable error $\overline{\textnormal{e}}_h^*$ across various meshes for the lowest-case VEM, under uniform $\Delta t$-refinement.}
    \label{fig:convergence3dTime}
\end{figure}

\begin{figure}[!h]
    \centering
    \subfigure[single-perforation]{\includegraphics[width=0.35\textwidth,trim={7.5cm 4.cm 7.5cm  4.cm},clip]{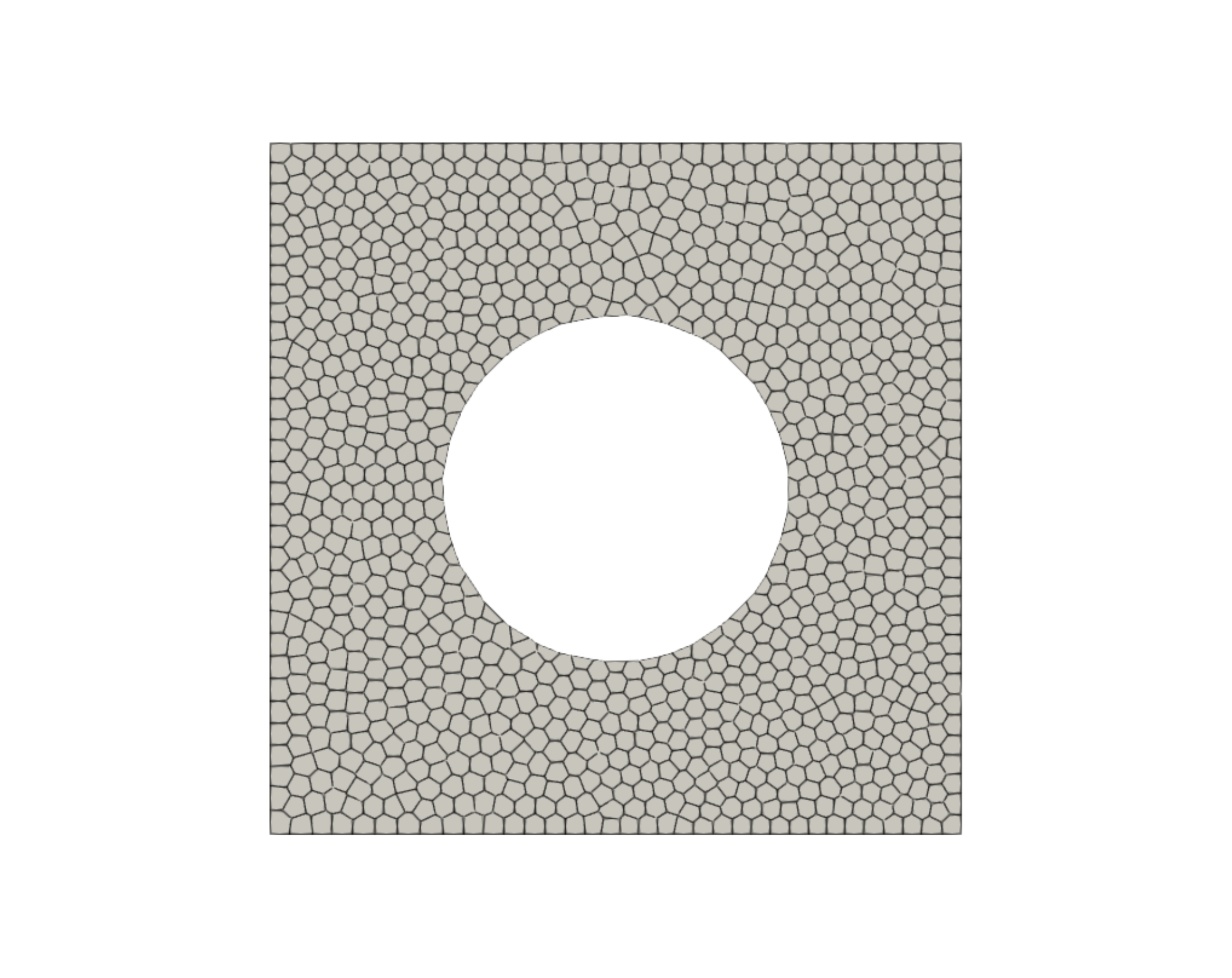}} 
    \subfigure[four-perforation]{\includegraphics[width=0.35\textwidth,trim={7.5cm 4.cm 7.5cm  4.cm},clip]{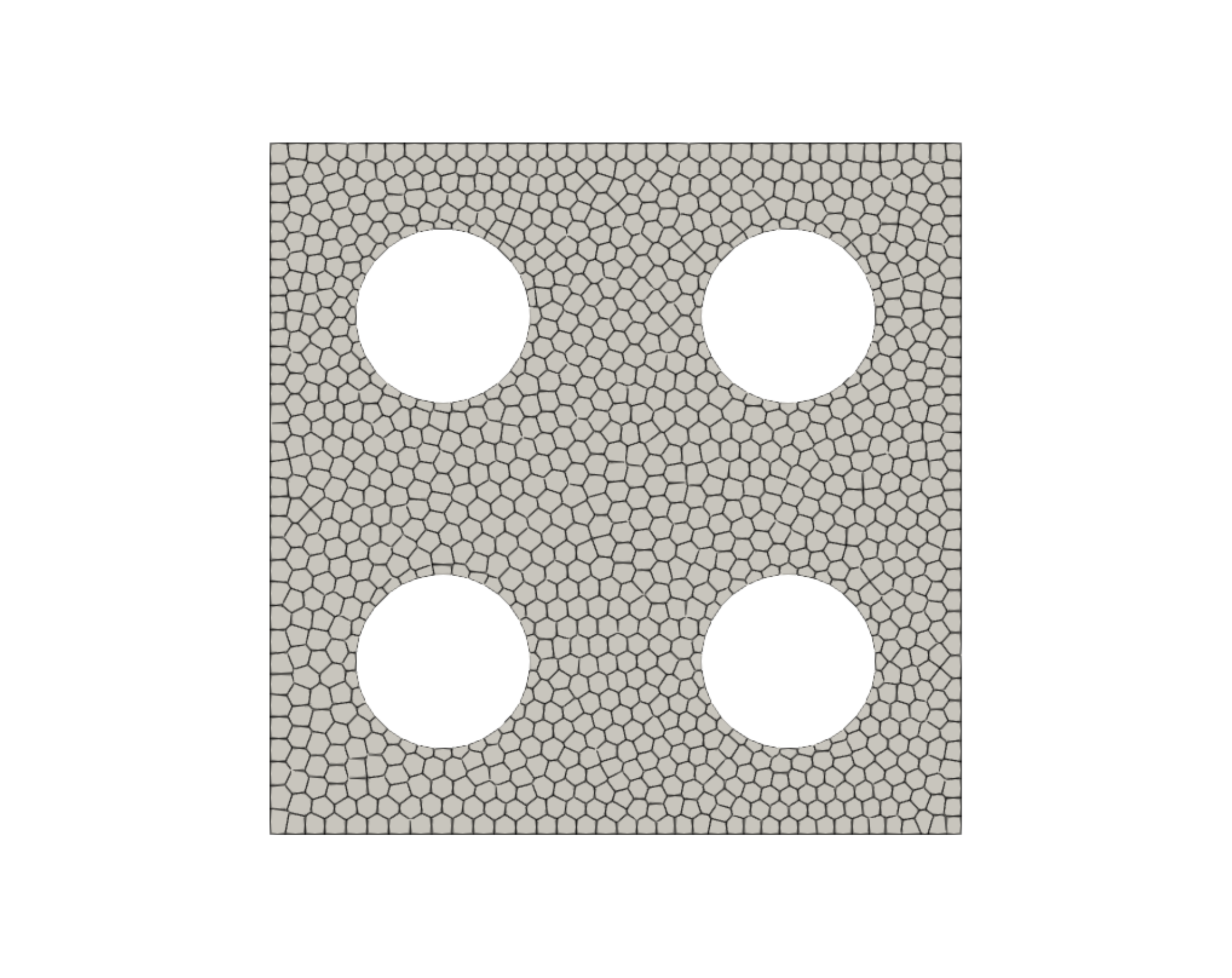}}  
    \caption{Experiment 3. Coarse domain discretizations used in the energy dissipation test, showing meshes with $1\text{,}056$ elements for the single-perforation (a) and $1\text{,}152$ elements for the four-perforation (b) geometries.}\label{fig:meshes2DPerfored}
\end{figure}

\subsection{Experiment 3: Energy dissipation and damping properties of perforated plates}
In this experiment, we test the transient behavior of perforated plates for linear viscoelastic (Kelvin--Voigt type) materials. Similar to \cite{meddahi23}, we set the density of the material as $1 \text{Kg}/\text{m}^3$ and the elastic and viscous Lam\'e coefficients as:
$$\mu_1 = 11.583 \text{Pa}, \quad \lambda_1=17.308 \text{Pa}, \quad \mu_2=13.423 \text{Pa s}, \quad \text{and} \quad \lambda_2=657.718 \text{Pa s}.$$

We consider two domain geometries on the unit square $\Omega = (0,1)^2$, as shown in Figure~\ref{fig:meshes2DPerfored}: a single-perforation domain containing a central circular hole of radius $0.25\text{m}$ discretised with $16\text{,}416$ elements and $64\text{,}956$ Degrees of Freedom, and a four-perforation domain featuring four circular holes of radius $0.0625\text{m}$ centered at $(0.25, 0.25)$, $(0.25, 0.75)$, $(0.75, 0.25)$, and $(0.75, 0.75)$ consisting of $16\text{,}512$ elements and $65\text{,}360$. The domain is clamped along the bottom boundary ($y = 0$) and free elsewhere. To model a time-dependent traction along the top boundary, we apply a downward load vector given by:
$$\bf(x,y;t)=\begin{cases}
    (0,-5\sin(\frac{\pi t}{5}))^{\tt t}, & y>0.9 \text{ and } t\leq1, \\
    (0,0)^{\tt t}, & \text{otherwise}.
\end{cases}$$
The computational meshes consist of 16,416 and 16,512 Voronoi elements for the single- and four-perforation domains, respectively. The total physical time is set to $T = 15\text{s}$ with a uniform time step of $\Delta t = 0.1\text{s}$.

\begin{figure}[!h]
    \centering
    \subfigure[$t= 0.5 \text{s}$]{%
        \begin{minipage}{0.22\textwidth}
            \centering
            \includegraphics[width=\textwidth,trim={6.5cm 3.5cm 6.75cm 3.25cm},clip]{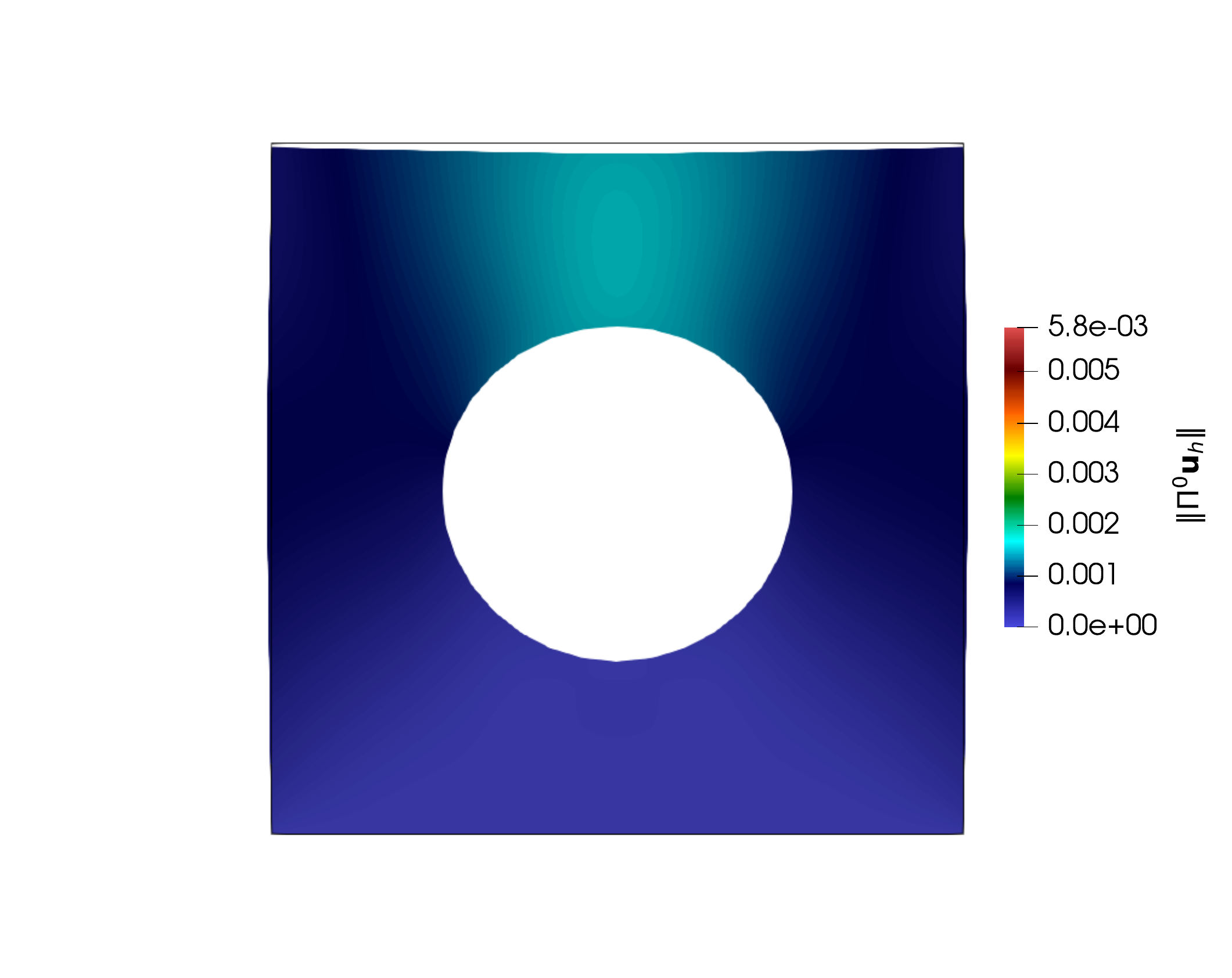}\\[0.5mm]
            \includegraphics[width=\textwidth,trim={6.5cm 3.5cm 6.75cm 3.25cm},clip]{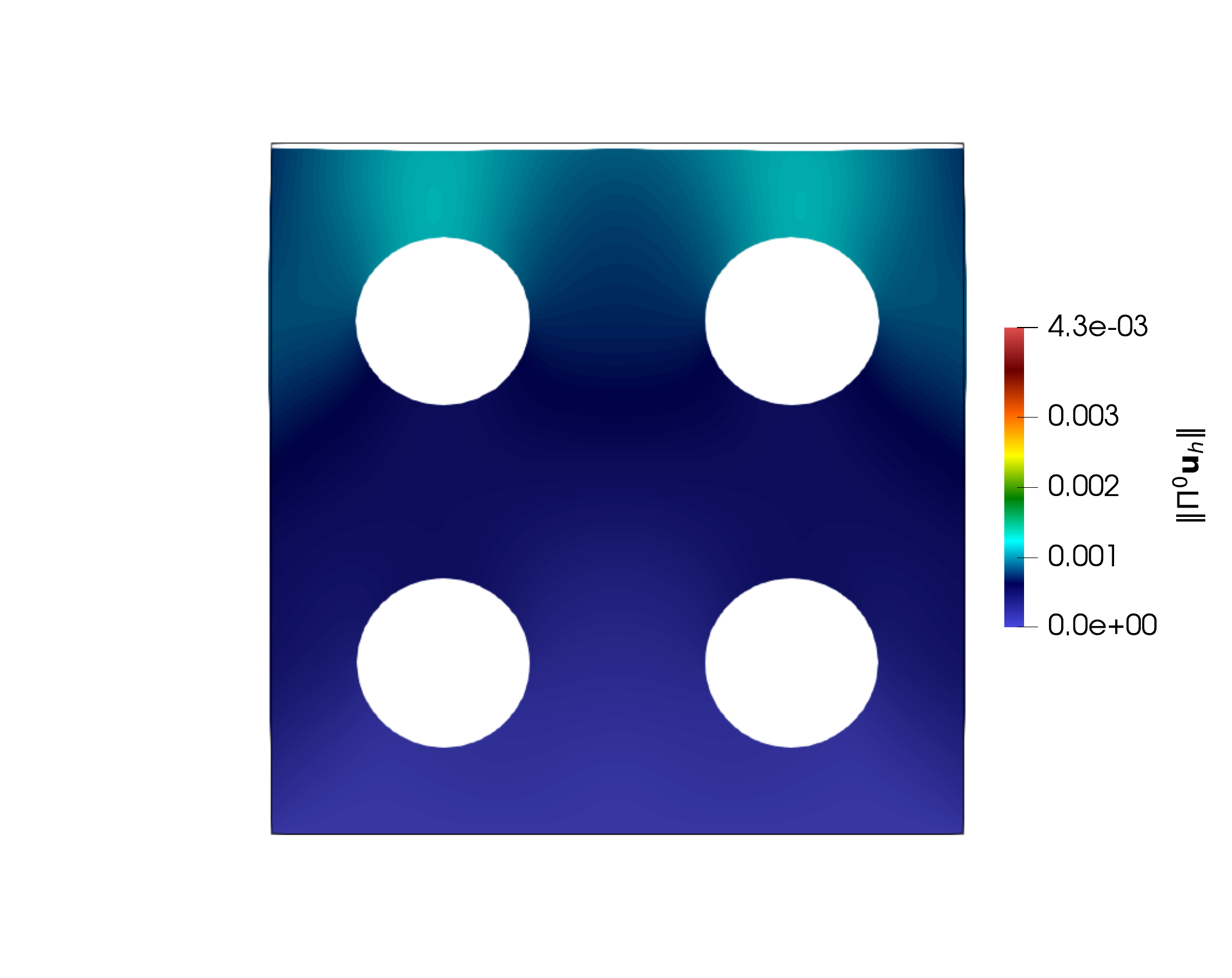}
        \end{minipage}%
    }\hfill
    \subfigure[$t= 1 \text{s}$]{%
        \begin{minipage}{0.22\textwidth}
            \centering
            \includegraphics[width=\textwidth,trim={6.5cm 3.5cm 6.75cm 3.25cm},clip]{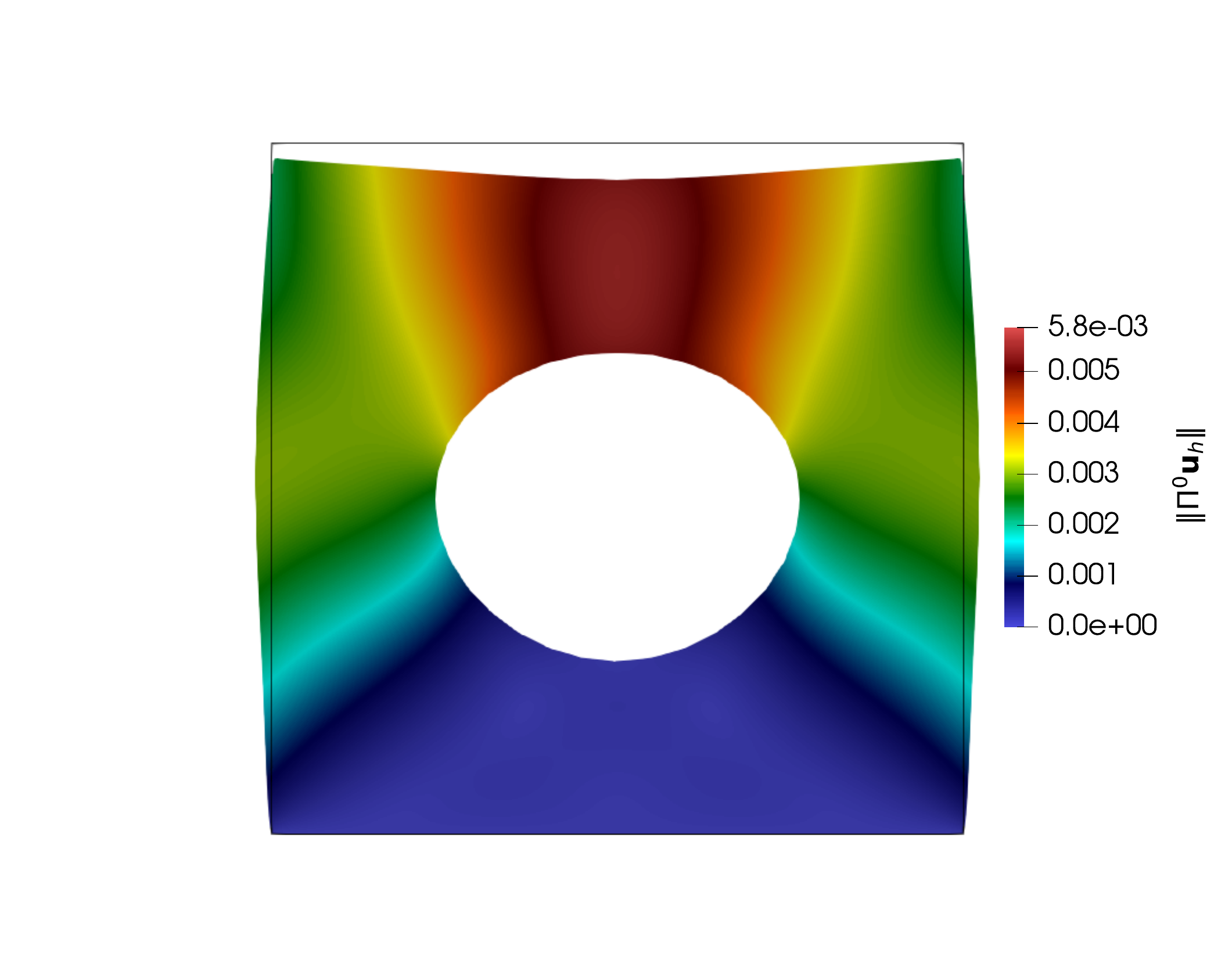}\\[0.5mm]
            \includegraphics[width=\textwidth,trim={6.5cm 3.5cm 6.75cm 3.25cm},clip]{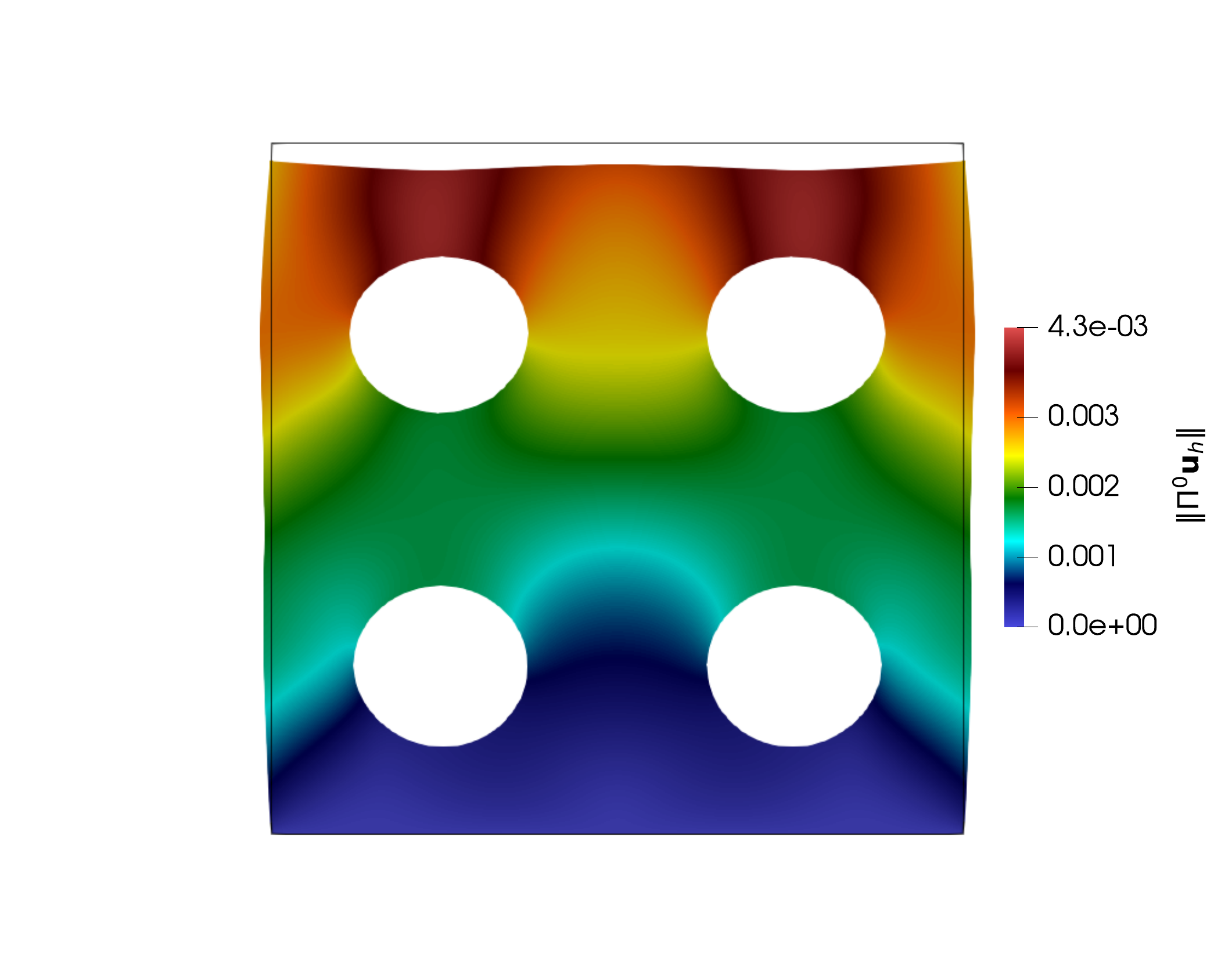}
        \end{minipage}%
    }\hfill
    \subfigure[$t= 1.5 \text{s}$]{%
        \begin{minipage}{0.285\textwidth}
            \centering
            \includegraphics[width=\textwidth,trim={6.5cm 3.5cm 0.25cm 3.25cm},clip]{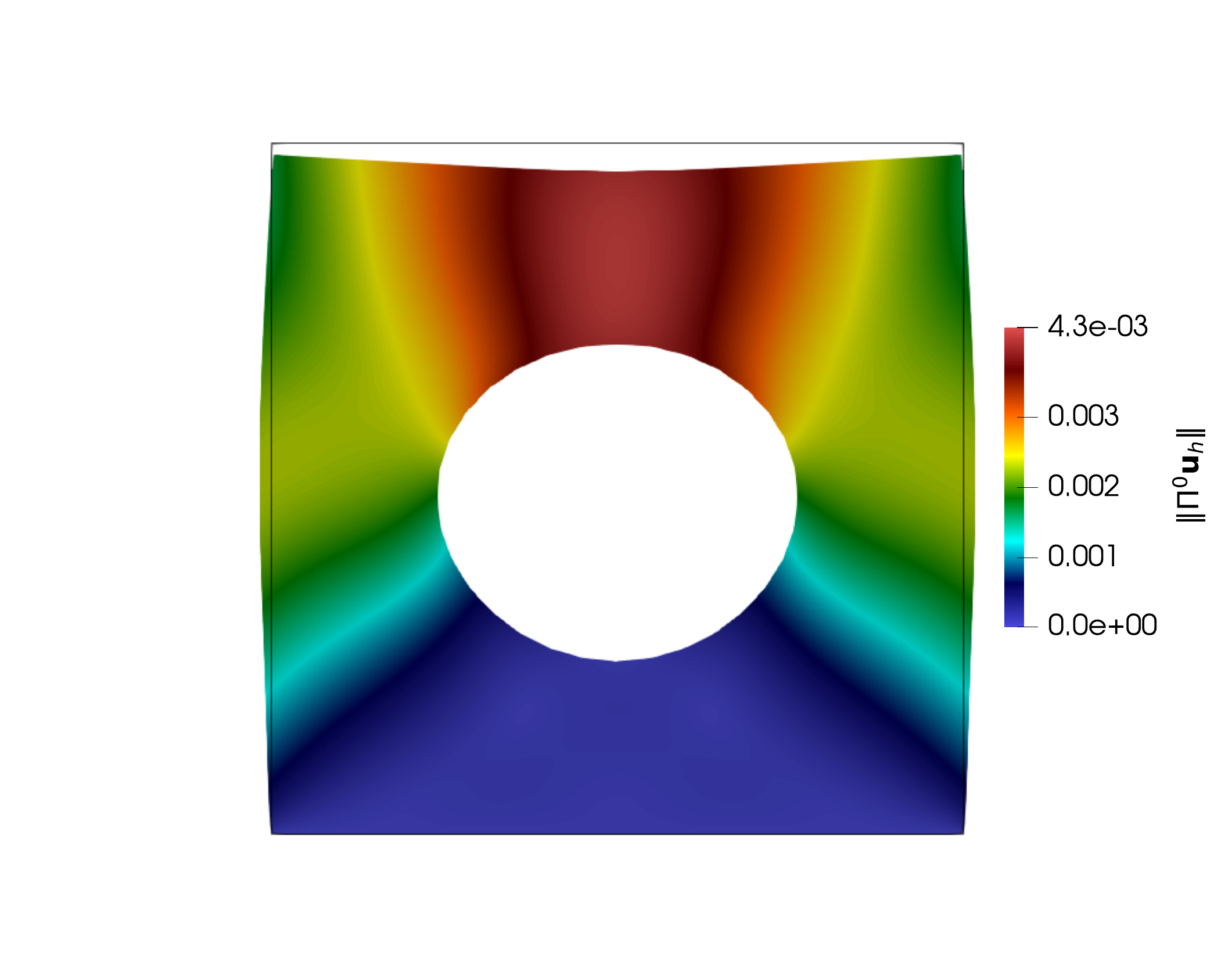}\\[0.5mm]
            \includegraphics[width=\textwidth,trim={6.5cm 3.5cm 0.25cm 3.25cm},clip]{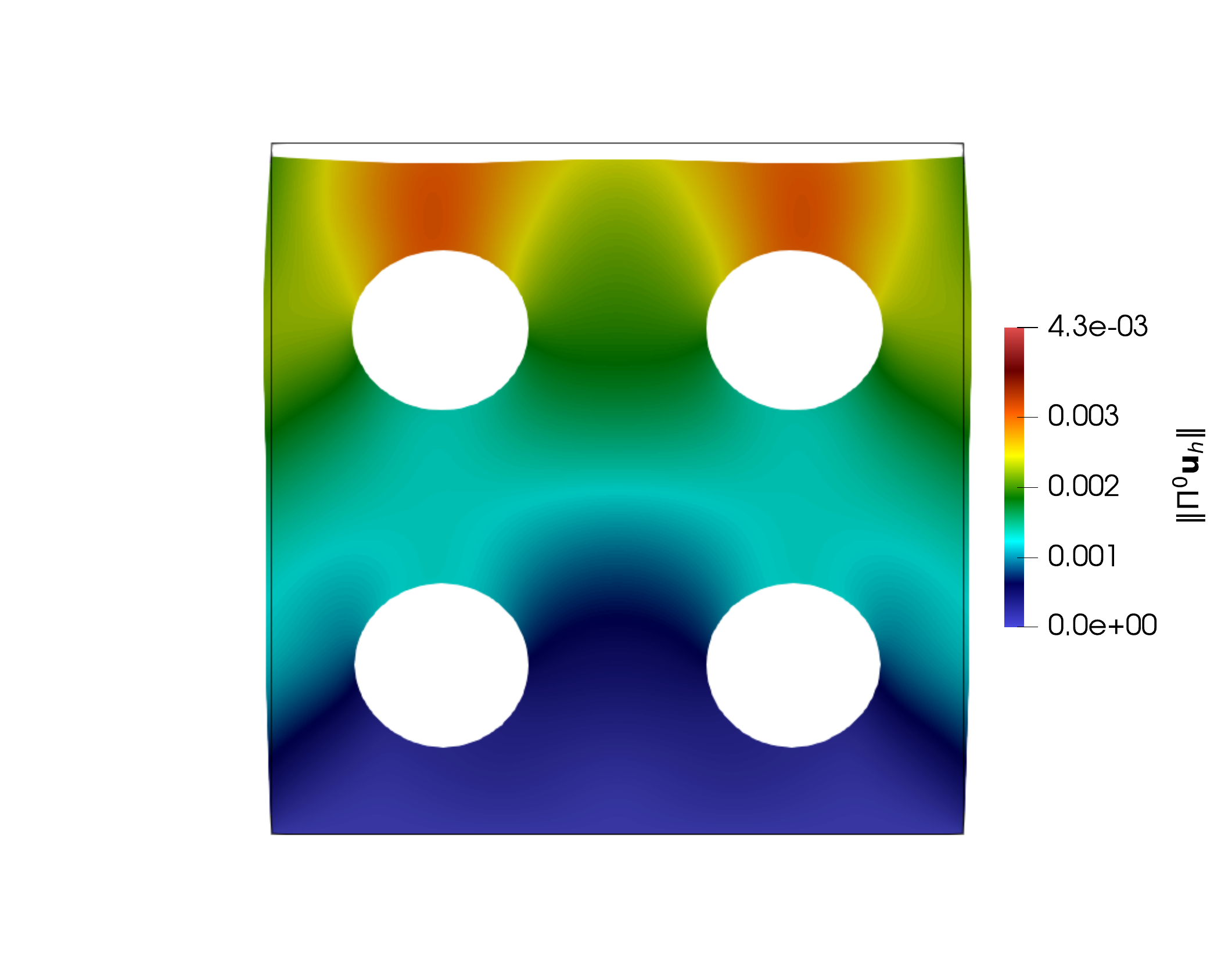}
        \end{minipage}%
    }\hfill
    \subfigure[Solution on $\mathbf{p}=(0.5,0.8)$]{%
        \begin{minipage}{0.2475\textwidth}
            \centering
            \includegraphics[width=\textwidth,trim={0.cm 0.cm 0.cm 0.cm},clip]{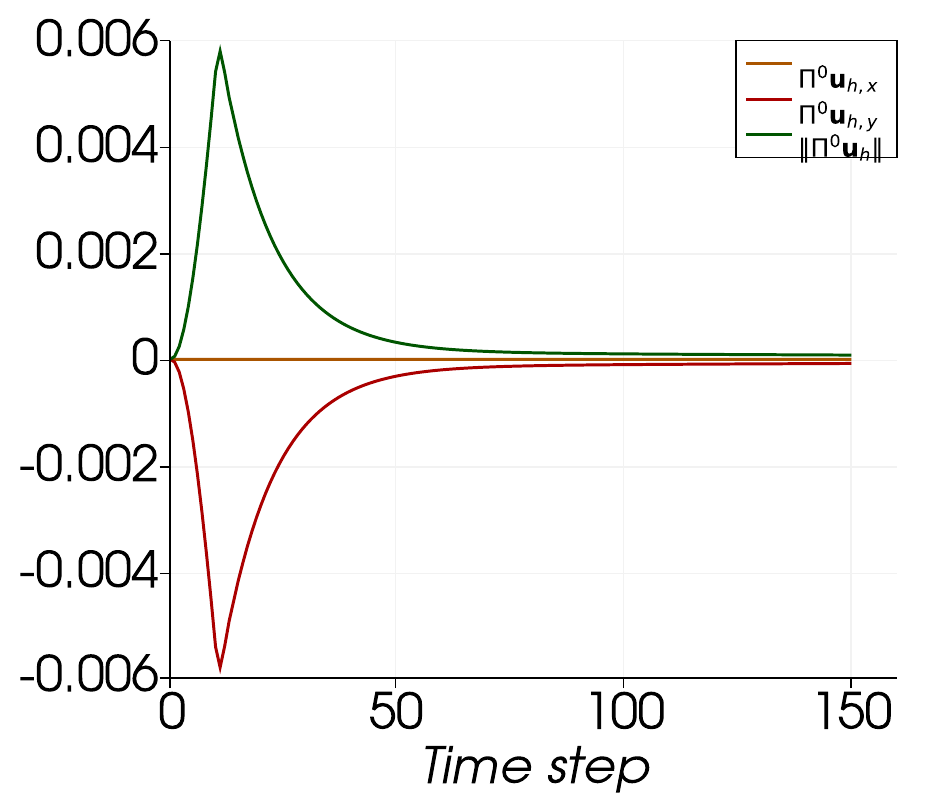}\\[1mm]
            \includegraphics[width=\textwidth,trim={0.cm 0.cm 0.cm 0.cm},clip]{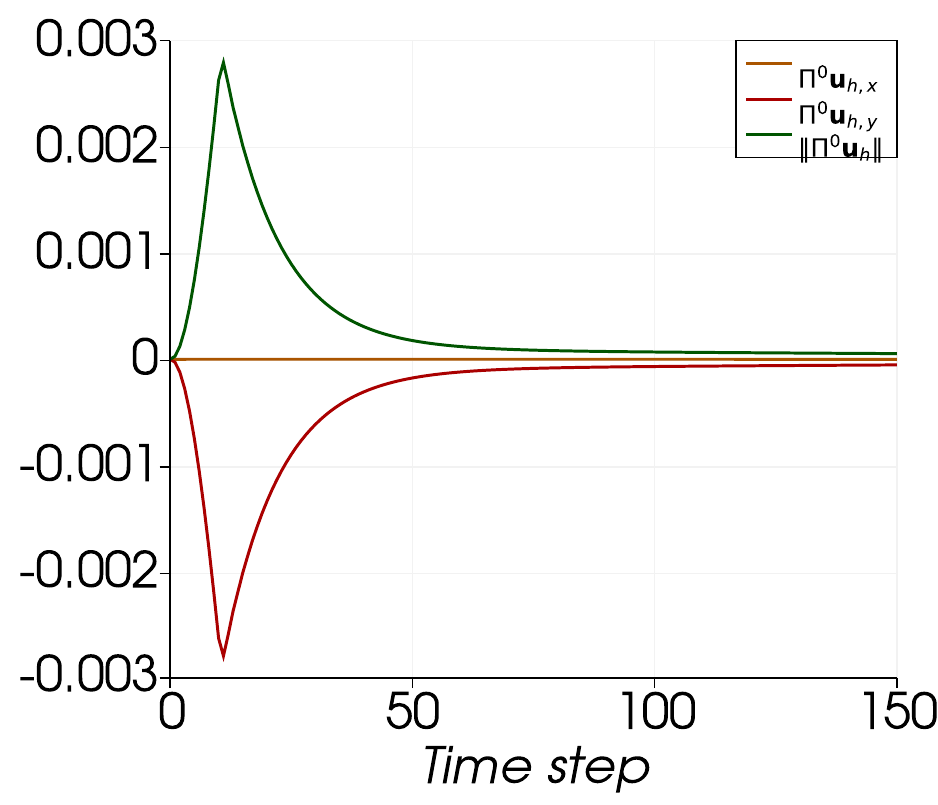}
        \end{minipage}%
    }
    \caption{Experiment 3. Snapshots of the polynomial projection of displacement $\Pi^0 \bu_h$ (magnified by $10$) taken at $t \in \{0.5\text{s}, 1.0\text{s}, 1.5\text{s}\}$ (a–c) for the single-perforation domain (top row) and four-perforation (bottom row), obtained with the linear-order method ($k=1$). Panel (d) shows the time history of the horizontal ($\Pi^0 \bu_{h,x}$), vertical ($\Pi^0 \bu_{h,y}$), and magnitude ($\|\Pi^0 \bu_h\|$) displacement components at point $\mathbf{p} = (0.5, 0.8)$.}
    \label{fig:snapshotsDissipation}
\end{figure}

Figure~\ref{fig:snapshotsDissipation} displays the numerical solutions computed with the fully discrete, linear-order Virtual Element scheme ($k=1$). Three snapshots of the polynomial projection of displacement $\Pi^0\bu_h$ (defined such that $\Pi^0\bu_h|_K = \Pi^{0,K}\bu_h$) are shown on the deformed configuration (magnified by $10\times$) at time instants $t \in \{0.5\text{s}, 1.0\text{s}, 1.5\text{s}\} $. The deformation gradient develops while the load is active ($t \le 1\text{s}$); once the load is removed, the domain gradually recovers  its undeformed state, demonstrating the characteristic energy dissipation and damping of Kelvin--Voigt materials. Moreover, we illustrate these properties by plotting the time history of the horizontal ($\Pi^0 \bu_{h,x}$), vertical ($\Pi^0 \bu_{h,y}$), and magnitude ($\|\Pi^0 \bu_h\|$) displacement components at the point $\mathbf{p} = (0.5, 0.8)$ in the right panel.

\subsection{Experiment 4: Sustained loading of viscoelastic slabs}

To evaluate the response of viscoelastic slabs under sustained constant loading, we consider the domain $\Omega = (0,1)^2 \times (0,2)$ and adopt the material parameters from \cite{rognes10}:
$$\rho = 10^{-3},\quad \mu_1 = 20 \text{Pa}, \quad \lambda_1=100 \text{Pa}, \quad \mu_2=80 \text{Pa s}, \quad \text{and}\quad  \lambda_2=100 \text{Pa s}.$$
The total duration of the creep test is $5\text{s}$. A constant traction along the $z$-direction is applied for the first $3\text{s}$, after which the slab is released. This loading condition is given by the vector field
$$\bf(x,y,z;t)=\begin{cases}
    (0,0,1)^{\tt t}, & z>0.9 \text{ and } t\leq3, \\
    (0,0,0)^{\tt t}, & \text{otherwise}.
\end{cases}$$

The computational domain is discretized starting from a $400$-element Voronoi mesh of the unit square, which is then extruded along the $z$-direction across $20$ layers using the native extrusion capabilities of \texttt{vem++}. This yields a polytopal mesh consisting of $8\text{,}000$ elements and $50\text{,}526$ degrees of freedom.

\begin{figure}[!h]
    \centering
    \subfigure[$t= 0.67 \text{s}$]{%
        \begin{minipage}{0.29\textwidth}
            \centering
            \includegraphics[width=\textwidth,trim={4.25cm 1.cm 8.75cm 5.cm},clip]{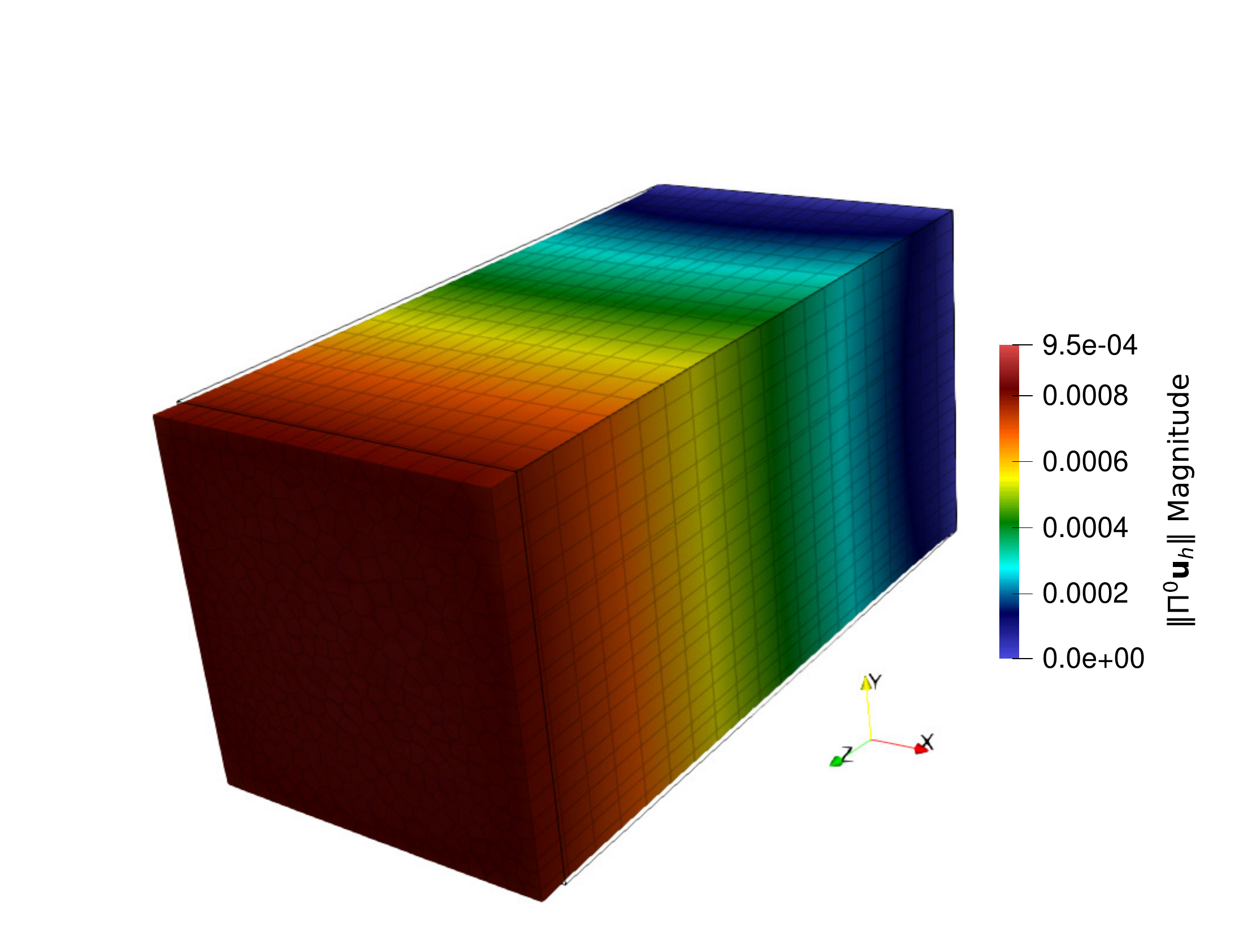}\\[0.5mm]
            \includegraphics[width=\textwidth,trim={4.25cm 1.cm 8.75cm 5.cm},clip]{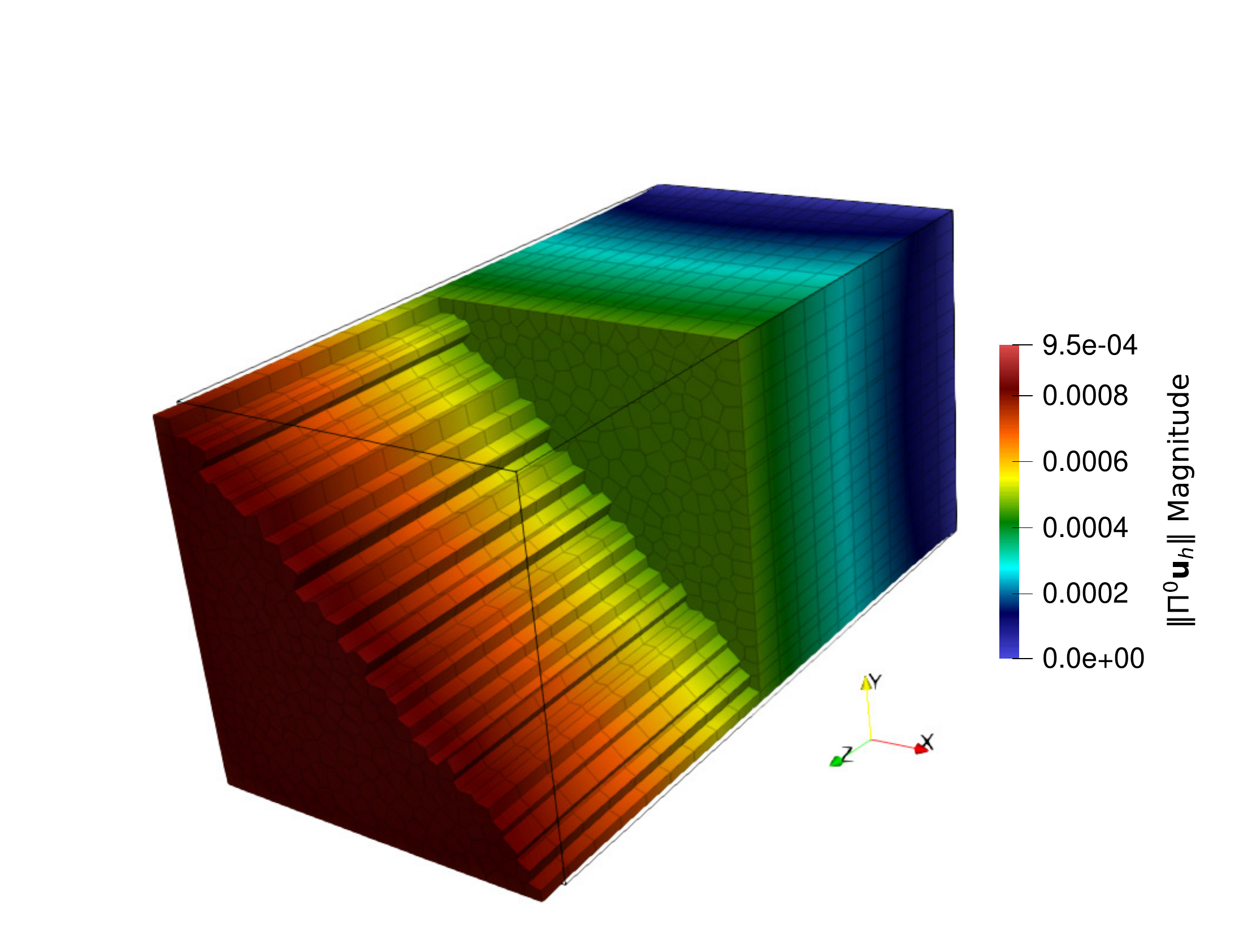}
        \end{minipage}%
    }\hfill
    \subfigure[$t= 2.5 \text{s}$]{%
        \begin{minipage}{0.29\textwidth}
            \centering
            \includegraphics[width=\textwidth,trim={4.25cm 1.cm 8.75cm 5.cm},clip]{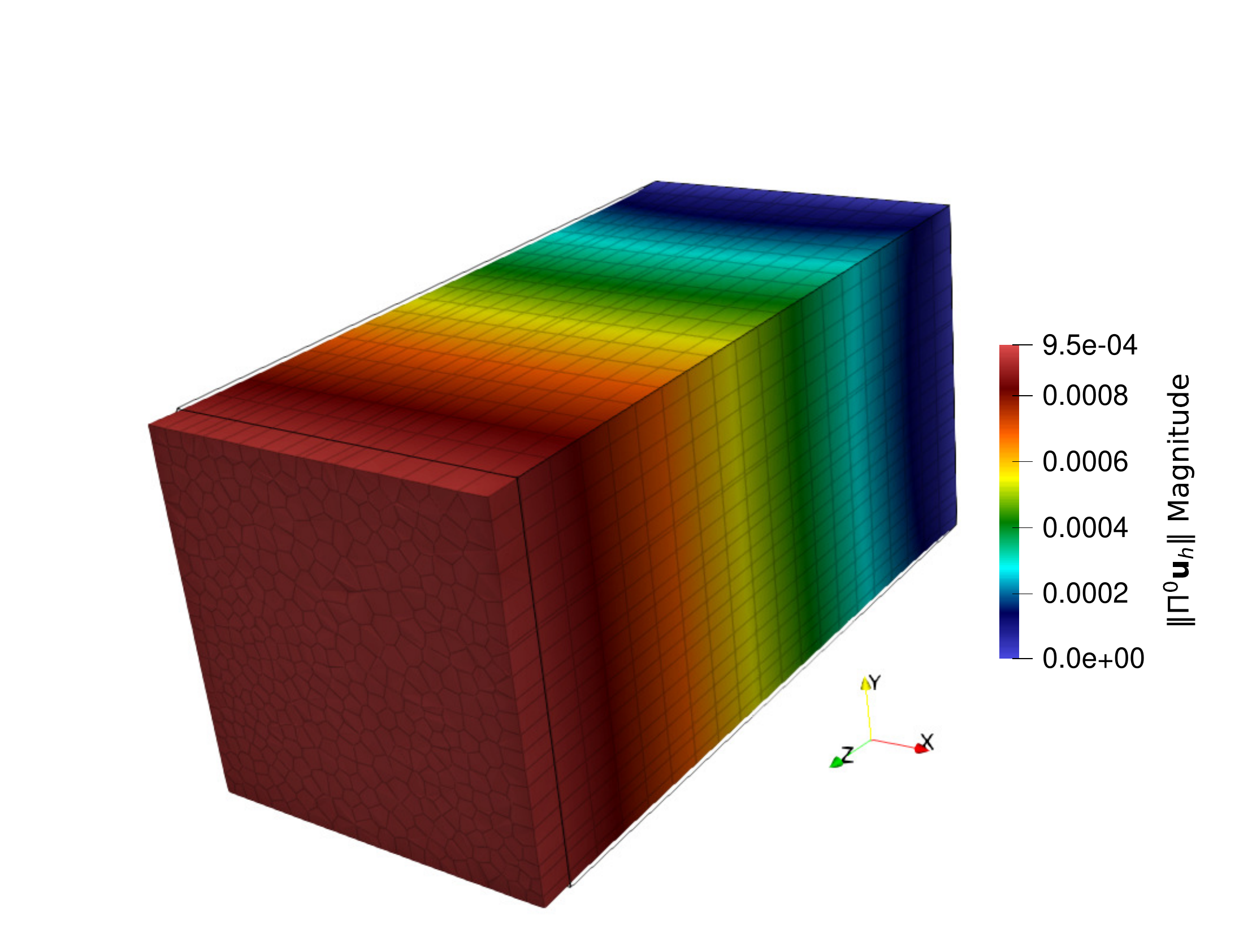}\\[0.5mm]
            \includegraphics[width=\textwidth,trim={4.25cm 1.cm 8.75cm 5.cm},clip]{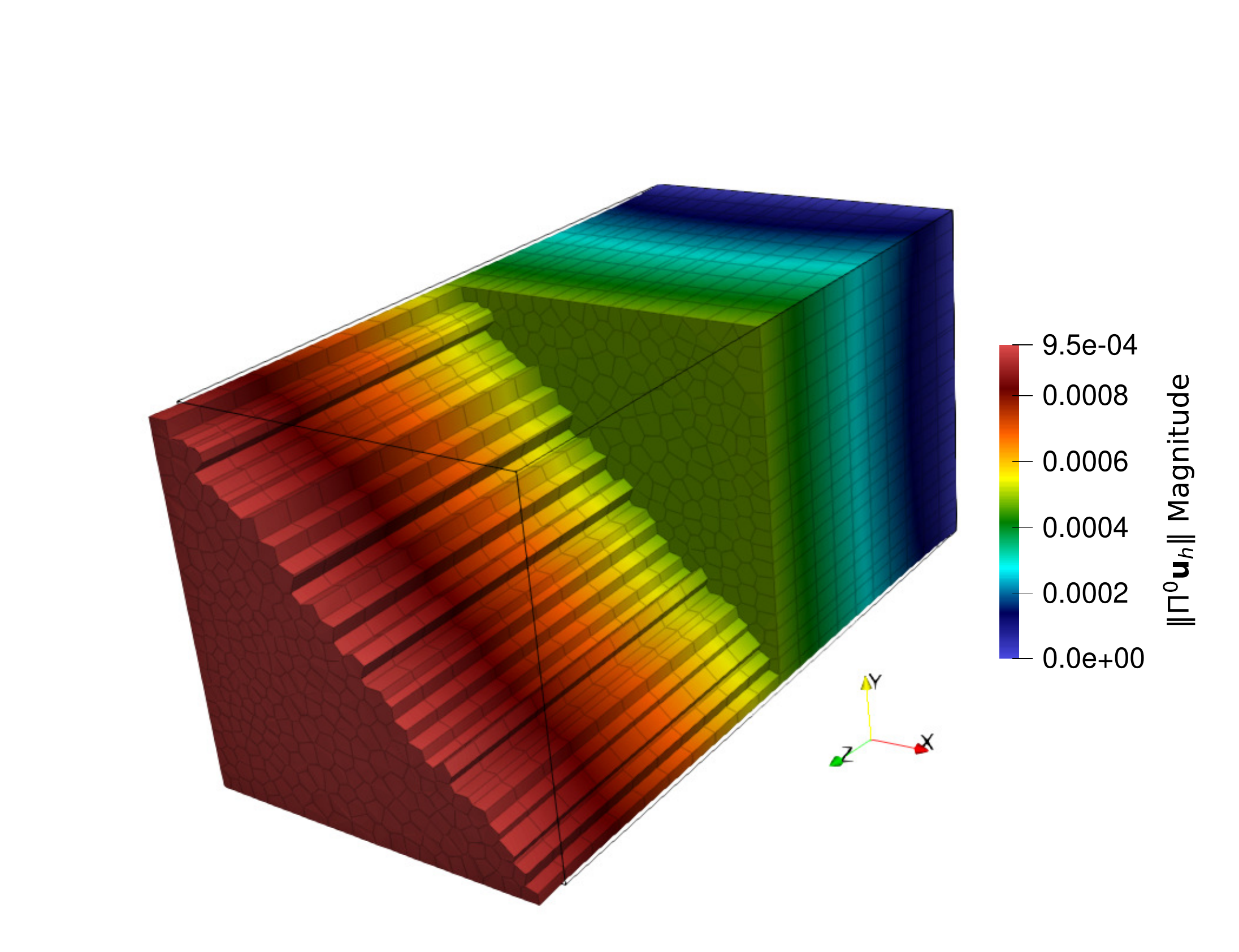}
        \end{minipage}%
    }\hfill
    \subfigure[$t= 3.3 \text{s}$]{%
        \begin{minipage}{0.3875\textwidth}
            \centering
            \includegraphics[width=\textwidth,trim={4.25cm 1.cm 1.25cm 5.cm},clip]{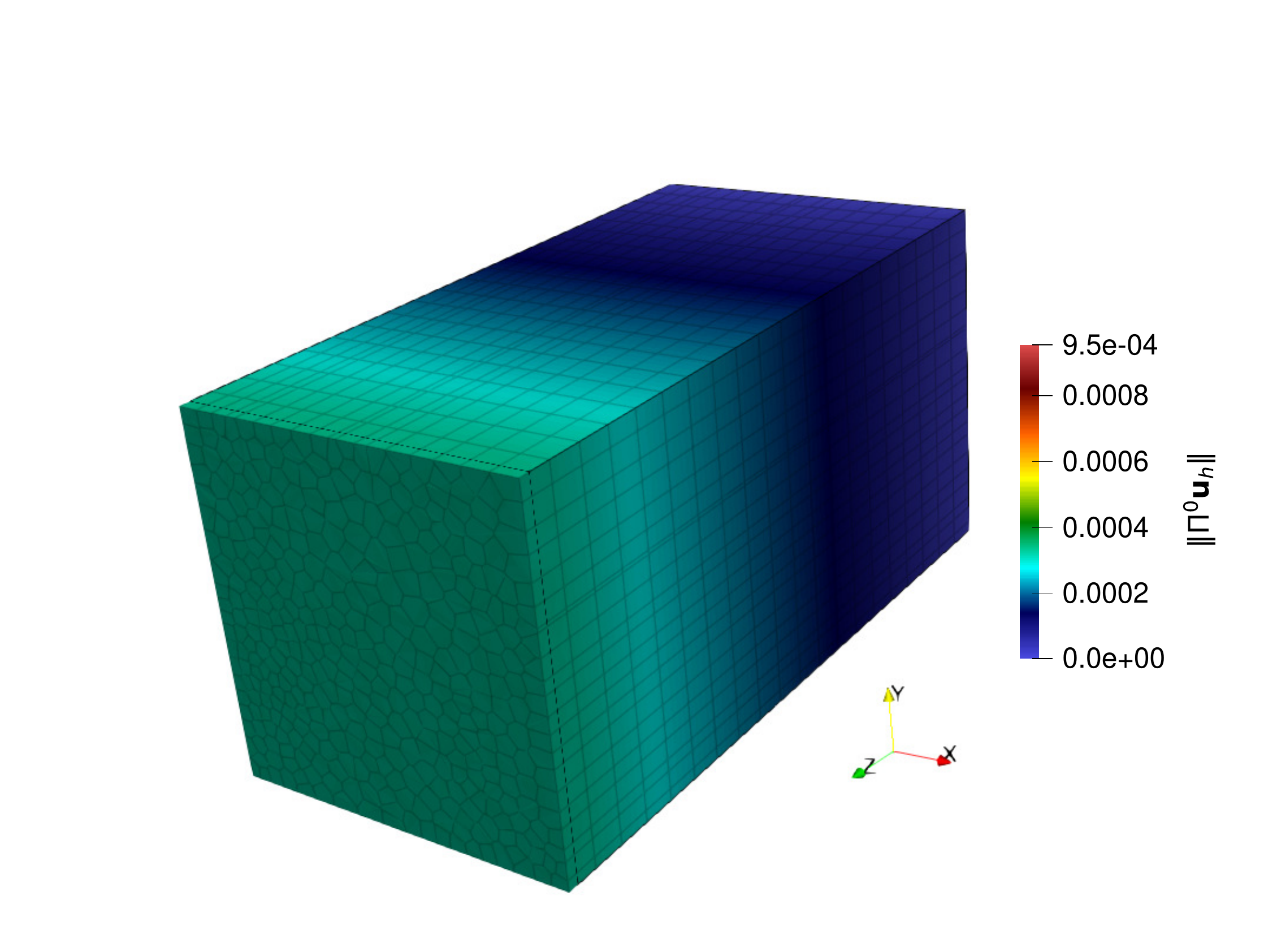}\\[0.5mm]
            \includegraphics[width=\textwidth,trim={4.25cm 1.cm 1.25cm 5.cm},clip]{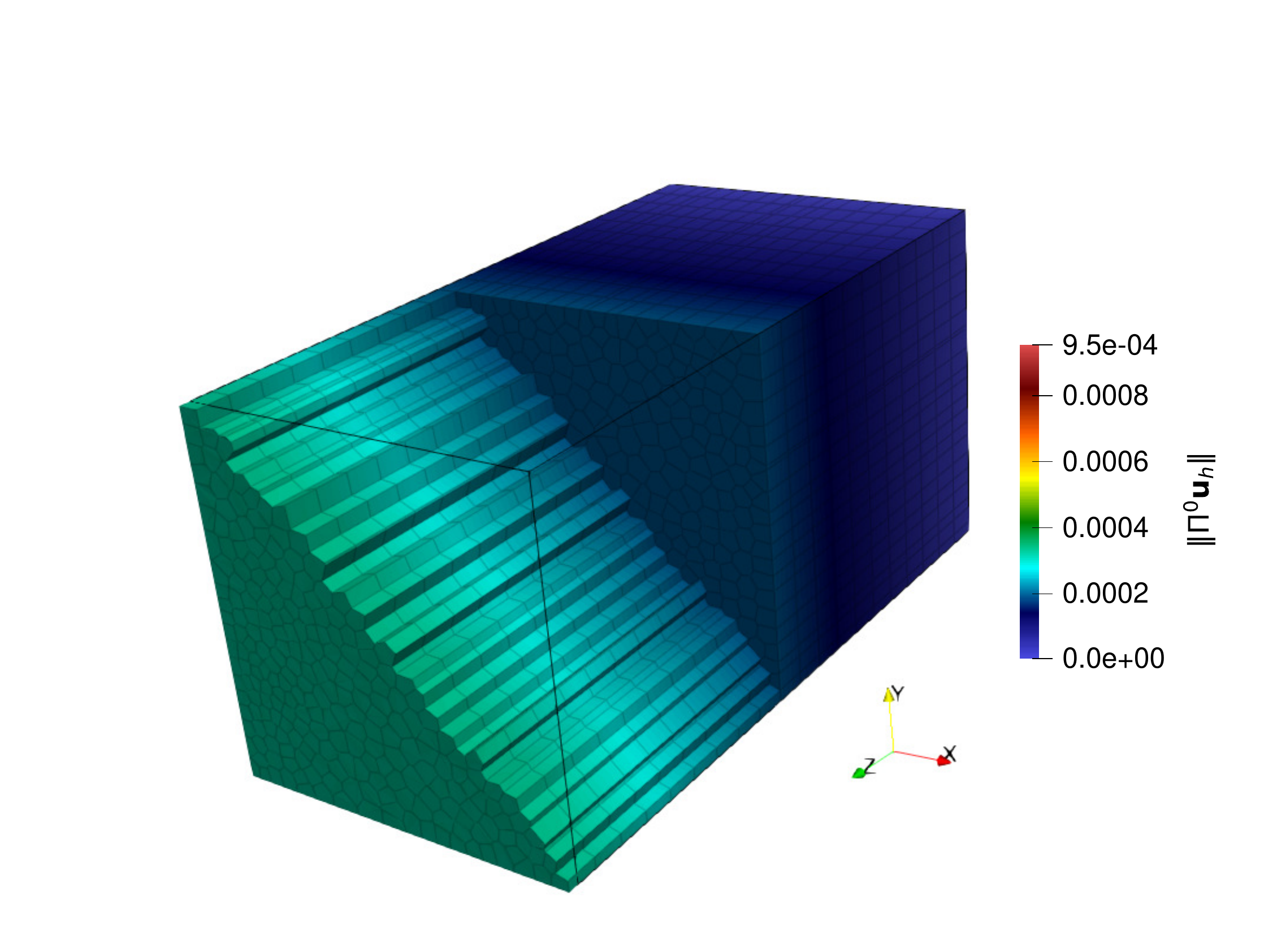}
        \end{minipage}%
    }\hfill
    \caption{Experiment 4. Snapshots of the polynomial projection of displacement $\Pi^0 \bu_h$ (magnified by $100$) taken at $t \in \{0.67\text{s}, 2.5\text{s}, 3.3\text{s}\}$ (a–c) for the viscoelastic slab domain (top row) and visualisation of the inner elements (bottom row), obtained with the linear-order method ($k=1$).}
    \label{fig:snapshotsViscoElasticSlab}
\end{figure}

Figure~\ref{fig:snapshotsViscoElasticSlab} presents snapshots of the sustained loading simulation at selected time instances ($t \in \{0.67\text{s}, 2.5\text{s}, 3.3\text{s}\}$). During the first three seconds, the viscoelastic slab undergoes creep deformation along the $z$-direction, subsequently relaxing back toward its reference configuration once the load is removed. This response closely matches the behavior reported in \cite{rognes10}. Furthermore, Figure~\ref{fig:avgViscoElasticSlab} plots the average spatial components ($x$, $y$, and $z$) and the norm of the projected displacement $\Pi^0\mathbf{u}_h$. As expected, the displacement grows toward the elastic limit under active loading and monotonically decays following release.

\begin{figure}[h!]
    \centering
    \includegraphics[width=0.4\linewidth]{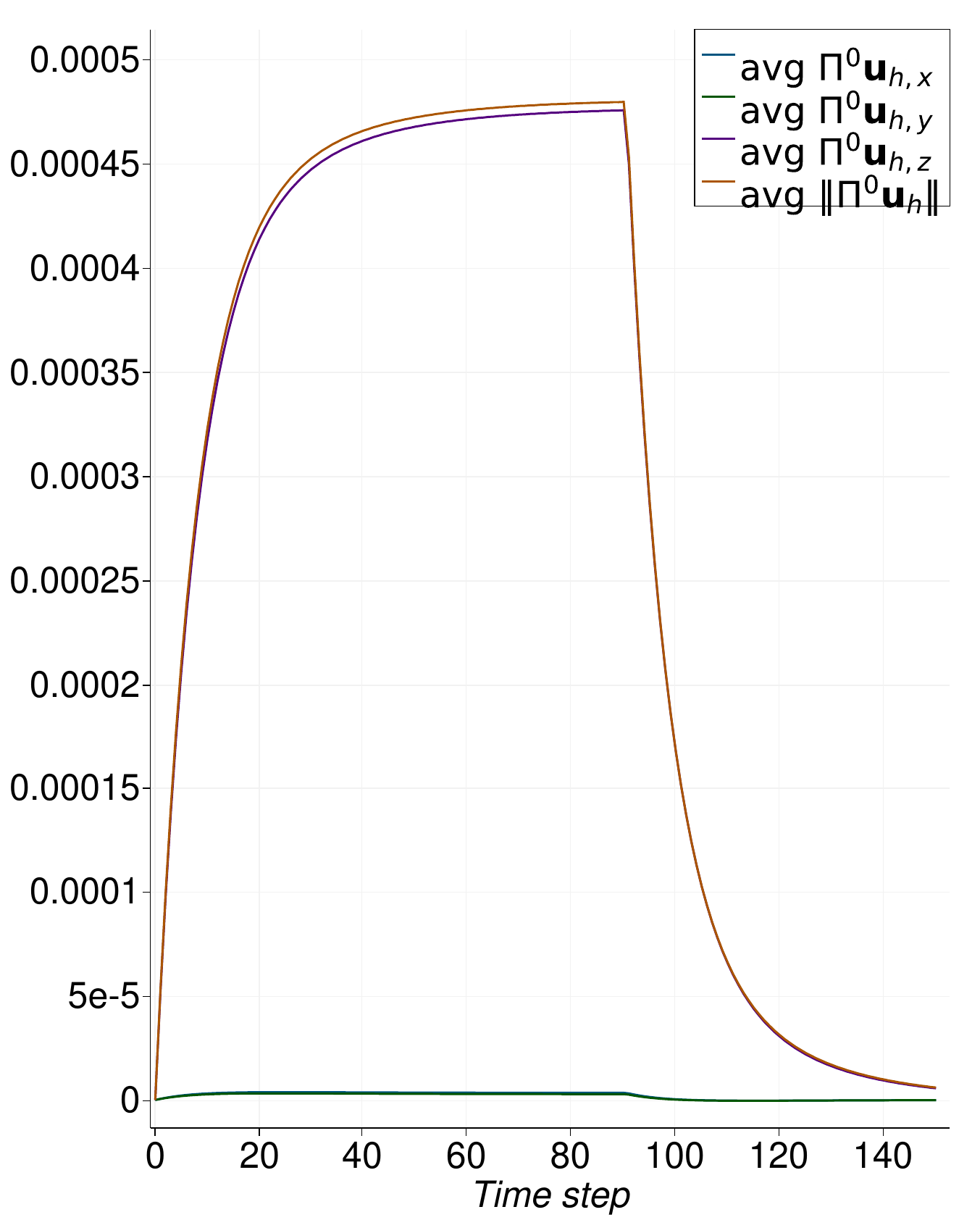}
    \caption{Experiment 4. Behavior of average (avg) value for the components $x, y, z$, and the norm of the polynomial projection of displacement $\Pi^0 \bu_h$ for the viscoelastic slab sustained loading test.}
    \label{fig:avgViscoElasticSlab}
\end{figure}

\subsection{Experiment 5. Linear Kelvin--Voigt modelling of the viscoelastic response in aneurysmal arterial tissue}
To evaluate the applicability and numerical stability of the proposed fully-discrete primal scheme for viscoelastic media, we consider a realistic biomechanical benchmark; the dynamic response of a human aortic aneurysm under physiological pulsatile loading. The objective of this test case is to demonstrate the method's capability to resolve complex 3D geometries (discretized here using $27,004$ an 8-node hexahedral cell with triangulated surface faces) and capture localized deformation fields. 
 
To model this behavior, the physical parameters of the arterial wall are chosen based on adapted physiological tissue data. The tissue density of the arterial wall is set to $\rho = 1.05 \times 10^{-9} \text{ tonne/mm}^3$. Based on experimental measurements by \cite{lt96}, the Young's modulus of human arterial walls under physiological pressures ($50\text{--}150 \text{mmHg}$) ranges across age groups, leading to our choice of $E = 0.2 \text{MPa}$ for the stiffened aneurysmal tissue. To model nearly incompressible tissue behavior, Poisson's ratio is set to $\nu = 0.45$. From these base parameters, the elastic shear modulus and Lamé's first parameter are calculated as $\mu_1 = E / [2(1+\nu)]$ and $\lambda_1 = E\nu / [(1+\nu)(1-2\nu)]$, respectively. Following Stokes' hypothesis for nearly incompressible media in small-strain regimes, the bulk viscosity is set to $\lambda_2 = 0 \text{MPa}\text{ s}$. The shear viscosity is defined as $\mu_2 = E b_1 = 0.02 \text{MPa}\text{ s}$, where $b_1 = 0.01 \text{s}$ denotes the characteristic relaxation time, in agreement with \cite{vj11}.

The boundary conditions are prescribed as follows: fixed constraints ($\bu = \bzero$) are applied at the arterial caps to simulate continuity with surrounding tissue, and a traction-free condition ($\bsigma\bn = \bzero$) is imposed on the outer adventitial surface. On the inner luminal surface, a compressive normal traction $\bsigma\bn = p(t)\bn$ is applied to model pulsatile blood pressure, where $p(t) = -0.01333 + 0.00267 \cos(2\pi t / T) \text{MPa}$ and $T = 0.8 \text{s}$ represents the cardiac cycle period. This loading function models simple harmonic pressure fluctuations relative to the mean arterial pressure over one full cardiac cycle.

\begin{figure}[!h]
    \centering
    \subfigure[$t= 0.2 \text{s}$]{%
        \begin{minipage}{0.425\textwidth}
            \centering
            \includegraphics[width=\textwidth,trim={9.cm 3.cm 12.cm 3.cm},clip]{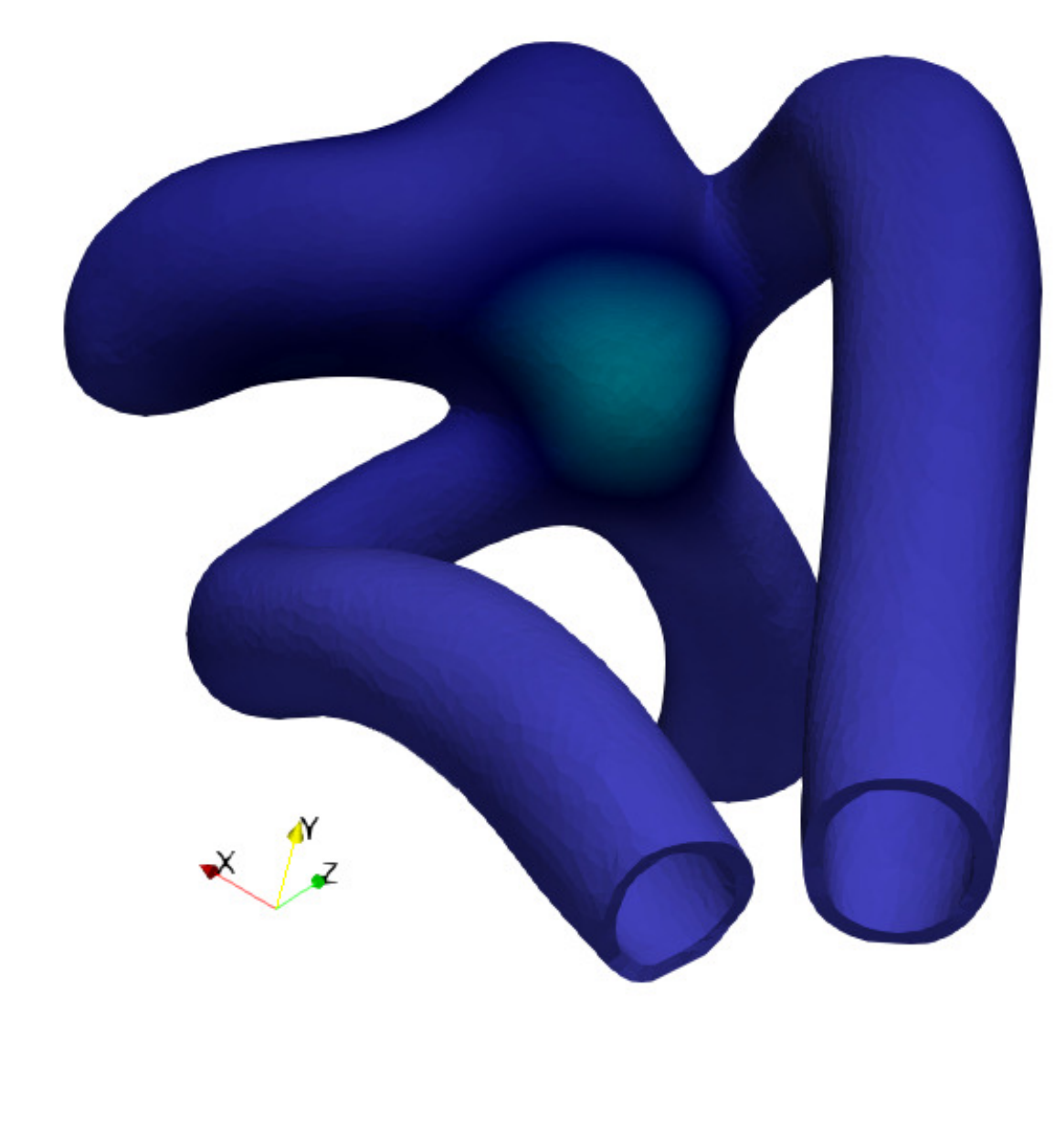}\\[0.5mm]
            \includegraphics[width=\textwidth,trim={9.cm 2.cm 11.cm 2.cm},clip]{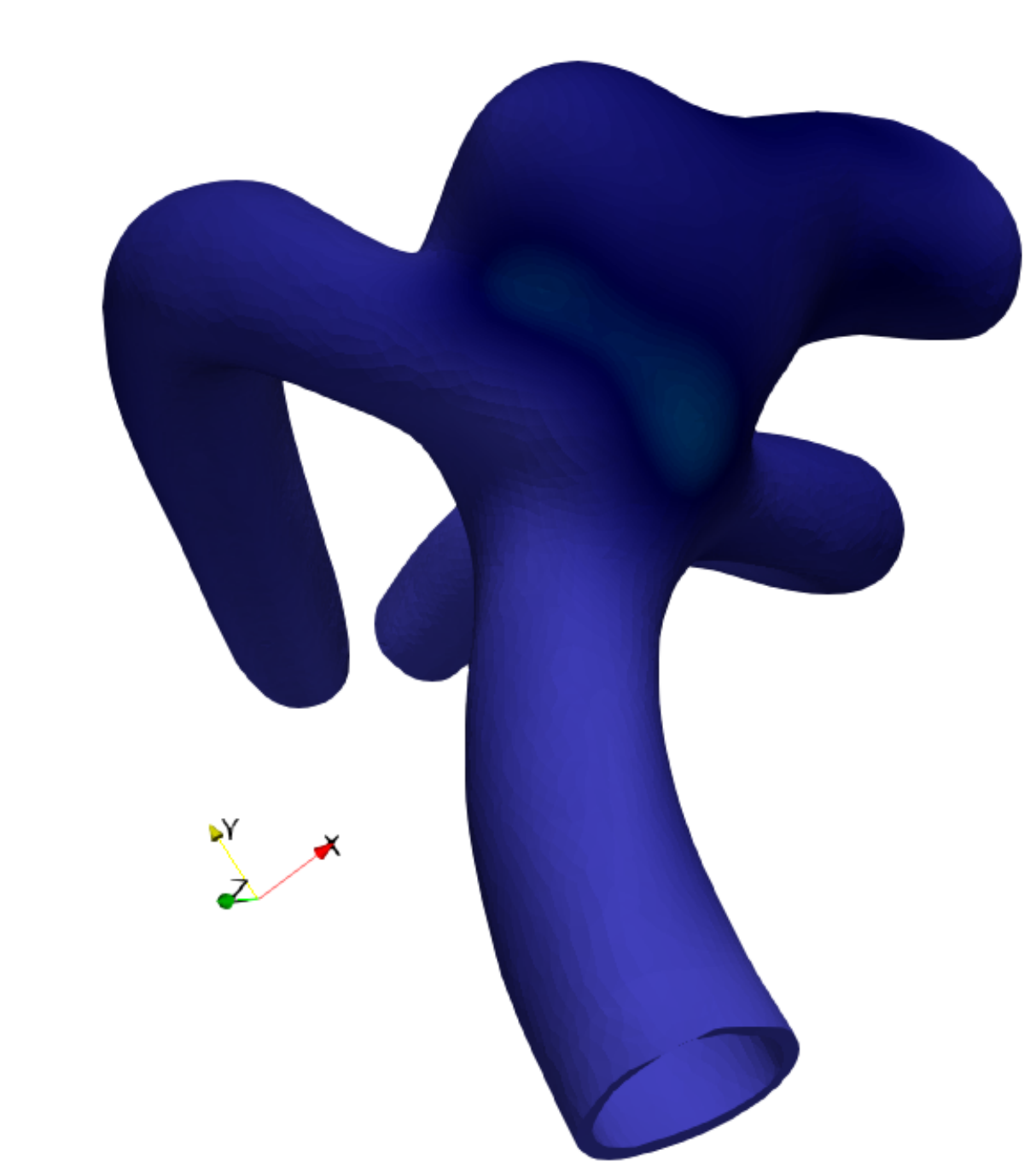}
        \end{minipage}%
    }\hfill
    \subfigure[$t= 0.4 \text{s}$]{%
        \begin{minipage}{0.5625\textwidth}
            \centering
            \includegraphics[width=\textwidth,trim={9.cm 3.cm 5.cm 3.cm},clip]{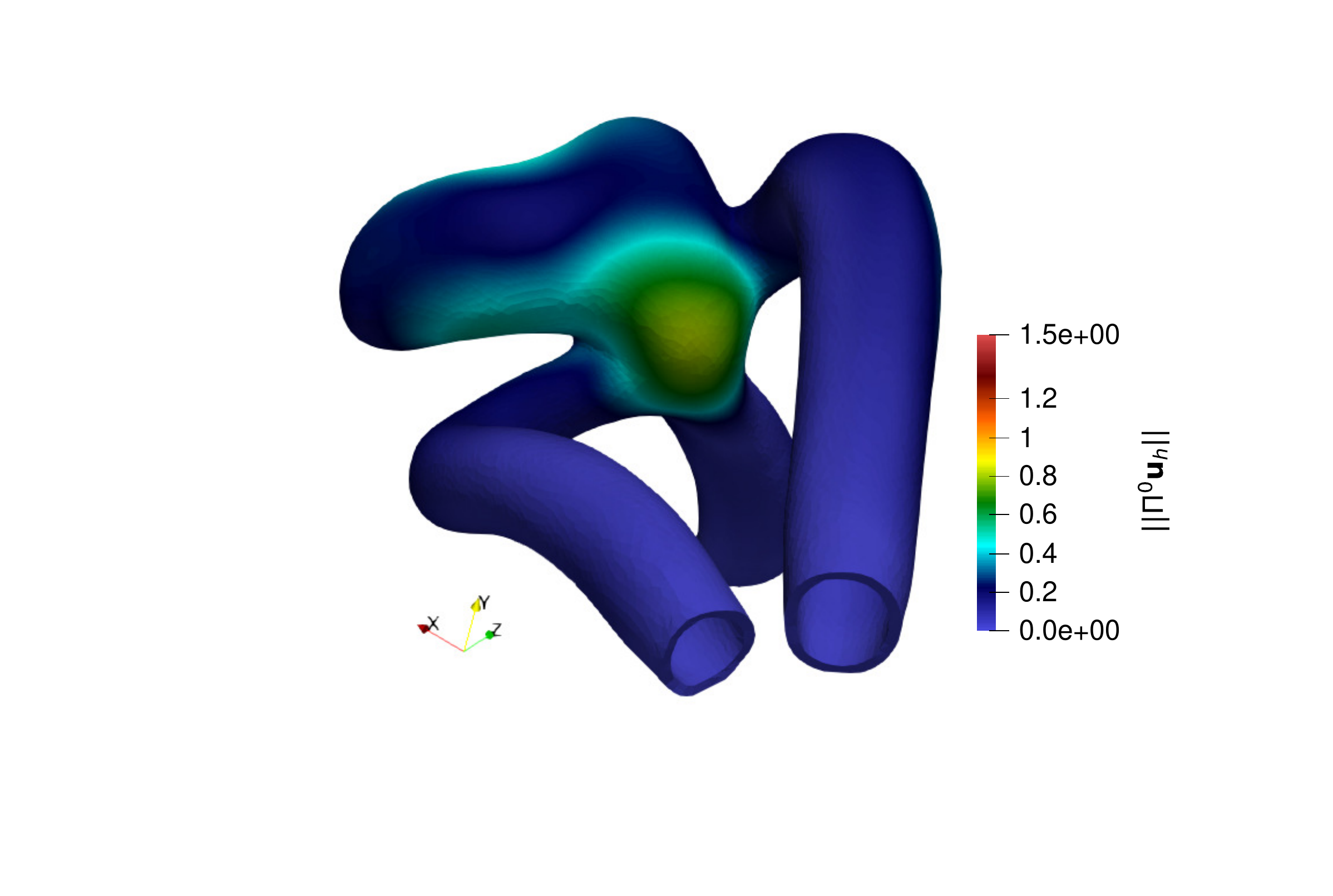}\\[0.5mm]
            \includegraphics[width=\textwidth,trim={9.cm 2.cm 5.cm 2.cm},clip]{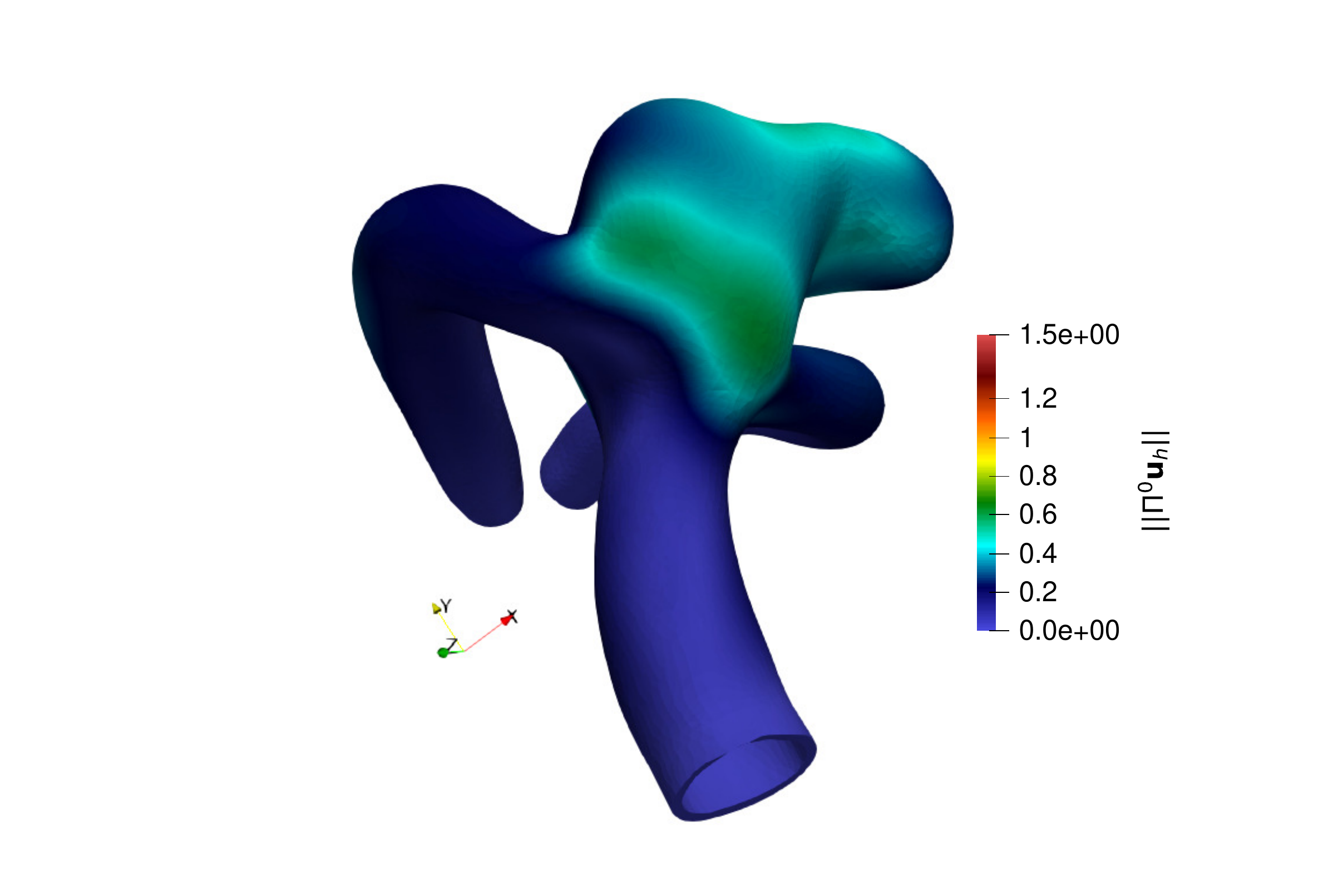}
        \end{minipage}%
    }\hfill
    \caption{Experiment 5. Snapshots of the polynomial projection of displacement $\Pi^0 \bu_h$ (magnified by $5$) taken at $t \in \{0.2\text{s}, 0.4\text{s}\}$ showing front (top row) and posterior (bottom row) views of the aneurysmal arterial tissue, obtained with the linear-order method ($k=1$).}
    \label{fig:snapshotsAneurysm1}
\end{figure}

Figures~\ref{fig:snapshotsAneurysm1} and~\ref{fig:snapshotsAneurysm2} display the simulation results obtained using the linear Kelvin--Voigt model with the lowest-order VEM discretization (cf. Section~\ref{sec:vem}). The color bar indicates the norm of the projected displacement $\Pi^0\mathbf{u}_h$ and the tissue is shown in its deformed configuration (magnified by $5$). Qualitative analysis shows that the highest displacement magnitudes are concentrated within the pathological tissue of the aneurysm, reaching a peak displacement magnitude of $1.5\text{ mm}$. As expected, localized tissue inflation corresponds directly to the cardiac cycle, as demonstrated by the selected time instances $t \in \{0.2\text{ s}, 0.4\text{ s}\}$ (Figure~\ref{fig:snapshotsAneurysm1}) and $t \in \{0.6\text{ s}, 0.8\text{ s}\}$ (Figure~\ref{fig:snapshotsAneurysm2}). These results confirm that the proposed fully-discrete model (see Section~\ref{sec:fully-discr}) stably captures the pulsatile inflation dynamics and maximum displacement concentrations within the aneurysmal region throughout the cardiac cycle. Moreover, the simulated maximum displacement magnitude of $1.5 \text{ mm}$ falls well within the wall displacement range ($1\text{--}3 \text{ mm}$) reported across the aneurysm region in previous studies \cite[Fig. 5]{gao21}.

\begin{figure}[!h]
    \centering
    \subfigure[$t= 0.6 \text{s}$]{%
        \begin{minipage}{0.425\textwidth}
            \centering
            \includegraphics[width=\textwidth,trim={9.cm 3.cm 12.cm 3.cm},clip]{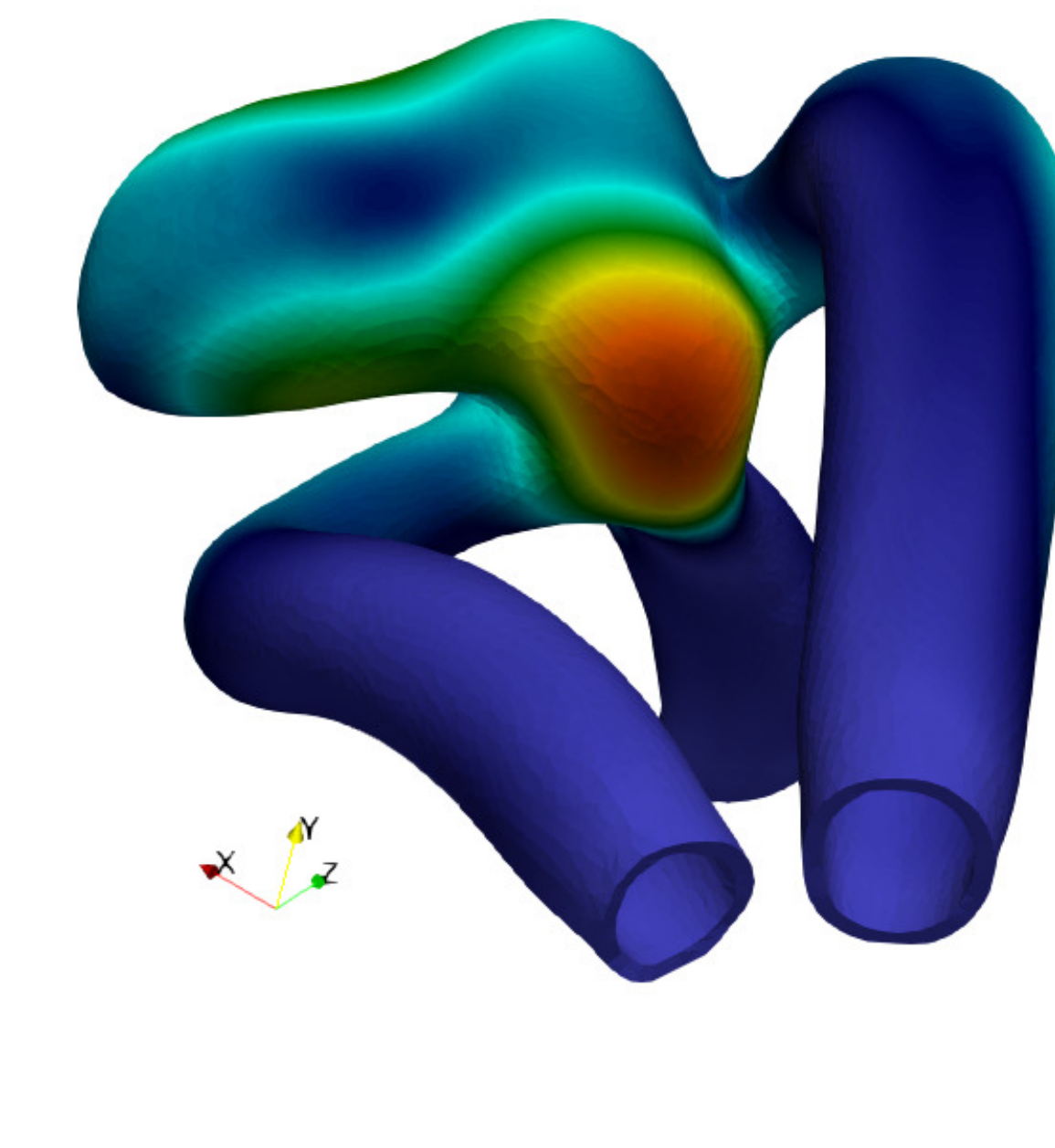}\\[0.5mm]
            \includegraphics[width=\textwidth,trim={9.cm 2.cm 11.cm 2.cm},clip]{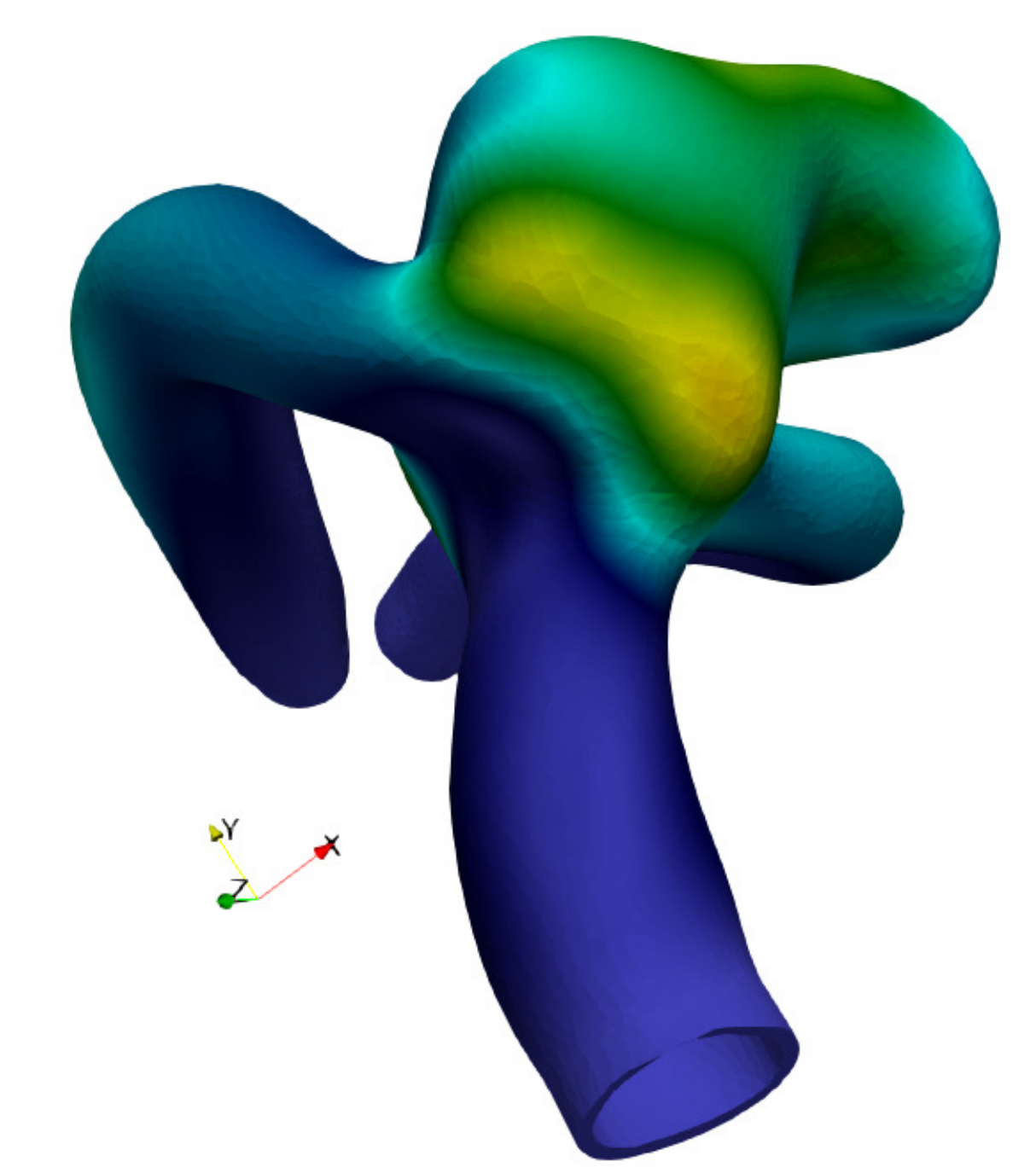}
        \end{minipage}%
    }\hfill
    \subfigure[$t= 0.8 \text{s}$]{%
        \begin{minipage}{0.5625\textwidth}
            \centering
            \includegraphics[width=\textwidth,trim={9.cm 3.cm 5.cm 3.cm},clip]{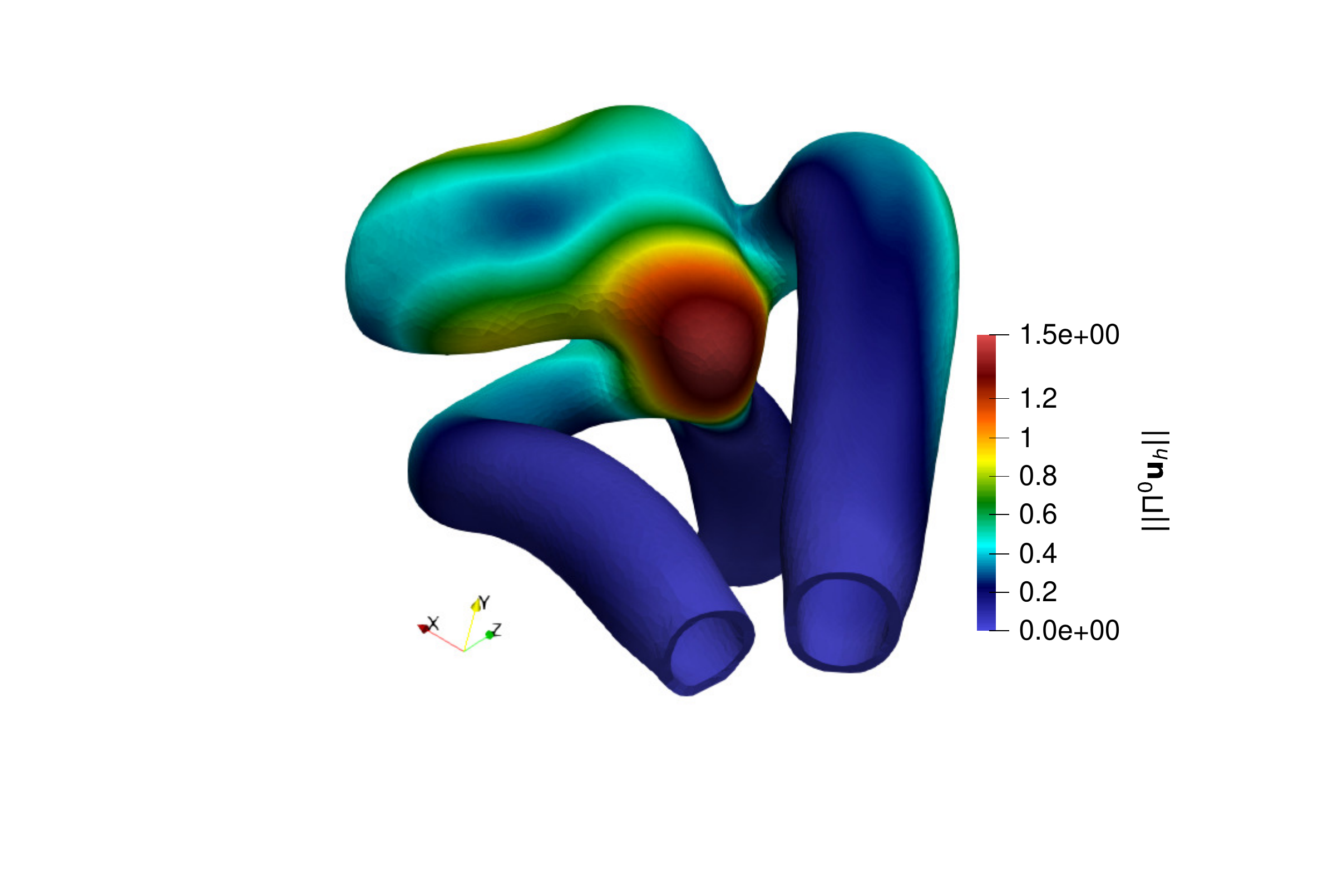}\\[0.5mm]
            \includegraphics[width=\textwidth,trim={9.cm 2.cm 5.cm 2.cm},clip]{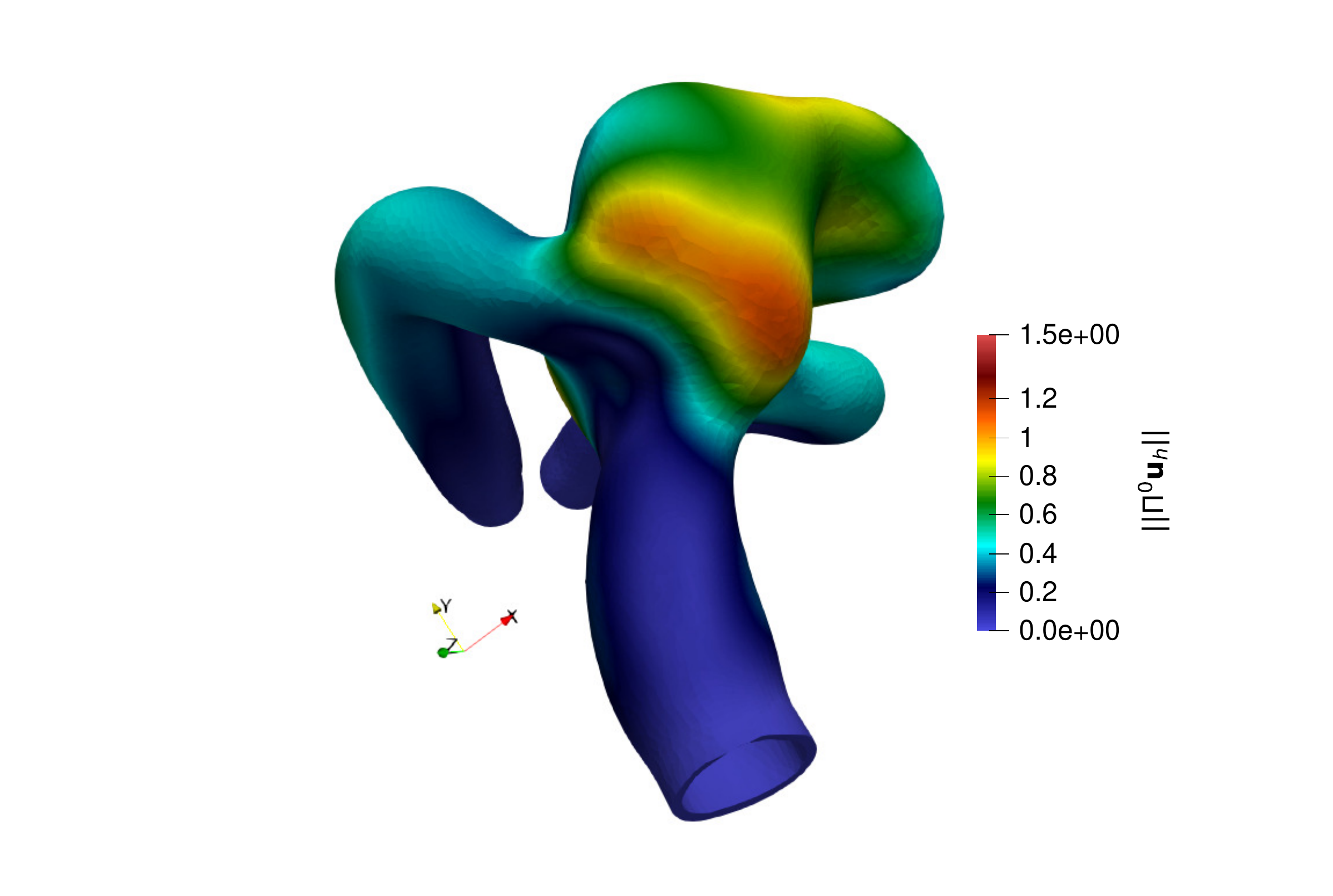}
        \end{minipage}%
    }\hfill
    \caption{Experiment 5. Snapshots of the polynomial projection of displacement $\Pi^0 \bu_h$ (magnified by $5$) taken at $t \in \{0.6\text{s},0.8\text{s}\}$ showing front (top row) and posterior (bottom row) views of the aneurysmal arterial tissue, obtained with the linear-order method ($k=1$).}
    \label{fig:snapshotsAneurysm2}
\end{figure}

\section*{Declarations} 

\noindent \textbf{Conflict of Interest.} The authors have no conflict of interest in this work. 

\small 
\bibliographystyle{siam}
\bibliography{fem}
\end{document}